\documentclass[reqno,10pt]{amsart}

\usepackage{showkeys}
\usepackage{amsmath,mathrsfs}
\usepackage{amssymb, amsmath,amsthm}
\usepackage{graphicx}
\usepackage{enumerate,color,xcolor}
\usepackage[colorlinks,
citecolor=green
]{hyperref}

\usepackage[normalem]{ulem}

\newtheorem{defi}{Definition}[section]
\newtheorem{thm}{Theorem}[section]
\newtheorem{lem}{Lemma}[section]
\newtheorem{rmk}{Remark}[section]
\newtheorem{cor}{Corollary}[section]
\newtheorem{prop}{Proposition}[section]

\newcommand{\vv}[1]{\boldsymbol{#1}}
\newcommand{\normm}[1]{{ \vert\kern-0.25ex \vert\kern-0.25ex \vert #1
		\vert\kern-0.25ex \vert\kern-0.25ex \vert}}

\makeatletter
\renewcommand\paragraph{\@startsection{paragraph}{4}{\z@}%
            {-1.5ex\@plus -1ex \@minus -.25ex}%
            {1.25ex \@plus .25ex}%
            {\normalfont\normalsize\bfseries}}
\makeatother
\numberwithin{equation}{section}

\newcommand{\beq}{\begin{equation}}
\newcommand{\eeq}{\end{equation}}
\newcommand{\ben}{\begin{eqnarray}}
\newcommand{\een}{\end{eqnarray}}
\newcommand{\beno}{\begin{eqnarray*}}
\newcommand{\eeno}{\end{eqnarray*}}
\let\f=\frac

\newcommand{\bit}{\begin{itemize}}
	\newcommand{\eit}{\end{itemize}}

\newcommand{\N}{{\mathbb N}}
\newcommand{\Z}{{\mathbb Z}}
\newcommand{\R}{{\mathbb R}}

\newcommand{\pa}{\partial}
\newcommand{\eps}{\varepsilon}
\newcommand{\ov}{\overline}

\newcommand{\mult}{\otimes}

\newcommand{\ba}{\begin{aligned}}
\newcommand{\ea}{\end{aligned}}
 \def\na{\nabla}

\let\pa=\partial

\let\wh=\widehat

\def\mP{\mathbf{P}}
\def\a{\mathfrak{a}}
\def\mF{\mathbf{F}}
\def\b{\mathfrak{b}}
\def\p{\mathfrak{p}}
\def\l{\mathfrak{l}}
\def\m{\mathfrak{m}}

\def\cF{{\mathcal F}}

\def\fM{{\mathfrak M}}

\def\cP{{\mathcal P}}

\def\cS{{\mathcal S}}

\def\virgp{\raise 2pt\hbox{,}}
\def\cdotpv{\raise 2pt\hbox{;}}

\def\C{\mathop{\mathbb C\kern 0pt}\nolimits}
\def\DD{\mathop{\mathbb D\kern 0pt}\nolimits}
\def\EE{\mathop{{\mathbb E \kern 0pt}}\nolimits}
\def\K{\mathop{\mathbb K\kern 0pt}\nolimits}
\def\N{\mathop{\mathbb N\kern 0pt}\nolimits}
\def\Q{\mathop{\mathbb Q\kern 0pt}\nolimits}
\def\R{\mathop{\mathbb R\kern 0pt}\nolimits}
\def\SS{\mathop{\mathbb S\kern 0pt}\nolimits}
\def\<{\langle}
\def\>{\rangle}

\def\F{{\mathfrak F }}
\def\U{{\mathfrak U }}
\def\tP{{\tilde{\mathcal{P}} }}
\def\tF{{\tilde{\mathfrak F} }}
\def\tS{\mathcal{S}}

\def\vf{\varphi}
\def\gs{\gtrsim}
\def\ls{\lesssim}
\def\S{\mathbb{S}}
\newcommand{\br}[1]{\langle #1 \rangle_v}
\newcommand{\rr}[1]{\left( #1 \right)}
\newcommand{\nr}[1]{\left| #1 \right|}
\newcommand{\LLO}[1]{\| #1 \|_{L^1}}
\newcommand{\LLT}[1]{\| #1 \|_{L^2}}
\newcommand{\LLF}[1]{\| #1 \|_{L^\infty}}
\newcommand{\HHL}[1]{\| #1 \|_{H^{-N}}}
\newcommand{\HHLL}[1]{\| #1 \|_{H^{-N}_{-N}}}
 \newcommand{\lr}[1]{\langle #1 \rangle}
\def\nn{{\mathrm n}}
\newcommand{\wei}{\langle v \rangle}

\def\sss{\mathsf s}

\def\vphi{\varphi}
\def\al{\alpha}
\def\tm{\tilde{m}}
\def\tphi{\tilde{\vphi}}
\def\fff{\mathbf f}
\def\ggg{\mathbf g}
\def\hhh{\mathbf h}
\begin{document}

\title[Landau equation in $H^{-\f12}$ space]{Existence, uniqueness and smoothing estimates for spatially homogeneous Landau-Coulomb equation in $H^{-\frac12}$ space with polynomial tail}

\author{Ling-Bing He}
\address[Ling-Bing He]{Department of Mathematical Sciences, Tsinghua University, Beijing, 100084, P.R. China.}
\email{hlb@tsinghua.edu.cn}

\author{Jie Ji}
\address[J. Ji]{School of Mathematics, Nanjing University  of Aeronautics and Astronautics\\
	Nanjing 211106,  P. R.  China.}
\email{jij\_24@nuaa.edu.cn}

\author{Yue Luo}
\address[Yue Luo]{Department of Mathematical Sciences, Tsinghua University, Beijing, 100084, P.R. China.}
\email{luo-y21@mails.tsinghua.edu.cn}

\begin{abstract} We demonstrate that the spatially homogeneous Landau-Coulomb equation exhibits  global existence and uniqueness around the space $H^{-\frac12}_3\cap L^1_{7}\cap L\log L$. Additionally, we furnish several quantitative assessments regarding the smoothing estimates in weighted Sobolev spaces. The new ingredients of the proof lie in the localized techniques in both phase space and frequency space.
\end{abstract}

\maketitle




\maketitle

\setcounter{tocdepth}{1}
\tableofcontents

\setcounter{tocdepth}{1}
\section{Introduction}
The Landau equation, initially formulated by Landau in 1936, serves as a valuable tool for describing the temporal evolution of a plasma system, characterized by charged particle collisions occurring under the influence of the Coulomb potential.  Mathematically, the  Cauchy problem of the Landau-Coulomb equation can be described as follows:
 \begin{equation}\label{1}
\partial_t f= Q(f,f)(v),
\end{equation}
complemented with initial data $f_0=f_0(v) \ge 0$. Here the Landau operator $Q$ is defined by
\begin{equation} \label{12d}
Q(g,h) =\nabla \cdot {\Big (}[ a*g ]\;\nabla h- [ a*\nabla g ] \; h{\Big )},
\end{equation}
with 
\begin{equation} \label{13d}
a(z)=|z|^{-1} \, \left(\mathrm{Id} -\frac{z\otimes z}{|z|^2} \right).
\end{equation}
We emphasize that $f := f(t,v) \ge 0$ stands for the distribution of particles that at time $t \in \R_+$ possess the velocity $v \in \R^3$.

\subsection{Fundamental properties of Landau equation} 

 To derive the basic properties of the equation, we first introduce  
the weak 
 formulation of the Landau operator $Q$ in the following way:  
\ben\label{Qweak1}
&& \int_{\R^3} Q(f,f) (v)\, \varphi(v) \, dv= - \frac12 \, \sum_{i=1}^3\sum_{j=1}^3 \iint_{\R^3 \times \R^3} a_{ij}(v-v_{*}) \\
&&\quad\times
 \left\{ \frac{\partial_j f}{f}(v) - \frac{\partial_j f}{f}(v_{*})  \right\} \left\{ \partial_i \varphi(v) - \partial_i \varphi(v_{*})  \right\}   
 f(v) f(v_{*}) \, dv \, dv_{*}, \nonumber
\een
where $\varphi$ is a suitable test function. Consequently,  
the operator indeed conserves (at the formal level) mass, momentum and energy, more precisely
\ben\label{cons}
\int_{\R^3} Q(f,f)(v) \, \varphi(v) \, dv = 0 \quad\text{for}\quad \varphi(v) = 1, v_i, \frac{|v|^2}2,\;i=1,2,3.
\een

We also deduce from formula (\ref{Qweak1}) the entropy structure of the operator (still at the formal level) by taking the test function $\varphi(v) = \log f(v)$, 
that is 
\ben\label{D(f)}
D(f)  := - \int_{\R^3} Q(f,f)(v)\, \log f(v) \, dv
\een
$$ =  \frac12 \sum_{i=1}^3\sum_{j=1}^3\iint_{\R^3 \times \R^3} a_{ij}(v-v_{*}) \left\{ \frac{\partial_i f}{f}(v) - \frac{\partial_i f}{f}(v_{*})  \right\} \left\{ \frac{\partial_j f}{f}(v) - \frac{\partial_j f}{f}(v_{*})  \right\}   \, f(v)\, f(v_{*}) \, dv \, dv_{*}.  $$

Note that $D(f) \ge 0$ since the matrix $a$ is (semi-definite) positive. One may check that if $D(f)=0$, then $f$ is a Maxwellian distribution, that is $f = \mu_{\rho,u,T}$, with
\ben\label{MaxGen}
\mu_{\rho,u,T}(v) = \frac{\rho}{(2\pi T)^{3/2}} \, e^{-\frac{|v-u|^2}{2T}},
\een
where $\rho\ge 0$ is the density, $u \in \R^3$ is the mean velocity and $T>0$ is the temperature of the plasma. 
They are defined by
\begin{equation} \label{cco}
\rho = \int_{\R^3} f(v) \, dv, \quad 
u = \frac{1}{\rho} \int_{\R^3} v \,f(v) \, dv, \quad
T = \frac{1}{3 \rho} \int_{\R^3} |v-u|^2\, f(v) \, dv.
\end{equation}
Thanks to \eqref{cons}, we have  (when $f: = f(t,v)$ is a solution of equations (\ref{1}-\ref{13d})),  
\ben\label{conssh}
\forall t \ge 0, \qquad \rho(t)  = \rho (0), \quad u(t)=u(0),
\quad T(t)=T(0),
\een
which implies that the parameters $\rho,u, T$ are constants. 
\medskip

Denoting  the relative entropy $H(t)$ by  
 \ben\label{DefHt}
H(t):=H(f|\mu_{\rho,u,T})(t) := \int_{\R^3} \bigg(f(t,v) \log \bigg( \f{f(t,v)}{\mu_{\rho,u,T}} \bigg) -f(t,v) +\mu_{\rho,u,T}\bigg) \, dv,\een 
then the famous $H$-Theorem(still at the formal level) holds, i.e.,
\ben\label{entropyeq}
\frac{d}{dt}H(t) = -D(f(t,\cdot)) \le 0.
\een

Finally, we remark that if we introduce the quantity
\ben\label{Defbc}
b_i(z): = \sum_{j=1}^3 \partial_j a_{ij}(z) = -2 \,z_i\, |z|^{-3}, \quad 
\een
 the Landau operator with Coulomb potential can also be written as
\ben\label{Qbis}
Q(f,f) 
= \sum_{i=1}^3 \partial_i \bigg( \sum_{j=1}^3(a_{ij}*f)\, \partial_j f - (b_i *f)\, f  \bigg)   
= \sum_{i=1}^3\sum_{j=1}^3(a_{ij}*f) \,\partial_{ij}f + 8\pi \,f^2,
\een
where we used the identity $\sum_{i=1}^3\pa_i b_i(z)= - 8\pi\delta_0(z)$. In this sense, the Landau equation is comparable to the semi-linear heat equation
\ben\label{SLHeq} \pa_tf-\Delta_v f=8\pi f^2.
\een

\subsection{Previous work} Before going further, let us provide   a brief review on the previous work,  specifically on the global well-posedness and the regularity problem.
 \smallskip

 \noindent$\bullet$ \underline{\bf{Uniqueness result.}} The conditional uniqueness is due to Fournier in \cite{NF}. In fact, he proved  the uniqueness of solutions for equations (\ref{1}-\ref{13d}) in the class $L^\infty_{loc}([0,\infty); L^1_2(\R^3)) \cap L^1_{loc}([0,\infty);L^\infty(\R^3))$. This result implies a local well-posedness result, assuming further that the initial data lie in $L^p (\R^3)$ with $p>3/2$. For more details, we refer readers to \cite{GL,GGL}.   It is worth mentioning that the authors of \cite{GSUN} recently applied the method of the \(\mathcal{M}\)-operator to obtain uniqueness of solutions in the critical space \(L^{3/2}\), where \(L^\infty\) is precisely non-integrable, so that Fournier's result does not apply.

 \noindent$\bullet$ \underline{\bf{$H$-solution and weak solution.}} Villani \cite{Vi} established the global existence of the so-called $H$-solutions for equations (\ref{1}-\ref{13d}) based on the entropy dissipation $D(f)$. The $H$-solution belongs to $L^1_{loc}([0,\infty); L^3_{-3}(\R^3))$ if it is well-approximated, thanks to the functional estimates in \cite{D, DX, CDH} for variants of inequality \eqref{D(f)}. This property is sufficient to show that the $H$-solution is indeed a weak solution in the usual sense. In \cite{SJi} the space $ L^3_{-3}$ was improved to $L^3_{-5/3}$. 
 Regarding the longtime behavior of the solution, we refer to \cite{CDH,DHJ} for the convergence to the equilibrium with quantitative estimates. See also \cite{GGIV} for the measure of the singular set via Hausdorff measure thanks to the De Giorgi's method applied to a scaled suitable weak solution.

 \noindent$\bullet$ \underline{\bf{Fisher information and global well-posedness.}} The breakthrough of the global well-posedness is due to Guillen and   Silvestre in \cite{GS}. They proved that the  Fisher information is monotone decreasing in time and thus the classical solutions to the equation will never blow up. See the recent development in \cite{GGL2,SJi}.

\subsection{Goals and difficulties}
Our primary focus lies on well-posedness for the spatially homogeneous Landau equation with a Coulomb potential. Specifically, we examine uniqueness and smoothing estimates where the initial data belong to negative Sobolev spaces with polynomial tail. 

From \cite{GS,SJi}, we have the following pictures: (1)  The weak solution can generate the Fisher information at any positive time; (2) Once it produces, the  Fisher information   will decrease in time. Thus the weak solutions are the classical solution for any positive time. This motives us to consider the following two questions:

${\bf (Q1)}$. Our initial query aims to identify the most fundamental function space required for the well-posedness of the Landau-Coulomb equation. The uniqueness inherent in this space ensures that the local well-posedness can be extrapolated to achieve global well-posedness.

${\bf (Q2)}$. Our second inquiry is to determine, using \eqref{ToyM} as a model, the growth of Sobolev regularity with respect to the velocity variable at infinity and the singularity with respect to the time variable at the origin.

\smallskip

Next, we will discuss the challenges associated with these two questions.

\noindent$\bullet$ Previous work shows that the equation is well-posed in weighted $L^p$ spaces with $p\geq3/2$. For the semi-linear heat equation 
\eqref{SLHeq}, we notice that $L^{3/2}$ is the critical space since the equation \eqref{SLHeq} is non-unique in $L^p$ with $p<3/2$(see \cite{HW}). It is natural to ask whether this fact holds or not for the Landau-Coulomb equation. This motivates us to consider the equation in $H^{-1/2}$ space. By the smoothing estimates for the heat equation, i.e.,
$$ \|e^{t\triangle}f_0\|_{L^\infty}\ls t^{-1}\|f_0\|_{H^{-1/2}},$$ we have no idea whether  Fournier's conditional uniqueness result can be applied or not in this situation. Thus we need a new framework to solve the problem.

\smallskip

\noindent$\bullet$ The smoothing effect for the homogeneous Landau equation has been extensively addressed and it can be proven using methods such as the De Giorgi technique or the standard energy method within the $L^2$ framework(refer to \cite{GL,GIF,CHDH} for more details). The former utilizes the De Giorgi method and Schauder estimates within the framework of parabolic equations to derive \( L^\infty \) and \( C^\alpha \) estimates, which are further refined to \( C^\infty \) estimates through a bootstrap argument. While this paper, adopting the latter's approach, focuses on smoothness in Sobolev spaces and aims to characterize the relationship between Sobolev regularity and the growth of weights.
In the case of the toy model, i.e.,
 \ben\label{ToyM} \pa_tf+(-\Delta_v)(\lr{v}^{-3}f)=0, \een 
  it has been proven in \cite{HJ2} that the solution with so-called {\it typical rough and slowly decaying data} can grow with polynomial rate in the $v$ variable for any positive time.  To illustrate this, following the informal computation in \cite{HJ2}, if the initial data $f_0$ is rough and decays very slowly in the high-velocity regime (i.e., $f\in L^2_\ell$ with $\ell<\infty$), the following scenario can occur:
\beno
\pa^\alpha f&\sim& \sum_{k\ge-1} (2^{-3k}t)^{-\frac{3+|\alpha|}{2}}(\pa^\alpha G)\big((2^{-3 k}t)^{-\frac{1}{2}}\cdot\big)(\cP_kf_0)\\
&=&\sum_{k\ge-1} t^{-\frac{|\alpha|}{2}}2^{(3\frac{|\alpha|}{2}-\ell)k}\big[ t^{-\frac{3}{2}}2^{\frac{3}{2}3 k}(\pa^\alpha G)\big((2^{-3 k}t)^{-\frac{1}{2}}\cdot\big)\big](2^{k\ell}\cP_kf_0),
\eeno
where $G$ is a Gaussian function and $\cP_kf(v)\sim f(v)\mathrm{1}_{|v|\sim 2^k}$. This implies that
\ben\label{PointwToyM}|(\pa^\alpha f)(t,v)|\sim C(t,f_0)\lr{v}^{3\frac{|\alpha|}{2}-\ell}.\een
Thus, if $|\alpha|\ge 2\ell/3$, then $\pa^\alpha f$ will grow with a polynomial rate of $3|\alpha|/2-\ell$ in the high-velocity regime, suggesting that enhancing the regularity of the solution necessitates the use of the Sobolev spaces with negative weights.

\subsection{Notations, main results and strategies}    We  list  notations and function spaces in the below.
\subsubsection{Notations}
$(i)$  We use the notations $a\ls b(a\gs b)$ and $a\ls_c b(a\gs_c b)$ to indicate that there is a constant $C$ which is uniform or depends on parameter $c$ and may be different on different lines, such that $a\leq Cb(a\geq Cb)$. We use the notation $a\sim b$($a\sim_c b$) whenever $a\ls b$ and $b\ls a$($a\ls_c b$ and $a\gs_cb$).

$(ii)$ We denote $C_{a_1,a_2,\cdots,a_n}$(or $C(a_1,a_2,\cdots,a_n)$) by a constant depending on parameters $a_1,a_2,\cdots,a_n$. Moreover,  parameter $\varepsilon$ is used  to represent different positive numbers much less than 1 and determined in different cases.

$(iii)$ We write $a\pm$ to indicate $a\pm\varepsilon$, where $\varepsilon>0$ is sufficiently small.  We set  $a^+:=\max\{0,a\}$, $a^-:=-\min\{a,0\}$ and use $[a]$ to denote the maximum integer which does not exceed $a$.

$(iv)$ $\mathbf{1}_\Omega$ is the characteristic function of the set $\Omega$. We use $\br{f,g}$ to denote the inner product of $f, g$ in space $L^2(\R^3_v)$.

$(v)$ We denote the Fourier transform with respect to the variable \( v \) by \( \mathcal{F} \), and the inverse Fourier transform by \( \mathcal{F}^{-1} \).

$(vi)$ Suppose $A$ and $B$ are two operators, then the commutator $[A,B]$ between $A$ and $B$ is defined by $[A,B]=AB-BA$. 


\subsubsection{Function spaces}
We  give several definitions to spaces involving different variables.
\smallskip

 $(1)$ \textit{Function spaces in $v$ variable.} Let $f=f(v)$ and $\<v\>:=(1+|v|^2)^{1/2}$. For $m,l\in\R $, we define the weighted Sobolev space $H_l^m$ and $H_l^{m,\sss}$ as follows:
$H^m_l:=\Big\{f(v)|\|f\|_{ H^m_l}=\|\<D\>^m\<\cdot\>^lf\|_{L^2}<+\infty\Big\}$ and
\ben\label{hms}H^{m,\sss}_l:=\Big\{f(v)|\|f\|_{ H^{m,\sss}_l}=\|\<D\>^m\log^\sss(2+\<D\>)\<\cdot\>^lf\|_{L^2}<+\infty\Big\}.\een
Here $a(D)$ is a pseudo-differential operator with the symbol $a(\xi)$, i.e., \beno
(a(D)f)(v):= \frac{1}{(2\pi)^3}\int_{\R^3}\int_{\R^3}e^{i(v-u)\xi}a(\xi)f(u)dud\xi.
\eeno

\noindent The $L\log L$ space is defined by the identities $\|f\|_{L\log L}:=\int_{\R^3} f\log f dv$ for $f\geq0$.


 \smallskip

  $(2)$ \textit{Function spaces in $t, v$ variables.} Let $f=f(t, v)$ and $X$ be a function space in $v$ variables. Then $L^p([0,T],X)$ and $ L^\infty([0,T],X)$ are defined by
 \[ L^p([0,T],X):=\bigg\{f(t,x,v)\big|\|f\|^p_{ L^p([0,T],X)}=\int_0^T\|f(t)\|^p_{X}dt<+\infty\bigg\},\quad 1\leq p<\infty,\]
  \[L^\infty([0,T],X):=\Big\{f(t,x,v)|\|f\|_{ L^\infty([0,T],X)}=\mathrm{esssup}_{t\in [0,T]}\|f(t)\|_{X}<+\infty\Big\}. \]
  We also define space $L^p_{loc}([0,\infty),X):=\cap_{T<\infty}L^p([0,T],X)$.

 \subsubsection{Main results} Our main result is to prove the existence and uniqueness results of the equation \eqref{1} around  $H^{-\f12}$ space. We have
 \begin{thm}[Existence and uniqueness in  $H^{-\f12}$]\label{uniquethm}
	Let $f_0\in L^1_{2l+1}\cap L\log L$ be a nonnegative function.
\begin{itemize}
	\item[(1).] If $\<v\>^lf_0\in H^{-\f12,\sss}$ with $\sss\in[0,1]$ and $l\geq3$,  then there exists $T>0$ and a local solution $f=f(t,v)$ of  the equation \eqref{1} satisfying that $f\in C([0,T];H^{-\f12,\sss}_l)\cap L^2([0,T];H^{\f12,\sss}_{l-\f32})$.
    \item[(2).] The solution $f$ constructed in (1) is indeed global in $ C([0,\infty);H^{-\f12,\sss}_l)\cap L^2_{loc}([0,\infty);H^{\f12,\sss}_{l-\f32})$.

	\item[(3).]  If $\<v\>^lf_0\in H^{-\f12,\sss}$ with  $\sss>\f12$ and $l\geq5$, then the above global solution $f$  is unique.
\end{itemize}
 \end{thm}

Several remarks are in order:

\begin{itemize}
	\item[(i).]  According to the definition \eqref{hms}, $H^{-\f12,\sss}$ will coincide with $H^{-\f12}$ when $\sss=0$. This means we construct a global solution in $H^{-\f12}$ space. Moreover,  the estimate 
	$\|e^{t\triangle}f_0\|_{L^\infty}\ls t^{-1}|\log t|^{-\sss}\|f_0\|_{H^{-\f12,\sss}}$ indicates that the property, $f\in L^1([0,T];L^\infty)$, may fail when $\sss\in [0,1]$. As a result, the Fournier's conditional uniqueness result in \cite{NF} cannot be applied here. 
	\item[(ii).] We provide a new approach for the uniqueness result in $H^{-\f12,\sss}$ with  $\sss>\f12$, which basically relies on the localization techniques both in phase and frequency spaces.
	\item[(iii).] It is widely recognized that if $f$ is a solution to \eqref{1}, then for any $\alpha,\lambda>0$, the function
	  $f_{\al,\lambda}:=\lambda^\alpha f(\lambda^\alpha t, \lambda v)$  is also a solution. If we set $\lambda^\al=t^{-1}$, then the function $t^{-1}f(1,t^{-\f1\al}v)$ dose not satisfy the conservation laws of mass, momentum, and energy (refer to  \eqref{cons}). Consequently, self-similar solutions are precluded. This implies that the non-uniqueness argument developed for the semi-linear heat equation \eqref{SLHeq} with initial data in $L^p$ space($p<3/2$)  is not directly applicable to equation \eqref{1}.
	\item[(iv).] The reason for the initial value \( f_0 \in L^1_{2l+1} \) is that we can utilize Fisher information and the propagation of \( L^1 \) moments(see \eqref{0000}) to ensure that the solution remains in \( H^{-\frac{1}{2}, \sss}_l \) after a period of time, thus maintaining the global solution  within this space.

\end{itemize}


The method developed for proving uniqueness also applies to regularity estimates, which are required for global existence. We therefore give the Sobolev regularity estimates as well. The following theorem establishes smoothing estimates for initial data with polynomial tails.

\begin{thm}[Smoothing estimates with negative weights]\label{T1}
 Let  $f\in C([0,\infty);H^{\mathsf{r}}_l)\cap L^2_{loc}([0,\infty);H^{\mathsf{r}+1}_{l-\f32})$ 
		be a global solution of the equation \eqref{1} with $\mathsf{r}\in[-\f12,0]$ and $l\geq \f{9-6\mathsf{r}}{4}$, satisfying $\inf_{t\in[0,\infty)}\|f(t)\|_{L^1}>0$ and $\sup_{t\in[0,\infty)}\|f(t)\|_{L^1_2\cap L\log L}<\infty$. Then 
		for any $\nn\geq\mathsf{r}$, it holds that
		\ben\label{SMesti}\|f(t)\|_{H^\nn_{l-\f{3}{2}\nn+\f{3\mathsf{r}}{2}}(\R^3)}\ls t^{-\f{\nn}{2}+\f{\mathsf{r}}{2}}\mathbf{1}_{t\leq1}+C(t)\mathbf{1}_{t>1},\quad \forall t\in(0,\infty),\een
		where \( C(t) \) is a locally uniformly bounded function.

\end{thm}

We remark that Theorem \ref{uniquethm} confirms the existence of such a global solution \( f \). It is comparable to the smoothing estimates for the linear heat equation with $H^{\mathsf{r}}$ initial data:
\[\|e^{-t\Delta}f_0\|_{H^\nn} \ls t^{-\frac{\nn}{2}+\f{\mathsf{r}}{2}}\|f_0\|_{H^{\mathsf{r}}}.\]
Thus, \eqref{SMesti} is sharp. Furthermore, by utilizing the interpolation inequality, if $\mathsf{r}\in [-\f12,0]$,   we can infer that the solution $f$ satisfies: 
$$\|f(t)\|_{L^\infty}\ls t^{-\f{3-2\mathsf{r}}{4}}+1.$$  Thus the uniqueness of solutions with initial data in \(H^{\mathsf{r}}\) for \(\mathsf{r}>-1/2\) is therefore a direct corollary of Fournier's result in \cite{NF}.

\subsubsection{Main strategy} The challenge in  achieving uniqueness and precise smoothing estimates stems from the degenerate and non-local characteristics of the collision operator  $Q$ (see \eqref{12d}).  To surmount these challenges, we employ the localization techniques introduced  in \cite{HE,HJ,HJ2}.  In essence, these techniques allow us to establish the sought-after results by utilizing only the weak form of the equation  \eqref{Qweak1}.  As a result, our smoothing estimates are independent of the uniqueness proof and impose fewer constraints on the required regularity of the solution itself.
\smallskip

To achieve the regularity estimates, our main strategies can be summarized as follows:
\smallskip

$(i).$ \underline{Localization in   frequency space.}  Unlike directly differentiating the equation, we utilize \eqref{Qweak1} with $\varphi:=\F_j f$, where $\F_j$ is a multiplier defined in \eqref{localOPxij}. Using the Bernstein inequality (see Lemma \ref{7.8}) in $L^2$ space, the $\alpha$-th derivative of $\F_jf$ can be equivalently expressed as multiplying a constant $2^{|\alpha|j}$. When returning to the original equation \eqref{1}, we need to localize the equation in frequency space. Therefore, we have to carefully estimate the commutator between the localized operator $\F_j$ and the collision operator $Q$.
\smallskip

$(ii).$ \underline{Localization in  phase space.}  To handle the potential growth in the $v$ variable, we localize the equation in phase space using dyadic decomposition as shown in \eqref{7.2}. The crucial observation is that both the second function $h$ and the third function $f$ in the inner product $\lr{Q(g,h),f}_v$ can be localized in the same phase space regime. This is distinct from the Boltzmann collision operator. It allows us to multiply negative weights and distribute them equally to the functions $h$ and $f$. However, implementing this approach requires dealing with complicated commutators.
\smallskip





\smallskip

When addressing the question of uniqueness, we adopt the $L^2$-based energy method, which proves effective for the toy model \eqref{SLHeq} when dealing with small initial data in $H^{-1/2}$. In this scenario, it is pivotal to recognize that the dissipation still dominates the non-linearity.  Thus, we circumvent the need for the condition that $f\in L^1_{loc}([0,\infty);L^\infty(\R^3))$. In order to transpose this idea to the Landau equation,  our primary innovation  lies in the following two aspects:
\begin{itemize}
	\item  The framework employed for the smoothing estimates remains effective for establishing uniqueness. The key distinction is that the regularity of the function space we utilize is slightly higher than $H^{-\f12}$, i.e. $H^{-\f12,\sss},\sss>\f12$, necessitated by the non-local terms $a*f$ and $b*f$.
	\item  We decompose the initial data into two components: large data in a more regular space, where we already have established uniqueness and obtained better estimates for the solution, and small data in the original space, which ensures that we can leverage the crucial property that dissipation dominates the non-linearity.
\end{itemize}
We refer readers to   Section \ref{section5} for more details.

\subsection{Organization of the paper}
The rest of this paper is arranged as follows. In Section \ref{section2}, we provide the coercivity estimates, upper bound estimates of the collision operator, and its commutator estimates with the localization operators. In Section \ref{section3}, we present the local well-posedness for initial data in the space \( H^{-\frac{1}{2}, \sss} \) with \( \sss > 0 \), as well as regularity estimates for \( H^{\mathsf{r}} \) with \( \mathsf{r} \in (-\frac{1}{2}, 0] \). While the local well-posedness and regularity estimates for initial data in \( H^{-\frac{1}{2}} \) will be given in Section \ref{section4}. Section \ref{section5} is dedicated to proving the uniqueness in \( H^{-\frac{1}{2}, \sss} \) with \( \sss > \frac{1}{2} \) and the global existence of solutions. Some auxiliary lemmas will be listed in the appendix. 

\smallskip

\section{Analysis of the collision operator}\label{section2}
In this section, we establish both the upper and lower bounds of the collision operator, as well as provide estimates for certain commutators. These results may be of independent interest.

\subsection{Coercivity estimate} We begin with the following proposition:

\begin{prop}\label{coer}
	Suppose that the non-negative function $g$ satisfies $\LLO{g}>\delta, \|g\|_{L^1_2}+\|g\|_{L\log L}<\lambda$. Then
	\ben
	\|\na f\|_{L^2_{-3/2}}^2+\|(-\Delta_{\S^2})^{\f{1}{2}}f\|_{L^2_{-3/2}}^2\leq C(\delta, \lambda)\int_{\R^3} (a*g):\na f\mult\na fdv.
	\een
	Here $-\Delta_{\S^2}=\sum\limits_{1\leq i<j\leq 3}(v_i\pa_j-v_j\pa_i)^2$ is the Laplace-Beltrami operator. For any $a\in\R$ and radial function $\phi$, we have(see Lemma 5.8 in  \cite{HE})
	\ben\label{LaplaceBeltrami}
	\phi(|D|)(-\Delta_{\S^2})^{a}=(-\Delta_{\S^2})^{a}\phi(|D|),\quad \mbox{and}\quad\phi(|v|)(-\Delta_{\S^2})^{a}=(-\Delta_{\S^2})^{a}\phi(|v|).
	\een
	Thus it holds that $\|(-\Delta_{\S^2})^{\f12}f\|^2_{L^2_{-\f32}}=\|v\times\na_v f\|^2_{L^2_{-\f32}}$, which will be used in the proof below.
\end{prop}

\begin{proof} We prove it in the spirit of \cite{DV1}. We claim that there exists a constant 
	  $ C(\delta, \lambda)>0$ such that for all  $v\in\R^3, \xi\in\S^2, a_{ij}*g(v)\xi_i\xi_j\geq C(\delta, \lambda)\wei^{-3}$.  
	For fixed $\xi\in\S^2, v\in\R^3$, let $D_{\theta, \xi}(v):=\{v_*\in\R^3:\big|\f{v-v_*}{|v-v_*|}\cdot\xi\big|>\cos\theta\}$, then we have  
	\[
	\ba
	a_{ij}*g(v)\xi_i\xi_j&\geq\int_{D_{\theta, \xi}(v)^c}|v-v_*|^{-1}\Pi(v-v_*):\xi\mult\xi\ g_*dv_*\\
	&\geq \sin^2\theta\int_{D_{\theta, \xi}(v)^c}|v-v_*|^{-1}g_*dv_*
	\gs\wei^{-1}\sin^2\theta\int_{D_{\theta, \xi}(v)^c}\<v_*\>^{-1}g_*dv_*,
	\ea\] while 
	\[\ba
	&\mathcal{I}:=\sin^2\theta\int_{D_{\theta, \xi}(v)^c}\<v_*\>^{-1}g_*dv_* \geq\f{\sin^2\theta}{\<R\>}\int_{D_{\theta, \xi}(v)^c\cap\{|v_*|\leq R\}}g_*dv_*=\f{\sin^2\theta}{\<R\>}\bigg(\LLO{g}-\int_{|v_*|\geq R}g_*dv_*  \\&-\int_{D_{\theta, \xi}(v)\cap\{|v_*|\leq R\}}g_*dv_*\bigg)\geq\f{\sin^2\theta}{\<R\>}\rr{\delta-\f{\lambda}{R^2}-M\left|D_{\theta, \xi}(v)\cap\{|v_*|\leq R\}\right|-\f{\lambda}{\log M}}.
	\ea\]
	Since $\left|D_{\theta, \xi}(v)\cap\{|v_*|\leq R\}\right|\leq 2\pi R(|v|+R)^2\tan^2\theta$, if we set $R^2:=4\lambda/\delta, M=\exp\{4\lambda/\delta\}$, then
	\[\mathcal{I}\geq C_1(\delta, \lambda)\sin^2\theta\rr{\delta-C_2(\delta,\lambda)(|v|+R)^2\tan^2\theta}.\]
	When $|v|\geq R$, we set $\tan^2\theta:=\f{\delta}{8C_2(\delta,\lambda)|v|^2}$, which implies that $\mathcal{I}\geq C(\delta, \lambda)|v|^{-2}$. For $|v|\leq R$,  we set $\tan^2\theta:=\f{\delta}{8C_2(\delta, \lambda)R^2}$, which implies that  $\mathcal{I}\geq C(\delta, \lambda)$. This ends the proof of the claim.
	
	Thanks to the above claim, we easily get that
	\ben\label{coer1}
	\|\na f\|_{L^2_{-3/2}}^2\leq_{\delta, \lambda}\int_{\R^3} (a*g):\na f\mult\na fdv.
	\een
	  Next we want to show that 
	 $\|(-\Delta_{\S^2})^{\f{1}{2}}f\|_{L^2_{-3/2}}^2\leq_{\delta, \lambda}\int (a*g):\na f\mult\na fdv$. 
	By direct computation, one has
	\[\ba
	&\int_{\R^3}(a*g):\na f\mult\na fdv=\iint_{\R^6} \f{|(v-v_*)\times \na f|^2}{|v-v_*|^{3}}g_*dv_*dv\geq \iint_{\substack{|v_*|\leq \f{|v|}{2}\\ |v|\geq R}} \f{|(v-v_*)\times \na f|^2}{|v-v_*|^3}g_*dv_*dv \\
	&\gs \iint_{\substack{|v_*|\leq \f{|v|}{2}\\ |v|\geq R}}|v|^{-3}|v\times\na f|^2g_*dv_*dv-\iint_{\substack{|v_*|\leq \f{|v|}{2}\\ |v|\geq R}}|v|^{-3}|v_*\times\na f|^2g_*dv_*dv.
    \ea\]
	For the first term,  since
	\[\int_{\R^3} g_*\vv{1}_{\substack{|v_*|\leq \f{|v|}{2}\\ |v|\geq R}}dv_*\geq \int_{\R^3} g_*\vv{1}_{|v_*|\leq\f{R}{2}}dv_*\geq \int_{\R^3} g_*dv_*-\f{4}{R^2}\int_{\R^3} g_*|v_*|^2dv_*\geq \delta-\f{4\lambda}{R^2},\]
	by choosing $R$ sufficiently large, we have
	\[\iint_{\substack{|v_*|\leq \f{|v|}{2}\\ |v|\geq R}}|v|^{-3}|v\times\na f|^2g_*dv_*dv\gs C(\delta, \lambda)\int_{|v|\geq R} |v|^{-3}|v\times \na f|^2dv.\]
	Thanks to the facts that 
	 $\int_{|v|\leq R} \wei^{-3}|v\times \na f|^2dv\ls_R \|\na f\|_{L^2_{-3/2}}^2$ 
	and
	 $\iint_{\substack{|v_*|\leq \f{|v|}{2}\\ |v|\geq R}}|v|^{-3}|v_*\times\na f|^2g_*dv_*dv\ls  \|g\|_{L^1_2}\|\na f\|_{L^2_{-3/2}}^2$, 
	we are led to that
	\[\|(-\Delta_{\S^2})^{\f{1}{2}}f\|_{L^2_{-3/2}}^2\ls_{\delta, \lambda}\|\na f\|_{L^2_{-3/2}}^2+\int_{\R^3} (a*g):\na f\mult\na fdv\leq_{\delta, \lambda}\int_{\R^3} (a*g):\na f\mult\na fdv.\] This, together with \eqref{coer1} completes the proof of this proposition.
\end{proof}

\subsection{Dyadic decomposition and its application to   the collision operator}\label{DDS}
Let $B_{\frac{4}{3}}:=\{\xi\in\R^3||\xi|\leq\frac{4}{3}\}$ and $C:=\{\xi\in\R^3||\frac{3}{4}\leq|\xi|\leq\frac{8}{3}\}$. Then one may introduce two radial functions $\psi\in C_0^\infty(B_{\frac{4}{3}})$ and $\vphi\in C_0^\infty(C)$ which satisfy
\ben
\label{7.1}\psi,\vphi\geq0,~~&and&~~\psi(\xi)+\sum_{j\geq0}\vphi(2^{-j}\xi)=1,~\xi\in\R^3,\\
\notag|j-k|\geq2&\Rightarrow& \rm{Supp}~\vphi(2^{-j}\cdot)\cap \rm{Supp}~\vphi(2^{-k}\cdot)=\emptyset,\\
\notag j\geq1&\Rightarrow& \rm{Supp}~\psi\cap \rm{Supp}~\vphi(2^{-j}\cdot)=\emptyset.
\een
Moreover, let us  assume that there exists a constant $c>0$ such that $\vf(x)\equiv 1$ for $\f{3}{2}-c\leq |x|\leq \f{3}{2}+c$.

\subsubsection{Dyadic decomposition in the phase space} The dyadic operator in the phase space $\cP_j$ can be defined as follows:
\ben
\label{7.2}\cP_{-1}f(x):=\psi(x)f(x),~\cP_jf(x):=\vphi(2^{-j}x)f(x),~j\geq0.
\een
By the definition, there exists a integer $N_0\geq2$ such that  $\cP_j\cP_k=0$ if $|j-k|\geq N_0$. In applications, we also introduce 
$\tP_lf(x):=\sum\limits_{|k-l|\leq N_0}\cP_kf(x)$ and $\U_jf(x):=\sum\limits_{k\leq j}\cP_kf(x)$.
  For any smooth function $f$, we have  \[f=\cP_{-1}f+\sum\limits_{j\geq0}\cP_jf.\]

\subsubsection{Dyadic decomposition in the frequency space}  We denote $\tm:=\cF^{-1}\psi$ and $\tphi:=\mathcal{F}^{-1}\vphi$ where they are the inverse Fourier transform of $\vphi$ and $\psi$. If we set $\tphi_j(x)=2^{3j}\tphi(2^jx)$, then the dyadic operator in the frequency space $\F_j$ can be defined as follows
\ben\label{localOPxij}
	\F_{-1}f(x):=\int_{\R^3}\tm(x-y)f(y)dy,~\F_jf(x):=\int_{\R^3}\tphi_j(x-y)f(y)dy,~j\geq0.
\een 
We also introduce $\tF_jf(x)=\sum_{|k-j|\leq 3N_0}\F_kf(x)$ and $\tS_jf(x)=\sum_{k\leq j}\F_kf$. Then for any $f\in \mathscr{S}'$, it holds that
 \[f=\F_{-1}f+\sum_{j\geq0}\F_jf.\]
 
We remark that the basic properties of the above dyadic operators, such as Bernstein's inequality, the characterization of weighted Sobolev spaces, and the estimates for their mutual commutators, are all given in the Appendix.

\subsubsection{Applications to the collision operator}  Recall that $a(z)$ defined in \eqref{13d}. By introducing
\ben\label{2.4} a_k(z)=\left\{\ba
	&a(z)\vf(2^{-k}z),\quad&\mbox{if}\quad k\geq 0;\\
	&a(z)\psi(z),\quad&\mbox{if}\quad k=-1,
\ea\right.
\een
we set $b_k(z):=\na\cdot a_k(z), \quad c_k(z):=\na\cdot b_k(z),\quad Q_k(g,h):=\na\cdot([a_k*g]\na h-[a_k*\na g]h)$. Then we have
\ben
\br{Q_k(g,h),f}&=&\int_{\R^3} (a_k*g):\na^2 hfdv-\int_{\R^3}(c_k*g)hfdv\notag\\
&=&\int_{\R^3}(a_k*g):h\na^2 fdv+2\int_{\R^3}(b_k*g)\cdot h\na f dv.
\een
On one hand, in phase space, applying the   dyadic decomposition to the operator, we deduce that
\ben\label{phasedecom}
\br{Q(g,h),f}=\sum_{k,\l\geq -1}\br{Q_k(\cP_{\l}g,h),f}=\sum_{k\geq N_0-1}\br{Q_k(\U_{k-N_0}g,\tP_{k}h),\tP_{k}f}\notag\\+\sum_{\l\geq k+ N_0}\br{Q_k(\cP_{\l}g,\tP_{\l}h),\tP_{\l}f}+\sum_{\m\leq k+2N_0}\br{Q_k(\tP_{k}g,\cP_{\m}h),\tP_{\m}f}.\label{Qdecom}
\een
On the other hand, owing to the  Bobylev's equality
 \beno
 \br{Q_k(g,h),f}=\iint_{\R^6} \widehat{a_k}(\xi-\eta):\xi\mult(\xi-2\eta)\hat{g}(\xi-\eta)\hat{h}(\eta)\ov{\hat{f}(\xi)}d\eta d\xi,
 \eeno
we infer that
\ben\label{Qkdecom}
&&\br{Q_k(g,h),f}=\sum_{\p,\l\geq-1}\br{Q_k(\F_{\p}g,\F_{\l}h),f}\notag=\sum_{\l\leq \p-N_0}\br{Q_k(\F_{\p}g,\F_{\l}h),\tF_{\p}f}\\
&&\notag+\sum_{ \l\geq-1}\br{Q_k(\tS_{\l-N_0}g,\F_{\l}h),\tF_{\l}f}+\sum_{\p\geq-1}\br{Q_k(\F_{\p}g,\tF_{\p}h),\tF_{\p}f}+\sum_{\m\leq \p-N_0}\br{Q_k(\F_{\p}g,\tF_{\p}h),\F_{\m}f}\notag\\
&&:=\notag\sum_{\l\leq \p-N_0}\fM^1_{k,\p,\l}(g,h,f)+\sum_{ \l\geq-1}\fM^2_{k,\l}(g,h,f)+\sum_{\p\geq-1}\fM^3_{k,\p}(g,h,f)+\sum_{\m\leq \p-N_0}\fM^4_{k,\p,\m}(g,h,f).\\
&&
\een
For more precise, we actually have the further expressions for $\fM^i(i=1,2,3,4)$ as follows:
\ben
 &&\notag\fM^1_{k,\p,\l}(g,h,f) = \Lambda_1(\tF_\p a_k, \tF_\p c_k,\F_{\p}g,\F_{\l}h, \tF_{\p}f);  \quad
 \fM^2_{k,\l}(g,h,f)=\Lambda_1(a_k,  c_k,\tS_{\l-N_0}g,\F_{\l}h, \tF_{\l}f); \\
 &&\label{fm12}  \\ 
 &&\notag\fM^3_{k,\p}(g,h,f) = \Lambda_1(\tF_\p a_k, \tF_\p c_k,\F_{\p}g,\tF_{\p}h, \tF_{\p}f); \quad
 \fM^4_{k,\p,\l}(g,h,f)=\Lambda_2(\tF_\p a_k, \tF_\p b_k,\F_{\p}g,\tF_{p}h, \tF_{m}f),\\
 &&\label{fm34} 
\een
where we denote that
 $\Lambda_1(a,c,g,h,f):= \int_{\R^3} (a*g):\na^2   h  fdv-\int_{\R^3}( c* g) h fdv$ and  $\Lambda_2(a,b,g,h,f):= \int_{\R^3} (a*g): h  \na^2  fdv+2\int_{\R^3}( b* g)\cdot h fdv$.

\subsection{Upper bound of the collision operator}\label{UPS} Now, we consider the upper bound of the collision operator. Before that,  recalling the definitions of symbol $S^m_{1,0}$ and pseudo-differential operator:
\begin{defi}\label{de2.1}
	A smooth function $a(v,\xi)$ is said to be a symbol of type  $S^m_{1,0}$ if $a(v,\xi)$ verifies for any multi-indices $\alpha$ and $\beta$,
	\beno
	|(\partial_\xi^\alpha\pa_v^\beta a)(v,\xi)|\leq C_{\alpha,\beta}\<\xi\>^{m-|\al|},
	\eeno
	where $C_{\alpha,\beta}$ is a constant depending only on $\alpha$ and $\beta$. $a(x,D)$ is called a pseudo-differential operator with the symbol $a(x,\xi)$ if it is defined by
	\beno
	(a(x,D)f)(x):=\frac{1}{(2\pi)^3}\int_{\R^3}\int_{\R^3}e^{i(x-y)\xi}a(x,\xi)f(y)dyd\xi.
	\eeno
\end{defi}

\begin{defi}\label{Fj} Let  $\alpha=(\alpha_1,\alpha_2,\alpha_3),|\alpha|:=\alpha_1+\al_2+\al_3$ and $\vphi_\alpha:=(\frac{1}{i}\pa_{x_1})^{\alpha_1}(\frac{1}{i}\pa_{x_2})^{\alpha_2}(\frac{1}{i}\pa_{x_3})^{\alpha_3}\vphi$. To simplify the presentation of the estimates for the commutator $[\cP_k, \F_j]$,  we introduce  $\cP_{j,\alpha}$, $\F_{j,\alpha}$ and $\hat{\F}_{j,\al}$ defined by
	\beno
	&&\cP_{-1,\alpha}f:=\psi_{\alpha}f,\quad\quad\quad\quad \cP_{j,\alpha}f:=\vphi_\alpha(2^{-j}\cdot)f,\quad j\geq0;\\
	&&\F_{-1,\alpha}f:=\psi_{\alpha}(D)f,\quad\quad\quad \F_{j,\alpha}:=\vphi_\alpha(2^{-j}D)f,\quad j\geq0;\\
	&&\hat{\F}_{-1,\alpha}f:=(\psi^2)_{\alpha}(D)f,\quad\quad \hat{\F}_{j,\alpha}:=(\vphi^2)_\alpha(2^{-j}D)f,\quad j\geq0.
	\eeno
	Similar to $\tP_j$   and $\tF_j$, we can also  introduce
	$$\tP_{\l,\alpha}:=\sum_{|k-\l|<N_0}\cP_{k,\alpha},\quad \U_{j,\alpha}:=\sum_{k\leq j}\cP_{k,\alpha}, \quad \tF_{j,\alpha}:=\sum_{|k-j|<3N_0}\F_{k,\alpha}.$$
	
	To unify  notations $\tF_j, \tF_{j,\alpha}$, $\hat{\F}_{j,\alpha}$ and $\tP_\l, \tP_{\l,\alpha}$, we introduce  localized operators $\mF_j$   and $\mP_j$ such that
	\smallskip
	
	\noindent {\rm (i)} The support of the Fourier transform of $\mF_jf$ and the support of $\mP_jf$  will be localized in the annulus $\{|\cdot|\sim 2^j\},j\geq0$ or in the sphere $\{|\cdot|\ls 1\},j=-1$;\smallskip
	
	\noindent{\rm (ii)} It holds that for   fixed $N\in\N$, $|\tF_j f|_{L^2}+\sum\limits_{|\alpha|\le N} (|\tF_{j,\alpha}f|_{L^2}+|\hat{\F}_{j,\alpha}f|_{L^2})\le C_N|\mF_jf|_{L^2}$ and $|\tP_j f|_{L^2}+\sum\limits_{|\alpha|\le N} |\tP_{j,\alpha}f|_{L^2}\le C_N|\mP_jf|_{L^2}$.
\end{defi}

For the upper bound of collision operator, we have the following proposition:
\begin{prop}\label{UpperboundofQ}
	\begin{itemize}
		\item[(i).] Let $a_2,b_2\in[0,1],\omega_i\in\R,i=1\cdots,4,\sss>\f12$ verifying that $a_2+b_2=1, \omega_1+\omega_2=-3, \omega_3+\omega_4=-2$. Then
		\ben\label{upperan}
		\nr{\br{Q(g,h),f}}&\ls_{\sss}&\rr{\|g\|_{L^1_3}+\|g\|_{H^{-\f{1}{2},\sss}_{3}}}\rr{\|\rr{-\Delta_{\S^2}}^{\f{1}{2}}h\|_{L^2_{\omega_1}}+\|h\|_{H^1_{\omega_1}}}\rr{\|\rr{-\Delta_{\S^2}}^{\f{1}{2}}f\|_{L^2_{\omega_2}}+\|f\|_{H^1_{\omega_2}}}\notag\\&&+\LLO{g}\|h\|_{H^{a_2}_{\omega_3}}\|f\|_{H^{b_2}_{\omega_4}}.\een 
	    \item[(ii).] Let $a_1,b_1\in[\f12,\f32], \omega_1,\omega_2\in\R,\sss>\f12$ verifying that $a_1+b_1=2, \omega_1+\omega_2=-1$. Then
	    \ben\label{uppersob}
	    \nr{\br{Q(g,h),f}}\ls_{\sss}\rr{\|g\|_{L^1_1}+\|g\|_{H^{-\f{1}{2},\sss}_{1}}}\|h\|_{H^{a_1}_{\omega_1}}\|f\|_{H^{b_1}_{\omega_2}}.\een
	\end{itemize}

\end{prop}
\begin{proof} Due to the decomposition \eqref{Qdecom} and \eqref{Qkdecom}, we split it into two steps.

\underline{Step 1: Preliminary estimates of $\br{Q_k(g,h),f}$ with $k\geq 0$.}  Thanks to \eqref{2.4}, one may have 
  $a_0\in \mathscr{S}$ and 
\[\ba 
&a_k(z)=2^{-k}a_0(2^{-k}z),\quad b_k(z)=2^{-2k}b_0(2^{-k}z),\quad c_k(z)=2^{-3k}c_0(2^{-k}z),\\
&\widehat{a_k}(\xi)=2^{2k}\widehat{a_0}(2^k\xi),\quad \F_p a_k(z)=2^{-k}(\F_{p+k}a_0)(2^{-k}z),\quad \mbox{for}\ p\geq0.
\ea\] Then for any $N\in\N$, it holds that
\ben\label{Fpa}
&& \notag\LLF{a_k}\ls 2^{-k},\quad \LLF{\F_pa_k}\leq 2^{-k}\LLF{\F_{p+k}a_0}\ls_{N}2^{-k}2^{-pN}2^{-kN};\\
&& \notag\LLF{b_k}\ls 2^{-2k},\quad \LLF{\F_pb_k}\leq 2^{-2k}\LLF{\F_{p+k}a_0}\ls_{N}2^{-k}2^{-pN}2^{-kN};\\
 &&\|c_k\|_{L^\infty}\ls 2^{-3k},\quad \LLF{\F_pc_k}\leq 2^{-3k}\LLF{\F_{p+k}c_0}\ls_{N}2^{-3k}2^{-pN}2^{-kN},\quad \mbox{for}\ p\geq 0.
 \een

\noindent$\bullet$ {\it Estimate of $\fM^1$.} Thanks to \eqref{fm12}, we first have
\beno
\fM^1_{k,\p,\l}(g,h,f) =\int ((\tF_\p a_k)*(\F_\p g)):\na^2  ( \F_{\l}h ) \tF_\p fdv-\int( (\tF_\p c_k)* (\F_\p g)) \F_\l h \tF_\p fdv.
\eeno
Thus by Cauchy inequality, Young's inequality and \eqref{Fpa}, together with the fact that $\l\leq \p-N_0$,  we can derive that
\begin{equation}\label{M1}
	\begin{aligned}
&\nr{\fM^1_{k,\p,\l}(g,h,f)} 
\leq\|\tF_\p a_k\|_{L^\infty}\|\F_\p g\|_{L^1}\|\na^2 \F_\l h\|_{L^2}\|\tF_\p f\|_{L^2}
+\|\tF_\p c_k\|_{L^\infty}\|\F_\p g\|_{L^1}\| \F_\l h\|_{L^2}\|\tF_\p f\|_{L^2}\\
\ls_N& 2^{-2kN}2^{-2\p N}\|\F_\p g\|_{L^1}\| \F_\l h\|_{L^2}\|\tF_\p f\|_{L^2}2^{2\l} 
 \ls_N  2^{-kN}2^{-\p N}2^{-\l N}\| g\|_{L^1}\| \F_\l h\|_{L^2}\|\tF_\p f\|_{L^2}.
\end{aligned}
\end{equation}
Thus we can obtain that
\ben\label{sumM1}
\notag\sum_{\l\leq \p-N_0}|\fM^1_{k,\p,\l}(g,h,f)|&\ls_N&\sum_{\l\leq \p-N_0}2^{-kN}2^{-\p N}2^{-\l N}\| g\|_{L^1}\| \F_\l h\|_{L^2}\|\tF_\p f\|_{L^2}\\
&\ls_N&2^{-kN}\|g\|_{L^1}\|h\|_{L^2}\|f\|_{L^2}.
\een

 \noindent$\bullet$ {\it Estimate of $\fM^4$.} By the similar argument as $\fM^1$, we can get that 
 \ben\label{M4} \nr{\fM^4_{k,\p,\m}(g,h,f)} 
\ls_N2^{-kN}2^{-\p N}2^{-\m N}\|g\|_{L^1}\| \tF_\p h\|_{L^2}\|\F_\m f\|_{L^2}.\een
Moreover, we have 
\ben\label{sumM4}
\sum_{\m\leq \p-N_0}|\fM^4_{k,\p,\m}(g,h,f)|\ls_N 2^{-kN}\|g\|_{L^1}\|h\|_{L^2}\|f\|_{L^2}.
\een

 \noindent$\bullet$ {\it Estimates of $\fM^2$ and $\fM^3$.}  From \eqref{fm12} and \eqref{Fpa}, we can deduce that
\ben\label{M2}|\fM^2_{k,\l}(g,h,f)| 
\ls 2^{-k}\|g\|_{L^1}\|\F_\l h\|_{L^2}\|\tF_\l f\|_{L^2}2^{2\l}.\een
Similarly, for $\fM^3$, one has
\ben\label{M3}
\nr{\fM^3_{k,\p}(g,h,f)} 
&\ls_{N}&\left\{\ba
 &2^{-k}\|g\|_{L^1}\|\tF_{-1}h\|_{L^2}\|\tF_{-1}f\|_{L^2}\quad \p=-1\\
 &2^{-kN}2^{-\p N}\|g\|_{L^1}\|\tF_\p h\|_{L^2}\|\tF_\p f\|_{L^2}\quad \p\geq0
 \ea\right.\\&\ls_{N}&2^{-k}\|g\|_{L^1}\|\tF_\p h\|_{L^2}\|\tF_\p f\|_{L^2}2^{-\p N}.\notag
 \een
 Thus combining \eqref{M2} and \eqref{M3}, we have
 \ben\label{sumM2M3}
 \sum_{ \l\geq-1}|\fM^2_{k,\l}(g,h,f)|+\sum_{\p\geq-1}|\fM^3_{k,\p}(g,h,f)|\ls 2^{-k}\|g\|_{L^1}\|h\|_{H^{a_1}}\|f\|_{H^{b_1}}
 \een
 with $a_1,b_1\in[0,2]$ and $a_1+b_1=2$.

 We can also estimate $\fM^2$ and $\fM^3$ in anisotropic spaces. For $\fM^2$,  we observe that 
\beno
\fM^2_{k,\l}(g,h,f)=-\int(a_k*\tS_{\l-N_0}g):\na \F_\l h\mult\na\tF_\l fdv+\int(b_k*\tS_{\l-N_0}g)\cdot\F_\l h\na\tF_\l fdv.
\eeno
Then one can infer that 
\[\ba
&\nr{\int(a_k*\tS_{\l-N_0}g):\na \F_\l h\mult\na\tF_\l fdv}\\
=&\nr{\iint|v-v_*|^{-3}\vf(2^{-k}(v-v_*))\rr{\tS_{\l-N_0}g}_* \rr{\rr{v-v_*}\times \na \F_\l h}\cdot\rr{\rr{v-v_*}\times \na \tF_\l f}dv_*dv}\\
\ls&2^{-3k}\rr{\iint\nr{\rr{\tS_{\l-N_0}g}_*}\nr{\rr{v-v_*}\times \na \F_\l h}^2dv_*dv}^{1/2}\rr{\iint\nr{\rr{\tS_{\l-N_0}g}_*}\nr{\rr{v-v_*}\times \na \tF_\l f}^2dv_*dv}^{1/2}\\
\ls&2^{-3k}\|g\|_{L^1_2}\rr{\LLT{\rr{-\Delta_{\S^2}}^{\f{1}{2}}\F_\l h}+\LLT{\F_\l h}2^\l}\rr{\LLT{\rr{-\Delta_{\S^2}}^{\f{1}{2}}\tF_\l f}+\LLT{\tF_\l f}2^\l}
\ea\]
and 
\beno
 \nr{\int(b_k*\tS_{\l-N_0}g)\cdot\F_\l h\na\tF_\l fdv}\ls 2^{-2k}\LLO{g}\LLT{\F_\l h}\LLT{\tF_\l f}2^\l.
 \eeno 
These imply that
\ben\label{aniso2}
\sum_{\l\geq-1}|\fM^2_{k,\l}(g,h,f)|\ls \sum_{\l\geq-1} 2^{-3k}\|g\|_{L^1_2}\rr{\LLT{\rr{-\Delta_{\S^2}}^{\f{1}{2}}\F_\l h}+\LLT{\F_\l h}2^\l}
\notag\\\times\rr{\LLT{\rr{-\Delta_{\S^2}}^{\f{1}{2}}\tF_\l f}+\LLT{\tF_\l f}2^\l}
+2^{-2k}\LLO{g}\LLT{\F_\l h}\LLT{\tF_\l f}2^\l\notag\\
\ls2^{-3k}\|g\|_{L^1_2}\rr{\LLT{\rr{-\Delta_{\S^2}}^{\f{1}{2}}h}+\|h\|_{H^1}}\rr{\LLT{\rr{-\Delta_{\S^2}}^{\f{1}{2}}f}+\|f\|_{H^1}}.
\een
Similarly, for $\p=-1$, we can get
\ben\label{aniso3}
&&\nr{\fM^3_{k,-1}(g,h,f)}\ls 2^{-3k}\|g\|_{L^1_2}\rr{\LLT{\rr{-\Delta_{\S^2}}^{\f{1}{2}}\tF_{-1}h}+\LLT{\tF_{-1}h}}\notag\\ &&\times
\rr{\LLT{\rr{-\Delta_{\S^2}}^{\f{1}{2}}\tF_{-1}f}+\LLT{\tF_{-1}f}}
+2^{-2k}\LLO{g}\LLT{\tF_{-1}h}\LLT{\tF_{-1}f}.
\een

Now patching together \eqref{sumM1}\eqref{sumM4} and \eqref{sumM2M3}, we can get
\ben\label{upperksob}
\nr{\br{Q_k(g,h),f}}\ls 2^{-k}\LLO{g}\|h\|_{H^{a_1}}\|f\|_{H^{b_1}}\een
for $a_1,b_1\in[0,2]$ and $a_1+b_1=2$. 

On the other hand, combining \eqref{sumM1}\eqref{sumM4}\eqref{aniso2} \eqref{aniso3} and \eqref{M3}($\p\geq0$), we have 
\ben\label{upperkanis}
\nr{\br{Q_k(g,h),f}}&\ls& 2^{-3k}\|g\|_{L^1_2}\rr{\LLT{\rr{-\Delta_{\S^2}}^{\f{1}{2}}h}+\|h\|_{H^1}}\rr{\LLT{\rr{-\Delta_{\S^2}}^{\f{1}{2}}f}+\|f\|_{H^1}}\notag\\&&+2^{-2k}\LLO{g}\|h\|_{H^{a_2}}\|f\|_{H^{b_2}}\een
for $a_2,b_2\in[0,1]$ and $a_2+b_2=1$.

\underline{Step 2:  Estimate of $\sum_{k\geq 0}\br{Q_k(g,h),f}$ involving phase decomposition.}  
Thanks to \eqref{Qdecom}, one has
\ben\label{involweight}
\ba
\sum_{k\geq 0}&\br{Q_k(g,h),f}=\sum_{k\geq N_0-1}\br{Q_k(\U_{k-N_0}g,\tP_{k}h),\tP_{k}f}\\
&+\sum_{\l\geq k+ N_0}\br{Q_k(\cP_{\l}g,\tP_{\l}h),\tP_{\l}f}+\sum_{\m\leq k+2N_0}\br{Q_k(\tP_{k}g,\cP_{\m}h),\tP_{\m}f}.\ea
\een

Combining the first term on the right hand side of \eqref{involweight} and \eqref{upperksob}, and using Cauchy inequality and Lemma \ref{profileofHSk}\eqref{Ber}, we can obtain that
\beno
&&\sum_{k\geq-1}|{\br{Q_k(\U_{k-N_0}g,\tP_{k}h),\tP_{k}f}}|\ls \sum_{k\geq-1}2^{-k}\|\U_{k-N_0}g\|_{L^1}\|\tP_kh\|_{H^{a_1}}\|\tP_kf\|_{H^{b_1}}\\
&\ls&\|g\|_{L^1}\|h\|_{H^{a_1}_{\omega_1}}\|f\|_{H^{b_1}_{\omega_2}},\quad \mbox{with}\quad a_1+b_1=2,\omega_1+\omega_2=-1.
\eeno
Similarly, for the second and the third term on the right hand side of \eqref{involweight}, we have 
\beno
&&\sum_{\l\geq k+ N_0}|\br{Q_k(\cP_{\l}g,\tP_{\l}h),\tP_{\l}f}|\ls \sum_{\l\geq k+N_0}2^{-k}2^{\l} \|\cP_\l g\|_{L^1}\|\tP_\l h\|_{H^{a_1}}\|\tP_\l f\|_{H^{b_1}}2^{-\l}\\
&\ls&\sum_\l \|{\cP_\l g}\|_{L^1_1}\tP_{\l}h\|_{H^{a_1}}\|\tP_\l f\|_{H^{b_1}}|\ls \|g\|_{L^1_1}\|h\|_{H^{a_1}_{\omega_1}}\|f\|_{H^{b_1}_{\omega_2}},
\eeno
and
\beno
&&\sum_{\m\leq k+2N_0}|\br{Q_k(\tP_{k}g,\cP_{\m}h),\tP_{\m}f}|\ls \sum_{\m\leq k+2N_0}2^{-(k-\m)}\|\tP_{k}g\|_{L^1}\|\cP_{\m}h\|_{H^{a_1}}\|\tP_{\m}f\|_{H^{b_1}}2^{-\m}\\
&\ls&\|g\|_{L^1}\|h\|_{H^{a_1}_{\omega_1}}\|f\|_{H^{b_1}_{\omega_2}},\quad \mbox{with}\quad a_1+b_1=2,\omega_1+\omega_2=-1.
\eeno
Patching together the above three three estimates, we obtain that
\ben\label{upperQksob}\sum_{k\geq0}|\br{Q_k(g,h),f}|\ls \|g\|_{L^1_1}\|h\|_{H^{a_1}_{\omega_1}}\|f\|_{H^{b_1}_{\omega_2}},\een
where $a_1,b_1\in[0,2]$, $a_1+b_1=2$ and $\omega_1+\omega_2=-1$.

Now we consider the estimates in anisotropic spaces. Combining the first term on the right hand side of \eqref{involweight} and \eqref{upperkanis},  we can obtain that
\[\ba
&\sum_{k\geq-1}|{\br{Q_k(\U_{k-N_0}g,\tP_{k}h),\tP_{k}f}}|\ls\sum_{k\geq-1}2^{-3k}\|g\|_{L^1_2}\rr{\LLT{\rr{-\Delta_{\S^2}}^{\f{1}{2}}\tP_kh}+\|\tP_kh\|_{H^1}}\\
&\rr{\LLT{\rr{-\Delta_{\S^2}}^{\f{1}{2}}\tP_kf}+\|\tP_kf\|_{H^1}}+\sum_{k\geq-1}2^{-2k}\LLO{g}\|\tP_kh\|_{H^{a_2}}\|\tP_kf\|_{H^{b_2}}\\
\ls&\|g\|_{L^1_2}\rr{\|\rr{-\Delta_{\S^2}}^{\f{1}{2}}h\|_{L^2_{\omega_1}}+\|h\|_{H^1_{\omega_1}}}\rr{\|\rr{-\Delta_{\S^2}}^{\f{1}{2}}f\|_{L^2_{\omega_2}}+\|f\|_{H^1_{\omega_2}}}+\LLO{g}\|h\|_{H^{a_2}_{\omega_3}}\|f\|_{H^{b_2}_{\omega_4}},
\ea\]
where $a_2,b_2\in[0,1], a_2+b_2=1$, $\omega_1+\omega_2=-3$ and $\omega_3+\omega_4=-2$.

For the second term and the third term on the right hand side of \eqref{involweight}, we use \eqref{upperksob} to get that
\beno
&&\sum_{\l\geq k+ N_0}|\br{Q_k(\cP_{\l}g,\tP_{\l}h),\tP_{\l}f}|+\sum_{\m\leq k+2N_0}|\br{Q_k(\tP_{k}g,\cP_{\m}h),\tP_{\m}f}|\\
&\ls&\sum_{\l\geq k+ N_0}2^{-k}\LLO{\cP_\l g}\|\tP_\l h\|_{H^{1}}\|\tP_\l f\|_{H^{1}}+\sum_{\m\leq k+2N_0}2^{-k}\LLO{\tP_kg}\|\cP_\m h\|_{H^{1}}\|\cP_\m f\|_{H^{1}}\\
&\ls& \|g\|_{L^1_3}\|h\|_{H^{1}_{\omega_1}}\|f\|_{H^1_{\omega_2}},\quad \omega_1+\omega_2=-3.
\eeno

Patching together the above two three estimates, we are led to 
\ben\label{upperQkan}  
\sum_{k\geq0}|\br{Q_k(g,h),f}|&\ls&\|g\|_{L^1_3}\rr{\|\rr{-\Delta_{\S^2}}^{\f{1}{2}}h\|_{L^2_{\omega_1}}+\|h\|_{H^1_{\omega_1}}}\rr{\|\rr{-\Delta_{\S^2}}^{\f{1}{2}}f\|_{L^2_{\omega_2}}+\|f\|_{H^1_{\omega_2}}}\notag\\&&+\|g\|_{L^1}\|h\|_{H^{a_2}_{\omega_3}}\|f\|_{H^{b_2}_{\omega_4}},
\een
where $a_2,b_2\in[0,1],a_2+b_2=1$, $\omega_1+\omega_2=-3$ and $\omega_3+\omega_4=-2$.


\underline{Step 3:  Estimate of $\br{Q_{-1}(g,h),f}$.}  We note that  $b_{-1}(z)=b(z)\psi(z)+a(z)\na\psi(z)$ and $c_{-1}(z)=\delta+b(z)\cdot\na\psi(z)+\na\cdot(a(z)\na\psi(z))$. Since
\[\widehat{a}(\xi)=\left(8\pi \f{\xi_i\xi_j}{|\xi|^4}\right)_{i\leq i,j\leq3},\quad\quad \widehat{a_{-1}}(\xi)=\widehat{a}*\hat{\psi}(\xi)=8\pi\int \f{(\xi-\eta)_i(\xi-\eta)_j}{|\xi-\eta|^4}\hat{\psi}(\eta)d\eta,\]
then we have $|\widehat{a_{-1}}(\xi)|\ls\lr{\xi}^{-2}$ and $\|\F_\p a_{-1}\|_{L^2}\ls2^{-\p/2}$.

Now apply these estimates to  $\fM^1$, then by  Bernstein inequality,  we get that
\ben\label{M1s}\nr{\fM^1_{-1,\p,\l}(g,h,f)} 
&\ls&2^{-\p/2}\|\F_\p g\|_{L^2}\|\F_\l h\|_{L^2}\|\tF_\p f\|_{L^2}2^{2\l}+\|\F_\p g\|_{L^2}\|\F_\l h\|_{L^\infty}\|\tF_\p f\|_{L^2}\notag\\
&\ls&\|\F_\p g\|_{L^2}\|\F_\l h\|_{L^2}\|\tF_\p f\|_{L^2}2^{\f32 \l}.\een

Thus for any $a_1\in[0,\f32],b_1\in[0,2]$ verifying $a_1+b_1=2$ and $\sss>\f12$, we have that
\beno
\sum_{\l\leq \p-N_0}\nr{\fM^1_{-1,\p,\l}(g,h,f)}&\ls& \sum_{\l\leq \p-N_0}2^{-\f \p 2}(2+\p)^\sss\|\F_\p g\|_{L^2}2^{a_1\l}\|\F_\l h\|_{L^2}2^{b_1\p}\|\tF_\p f\|_{L^2}2^{(\f32-a_1) \l}2^{(\f12-b_1)\p}(2+\p)^{-\sss}\\
&\ls&\|g\|_{H^{-\f12,\sss}}\|h\|_{H^{a_1}}\|f\|_{H^{b_1}},
\eeno
where we use the fact that $(\f32-a_1)\l+(\f12-b_1)\p=-(\f32-a_1)(\l-\p), a_1\in [0,\f32]$ and $\sss>\f12$.

 Similarly for $\fM^4$, by Hardy-Littlewood-Sobolev inequality we have 
\ben\label{M4s} \nr{\fM^4_{-1,\p,\m}(g,h,f)}&=&\nr{\int (\tF_\p a_{-1}*\F_\p g):\tF_\p h\na^2\F_\m fdv+2\int (b_{-1}*\F_\p g)\cdot\tF_\p h\na\F_\m fdv}\notag\\
&\ls&\|\F_\p g\|_{L^2}\|\tF_\p h\|_{L^2}\|\F_\m f\|_{L^2}2^{\f32\m}.\een
Thus for any $a_1\in[0,2],b_1\in[0,\f32]$ verifying $a_1+b_1=2$ and $\sss>\f12$, we have that
\beno
\sum_{\m\leq \p-N_0}\nr{\fM^4_{-1,\p,\m}(g,h,f)}&\ls& 2^{-\f12\p}(2+\p)^\sss\|\F_\p g\|_{L^2}2^{a_1\p}\|\tF_\p h\|_{L^2}2^{b_1\m}\|\F_\m f\|_{L^2}2^{(\f32-b_1)\m}2^{(\f12-a_1)\p}(2+\p)^{-\sss}\\
&\ls&\|g\|_{H^{-\f12,\sss}}\|h\|_{H^{a_1}}\|f\|_{H^{b_1}},
\eeno
where we use the fact that $(\f32-b_1)\m+(\f12-a_1)\p=-(\f32-b_1)(\m-\p), b_1\in [0,\f32]$ and $\sss>\f12$.

For $\fM^2$ and $\fM^3$, we have
\ben\label{M2s}\nr{\fM^2_{-1,\l}(g,h,f)} 
&\ls& \sum_{\p\leq \l-N_0}2^{-\f{\p}2}\|\F_{\p} g\|_{L^2}\|\F_\l h\|_{L^2}\|\tF_\l f\|_{L^2}2^{2\l},\\
 \label{M3s}|\fM^3_{-1,\p}(g,h,f)|&\ls&\|\F_\p g\|_{L^2}\|\tF_\p h\|_{L^2}\|\tF_\p f\|_{L^2}2^{\f32p}.\een
 Then it is easy to derive that
 \beno
 \sum_{\l\geq-1}\nr{\fM^2_{-1,\l}(g,h,f)}+\sum_{\p\geq-1}\nr{\fM^3_{-1,\p}(g,h,f)}\ls \|g\|_{H^{-\f12,\sss}}\|h\|_{H^{a_1}}\|f\|_{H^{b_1}}
 \eeno
 with $a_1,b_1\in[0,2]$ and $a_1+b_1=2$.

Patching together all the above estimates, if $a_1,b_1\in[\f12,\f32]$ and $a_1+b_1=2$, we have 
\ben\label{q-1}
\nr{\br{Q_{-1}(g,h),f}}\ls_{\eps}\|g\|_{H^{-\f{1}{2},\sss}}\|h\|_{H^{a_1}}\|f\|_{H^{b_1}}.
\een

Observing that for $Q_{-1}$, it holds that $|v-v_*|\ls 1$ which implies $\<v\>\sim\<v_*\>$, tigether with \eqref{q-1} yields that
\ben\label{upperQ-1}\nr{\br{Q_{-1}(g,h),f}}\ls_{\eps}\sum_\l\|\cP_\l g\|_{H^{-\f{1}{2},\sss}}\|\tP_\l h\|_{H^{a_1}}\|\tP_\l f\|_{H^{b_1}}
\ls\|g\|_{H^{-\f{1}{2},\sss}_{\omega_1}}\|h\|_{H^{a_1}_{\omega_2}}\|f\|_{H^{b_1}_{\omega_3}}
\een
for $a_1,b_1\in[\f12,\f32], a_1+b_1=2$ and $\omega_1+\omega_2+\omega_3=0$.

One may easy to check that \eqref{upperQksob} and \eqref{upperQ-1} yield \eqref{uppersob}, \eqref{upperQkan} and \eqref{upperQ-1} yield  \eqref{upperan}. Then we complete the proof of this proposition.
\end{proof}

\subsection{Commutator between $\F_j$ and $Q$}
Now we estimate the commutator between $\F_j$ and $Q$. We split it into two parts:
\ben\label{D1D2definition}
\notag\br{\F_jQ(g,h)-Q(g,\F_jh),\F_jf}=\underbrace{\sum_{k\geq0}\br{\F_jQ_k(g,h)-Q_k(g,\F_jh),\F_jf}}_{:=\mathfrak{D}_1} 
+\underbrace{\br{\F_jQ_{-1}(g,h)-Q_{-1}(g,\F_jh),\F_jf}}_{:=\mathfrak{D}_2}.\\
\een

\subsubsection{Estimate of $\mathfrak{D}_1$} We first give the upper bound of $\mathfrak{D}_1$.
\begin{prop} It holds that
	\ben
	&&|\mathfrak{D}_1|\label{FjQ}\ls_{N}\sum_{k\geq-1} 2^{-2k}\|g\|_{L^1_2}\LLT{\mF_j\mP_kh}\LLT{\mF_j\mP_kf}2^j\notag\\
	&+&\sum_{k\geq-1}2^{-2k}2^{-jN}\|g\|_{L^1_2}\HHL{\mP_kh}\HHL{\mP_kf}+2^{-jN}\LLO{g}\HHLL{h}\HHLL{f},\een
	where $N\in \N$ can be suitably large, and the localized operators $\mF,\mP$ are defined in Definition \ref{Fj}.
\end{prop}

\begin{proof} We divide the proof into several steps.
	
\smallskip

\underline{Step 1: Preliminary estimate of $\br{\F_jQ_k(g,h)-Q_k(g,\F_jh),\F_jf}$ with $j\ge0$.}
 Thanks to \eqref{Qkdecom}, we have
\begin{equation}\label{decom}
	\begin{aligned}
		&\br{\F_jQ_k(g,h)-Q_k(g,\F_jh),\F_jf}=\sum_{\p\geq j+N_0}\br{Q_k(\F_\p g,\tF_\p h),\F_j^2f}+\sum_{\p\leq j-N_0}\br{\F_jQ_k(\F_\p g,\tF_jh)\\
			&-Q_k(\F_\p g,\F_jh),\F_jf}+\sum_{\a\leq j-N_0}\br{Q_k(\tF_jg,\F_\a h),\F_j^2f} +\br{Q_k(\tF_jg,\tF_jh),\F_j^2f}-\br{Q_k(\tF_jg,\F_jh),\F_jf}\\
		&:=\mathscr{B}_1+\mathscr{B}_2+\mathscr{B}_3+\mathscr{B}_4.
	\end{aligned}
\end{equation}

Observing that the $\mathscr{B}_1$, $\mathscr{B}_3$ and $\mathscr{B}_4$ enjoy the similar structure as $\fM^4_{k,j,\a}$,  $\fM^1_{k,j,\a}$ and $\fM^3_{k,j} $, respectively. Thus we can derive from \eqref{M1}\eqref{M4} and \eqref{M3} that
\[\ba
\nr{\mathscr{B}_1}+|\mathscr{B}_3|+|\mathscr{B}_4|&\ls_{N}2^{-kN}2^{-jN}\LLO{g}\|h\|_{H^{-N}}\|f\|_{H^{-N}}.
\ea\]
For $\mathscr{B}_2$, by Fourier transform, 
\beno
\mathscr{B}_2=2^{2k}\iint \widehat{a_0}(2^k(\xi-\eta)):\xi\mult(\xi-2\eta)\widehat{\cS_{j-N_0}g}(\xi-\eta)\widehat{\tF_jh}(\eta)\ov{\widehat{\F_jf}(\xi)}(\vf(2^{-j}\xi)-\vf(2^{-j}\eta))d\eta d\xi.
\eeno
Since $|\xi|,|\xi-2\eta|\ls 2^j$ and $|\vphi(2^{-j}\xi)-\vphi(2^{-j}\eta)|\ls 2^{-j}|\xi-\eta| $, we have that
\beno
|\mathscr{B}_2|
&\ls& 2^{2k}2^j\left(\int |\widehat{a_0}(2^k\xi)||\xi|d\xi\right)\LLO{\cS_{j-N_0}g}\LLT{\tF_jh}\LLT{\F_jf}\ls 2^{-2k}\LLO{g}\LLT{\tF_jh}\LLT{\F_jf}2^j. 
\eeno
These lead to that 
\ben\label{FjQk}
 &&\notag\left|\br{\F_jQ_k(g,h)-Q_k(g,\F_jh),\F_jf}\right|\\
&\ls&2^{-2k}\LLO{g}\LLT{\tF_jh}\LLT{\F_jf}2^j+C_N2^{-2kN}2^{-jN}\LLO{g}\|h\|_{H^{-N}}\|f\|_{H^{-N}}.\een
Following the same argument, one may find that \eqref{FjQk} still holds for $j=-1$.

\smallskip

\underline{Step 2: Dyadic decomposition in phase space.} Recall the definition of $\mathfrak{D}_1$ in \eqref{D1D2definition} and by \eqref{phasedecom}, one may check that 
$\mathfrak{D}_1=  
 \mathfrak{D}_{1,1}+\mathfrak{D}_{1,2}+\mathfrak{D}_{1,3}$ with
\[\ba
\mathfrak{D}_{1,1}=&\sum_{k\geq N_0-1}\br{\F_jQ_k(\U_{k-N_0}g,\tP_{k}h)-Q_k(\U_{k-N_0}g,\F_j\tP_{k}h),\F_j\tP_kf} +\bigg(\sum_{k\geq N_0-1}\br{Q_k(\U_{k-N_0}g,\tP_kh),[\tP_k,\F_j^2]f}\\
& +\br{Q_k(\U_{k-N_0}g,[\F_j,\tP_{k}]h),\F_j\tP_kf}+\br{Q_k(\U_{k-N_0}g,\tP_{k}\F_jh),[\F_j,\tP_{k}]f}\bigg) 
:=\mathfrak{D}_{1,1}^1+\mathfrak{D}_{1,1}^2;\\
\mathfrak{D}_{1,2}=&\sum_{\l\geq k- N_0}\br{\F_jQ_k(\cP_{\l}g,\tP_{\l}h)-Q_k(\cP_{\l}g,\F_j\tP_{\l}h),\F_j\tP_{\l}f} 
 +\bigg(\sum_{\l\geq k- N_0}\br{Q_k(\cP_\l g,\tP_\l h),[\tP_\l,\F_j^2]f}\\
&+ \br{Q_k(\cP_{\l}g,[\F_j,\tP_{\l}]h),\F_j\tP_{\l}f}+\br{Q_k(\cP_{\l}g,\tP_{\l}\F_jh),[\F_j,\tP_{\l}]f}\bigg) 
:=\mathfrak{D}_{1,2}^1+\mathfrak{D}_{1,2}^2;\\
\mathfrak{D}_{1,3}=&\sum_{\m\leq k+2N_0}\br{\F_jQ_k(\tP_kg,\cP_\m h)-Q_k(\tP_kg,\F_j\cP_\m h),\F_j\tP_\m f} 
 +\bigg(\sum_{\m\leq k+2N_0}\br{Q_k(\tP_kg,\cP_\m h),[\tP_\m,\F_j^2]f}\\
&+\br{Q_k(\tP_kg,[\F_j,\cP_\m] h),\F_j\tP_\m f}+\br{Q_k(\tP_{k}g,\cP_{\m}\F_jh),[\F_j,\tP_{\m}]f}\bigg)
:=\mathfrak{D}_{1,3}^1+\mathfrak{D}_{1,3}^2,
\ea\]
where \(\mathfrak{D}_{1,i}^1\) (\(i=1,2,3\)) are the main terms with the structure of commutators between \(\mathcal{F}_j\) and \(Q\), while the remaining terms are commutators between \(\mathcal{F}\) and \(\mathcal{P}\).

\smallskip

$\bullet$ {\it Estimate of $\mathfrak{D}_{1,i}^1$.} Thanks to \eqref{FjQk}, we get that
\[\ba
|\mathfrak{D}_{1,1}^1|&\ls\sum_k2^{-2k}\LLO{g}\LLT{\tF_j\tP_kh}\LLT{\F_j\tP_kf}2^j+C_N\sum_k 2^{-2kN}2^{-jN}\LLO{g}\|\tP_kh\|_{H^{-N}}\|\tP_kf\|_{H^{-N}}\\
&\ls\sum_k2^{-2k}\LLO{g}\LLT{\tF_j\tP_kh}\LLT{\F_j\tP_kf}2^j+C_N2^{-jN}\LLO{g}\|h\|_{H^{-N}_{-N}}\|f\|_{H^{-N}_{-N}};
\ea\]
\[\ba
|\mathfrak{D}_{1,2}^1|&\ls\sum_{k\leq \l+N_0}2^{-2k}\LLO{\cP_\l g}\LLT{\tF_j\tP_\l h}\LLT{\F_j\tP_\l f}2^j+C_N\sum_{k\leq \l+N_0}2^{-2kN}2^{-jN}\LLO{\cP_\l g}\|\tP_\l h\|_{H^{-N}}\|\tP_\l f\|_{H^{-N}}\\
&\ls\sum_{\l}2^{-2\l}\|g\|_{L^1_2}\LLT{\tF_j\tP_\l h}\LLT{\F_j\tP_\l f}2^j+C_N\sum_l2^{-jN}2^{-2\l}\|g\|_{L^1_2}\|\tP_\l h\|_{H^{-N}}\|\tP_\l f\|_{H^{-N}};
\ea\]
\[\ba
|\mathfrak{D}_{1,3}^1|&\ls\sum_{\m\leq k+2N_0}\left(\LLO{\tP_kg}\LLT{\tF_j\cP_\m h}\LLT{\F_j\tP_\m f}2^{-2\m}2^j+C_N2^{-2kN}2^{-jN}\LLO{\tP_kg}\|\cP_\m h\|_{H^{-N}}\|\tP_\m f\|_{H^{-N}}\right)\\
&\ls\sum_\m 2^{-2\m}\LLO{g}\LLT{\tF_j\cP_\m h}\LLT{\F_j\tP_\m f}2^j+C_N2^{-jN}\LLO{g}\|h\|_{H^{-N}_{-N}}\|f\|_{H^{-N}_{-N}}.
\ea\]

$\bullet$ {\it Estimate of $\mathfrak{D}^2_{1,i}$.}  We notice that all the terms in $\mathfrak{D}^2_{1,i}$ are similar. Thus we only provide a detailed estimate for  the first term in $\mathfrak{D}^2_{1,1}$ in the sequel. We observe that 
\[\ba
\br{Q_k&(\U_{k-N_0}g,\tP_kh),[\tP_k,\F_j^2]f}=\sum_{\a\leq \p-N_0}\fM^1_{k,\p,\a}(\U_{k-N_0}g,\tP_kh,[\tP_k,\F_j^2]f)+\sum_{\a\geq-1}\fM^2_{k,\a}(\U_{k-N_0}g,\tP_kh,\\
&[\tP_k,\F_j^2]f)+\sum_{\p\geq-1}\fM^3_{k,\p}(\U_{k-N_0}g,\tP_kh,[\tP_k,\F_j^2]f) +\sum_{\b\leq \p-N_0}\fM^4_{k,\p,\b}(\U_{k-N_0}g,\tP_kh,[\tP_k,\F_j^2]f).
\ea\]

Using \eqref{M1} and Lemma \ref{PFCom}, for some number $\mathfrak{n}>10$, we have
\[\ba
\sum_{\a\leq \p-N_0}&|\fM^1_{k,\p,\a}(\U_{k-N_0}g,\tP_kh,[\tP_k,\F_j^2]f)|\ls_{N}\sum_{\a\leq \p-N_0}2^{-\p N}2^{-kN}2^{-\a N}\LLO{g}\LLT{\F_{\a}\tP_kh}\LLT{\tF_\p[\tP_k,\F_j^2]f}\\
&\ls_{N}2^{-kN}\LLO{g}\HHL{\tP_kh}\Big(\sum_{|\p-j|\leq \mathfrak{n}N_0}2^{-\p N}\LLT{\tF_\p[\tP_k,\F_j^2]f}+\sum_{|\p-j|> \mathfrak{n}N_0}2^{-\p N}\LLT{\tF_\p[\tP_k,\F_j^2]f}\Big).\\&\ls_{N}2^{-kN}2^{-jN}\LLO{g}\|\tP_kh\|_{H^{-N}}\left(\|\mP_kf\|_{H^{-N}}+\HHLL{f}\right).
\ea\]
  
Similarly we can get that
\[\sum_{\b\leq \p-N_0}|\fM^4_{k,\p,\b}(\U_{k-N_0}g,\tP_kh,[\tP_k,\F_j^2]f)|\ls_{N}2^{-kN}2^{-jN}\LLO{g}\|\tP_kh\|_{H^{-N}}\left(\|\mP_kf\|_{H^{-N}}+\HHLL{f}\right).\]

For the term involving $\fM^2_{k,\p,\a}$, by \eqref{M2} and Lemma \ref{PFCom}, one can obtain that
\[\ba
&\sum_{\a\geq-1}|\fM^2_{k,\a}(\U_{k-N_0}g,\tP_kh,[\tP_k,\F_j^2]f)|\ls \sum_{\a\geq-1}2^{-k}\LLO{g}\LLT{\F_\a\tP_kh}\LLT{\tF_\a [\tP_k,\F_j^2]f}2^{2\a}\\
\ls&2^{-k}\LLO{g}\Big(\sum_{|\a-j|\leq \mathfrak{n}N_0}\LLT{\F_\a\tP_kh}\LLT{\tF_\a [\tP_k,\F_j^2]f}2^{2j}+\sum_{|\a-j|> \mathfrak{n}N_0}\LLT{\F_\a\tP_kh}\LLT{\tF_\a [\tP_k,\F_j^2]f}2^{2\a}\Big)\\
\ls_{N}&2^{-2k}\LLO{g}\Big(\sum_{|\a-j|\leq \mathfrak{n}N_0}\LLT{\F_\a\tP_kh}\Big)\Big(\sum_{|\alpha|=1}^{2N}\LLT{\hat{\F}_{j,\alpha}\tP_{k,\alpha}f}\Big)2^j+2^{-kN}2^{-jN}\LLO{g}\|\mP_kh\|_{H^{-N}}\HHLL{f}.
\ea\]
Similarly it holds that
\[\ba
&\sum_{\p\geq-1}|\fM_{k,\p}^3(\U_{k-N_0}g,\tP_kh,[\tP_k,\F_j^2]f)|\\
\ls_{N}&2^{-2k}\LLO{g}\Big(\sum_{|\p-j|\leq \mathfrak{n}N_0}\LLT{\F_\p\tP_kh}\Big)\Big(\sum_{|\alpha|=1}^{2N}\LLT{\hat{\F}_{j,\alpha}\tP_{k,\alpha}f}\Big)2^j+2^{-kN}2^{-jN}\LLO{g}\|\mP_kh\|_{H^{-N}}\HHLL{f}.
\ea\]

Recall Definition \ref{Fj} for $\mF_j$ and $\mP_k$ and  patch together all the estimates, we have that 
\[|\mathfrak{D}_{1,1}^2|\ls_{N}\sum_k 2^{-2k}\LLO{g}\LLT{\mF_j\mP_kh}\LLT{\mF_j\mP_kf}2^j+2^{-jN}\LLO{g}\HHLL{h}\HHLL{f},\]
which is enough to conclude our desired result.
\end{proof}

\subsubsection{Estimate of $\mathfrak{D}_2$}  Recall the definition of $\mathfrak{D}_2$ in \eqref{D1D2definition} and we give the upper bound of $\mathfrak{D}_2$.
\begin{prop} It holds that
	\ben\label{FjQ-1}
&&|\mathfrak{D}_2|\ls_N\sum_{\l\geq-1}\sum_{\p\geq j+N_0}\LLT{\mF_\p\mP_\l g}\LLT{\mF_\p\mP_\l h}\LLT{\mF_j\mP_\l f}2^{\f32j}\notag\\
	&+&\sum_{\l\geq-1}\sum_{\p\leq j-N_0}2^{\f{\p}{2}}\LLT{\mF_\p\mP_\l g}\LLT{\mF_j\mP_\l h}\LLT{\mF_j\mP_\l f}2^j+\sum_{\l\geq-1}\sum_{\a\leq j-N_0}\LLT{\mF_j\mP_\l g}\LLT{\mF_\a\mP_\l h}\LLT{\mF_j\mP_\l f}2^{\f32\a}\notag\\
	&+&\sum_{\l\geq-1}\LLT{\mF_j\mP_\l g}\LLT{\mF_j\mP_\l h}\LLT{\mF_j\mP_\l f}2^{\f32j}+\Big(\sum_{\l\geq-1}\sum_{\p\geq j+N_0}\LLT{\mF_\p\mP_\l g}\LLT{\mF_\p\mP_\l h}\LLT{\mF_j\mP_\l f}2^{\f j2}2^{-\l}\notag\\
		&+&\sum_{\l\geq-1}\sum_{\p\leq j-N_0}\LLT{\mF_\p\mP_\l g}2^{-\f{\p}{2}}\LLT{\mF_j\mP_\l h}\LLT{\mF_j\mP_\l f}2^j2^{-\l}+\sum_{\l\geq-1}\sum_{\a\leq j-N_0}2^{\f \a2}2^{-\l}\LLT{\mF_j\mP_\l g}\LLT{\mF_\a\mP_\l h}\notag\\
		&&\times\LLT{\mF_j\mP_\l f}
		+\sum_{\l\geq-1}\LLT{\mF_j\mP_\l g}\LLT{\mF_j\mP_\l h}\LLT{\mF_j\mP_\l f}2^{\f j2}2^{-\l}+2^{-jN}\|g\|_{L^2_{-N}}\|h\|_{L^2_{-N}}\HHLL{f}\Big),
	\een
		where $N\in \N$ can be suitably large, and the localized operators $\mF,\mP$ are defined in Definition \ref{Fj}.
\end{prop}
\begin{proof} We split the proof into several steps.

\smallskip

\underline{Step 1: Dyadic decomposition in frequency space.} Similar to the decomposition as \eqref{decom}, one may check that 
\[\ba
 \mathfrak{D}_2&=\sum_{\p\geq j+N_0}\br{Q_{-1}(\F_\p g,\tF_\p h),\F_j^2f} +\sum_{\p\leq j-N_0}\br{\F_jQ_{-1}(\F_\p g,\tF_jh)-Q_{-1}(\F_\p g,\F_jh),\F_jf}+\sum_{\a\leq j-N_0}\br{Q_{-1}(\tF_jg,\\
&\F_\a h),\F_j^2f}+\br{Q_{-1}(\tF_jg,\tF_jh),\F_j^2f}-\br{Q_{-1}(\tF_jg,\F_jh),\F_jf} :=\mathscr{B}'_1+\mathscr{B}'_2+\mathscr{B}'_3+\mathscr{B}'_4.
\ea\]

Observing that $\mathscr{B}'_1$, $\mathscr{B}'_3$ and $\mathscr{B}'_4$ have the same structure as $\fM^4_{-1,\p,j}$, $\fM^1_{-1,j,\a}$ and $\fM^3_{-1,j}$, respectively. Thus 
by \eqref{M4s},\eqref{M1s} and \eqref{M3s},
\beno
|\mathscr{B}'_1|+|\mathscr{B}'_3|+|\mathscr{B}'_4|&\ls&\sum_{\p\geq j+N_0}\LLT{\F_\p g}\LLT{\tF_\p h}\LLT{\F_j^2f}2^{\f32j}+\sum_{\a\leq j-N_0}\LLT{\tF_jg}\LLT{\F_\a h}\LLT{\F_j^2f}2^{\f32\a}\\&&+\LLT{\tF_jg}\LLT{\tF_jh}\LLT{\F_jf}2^{\f32j}.
\eeno
Now we estimate $\mathscr{B}'_2$. By Fourier transform, we deduce that
\[\ba
&\left|\br{\F_jQ_{-1}(\F_\p g,\tF_jh)-Q_{-1}(\F_\p g,\F_jh),\F_jf}\right|=\bigg|\iint\wh{a_{-1}}(\xi-\eta):\xi\mult(\xi-2\eta)\widehat{\F_{\p}g}(\xi-\eta)\widehat{\tF_jh}(\eta)\ov{\widehat{\F_jf}(\xi)}(\vf(2^{-j}\xi)\\&-\vf(2^{-j}\eta))d\eta d\xi\bigg| 
\ls \Big(\int |\wh{\tF_\p a_{-1}}(\xi)|^2|\xi|^2d\xi\Big)^{1/2}\LLT{\F_\p g}\LLT{\tF_jh}\LLT{\F_jf}2^j
\ls 2^{\f\p2}\LLT{\F_\p g}\LLT{\tF_jh}\LLT{\F_jf}2^{j},
\ea\]
which yields that 
\beno
|\mathscr{B}'_2|\ls \sum_{\p\leq j-N_0}2^{\p/2}\LLT{\F_\p g}\LLT{\tF_jh}\LLT{\F_jf}2^{j}.
\eeno
Then we conclude that
\ben\label{FjQ-1fre}
 |\mathfrak{D}_2|&\ls&\sum_{\p\geq j+N_0}\LLT{\F_\p g}\LLT{\tF_\p h}\LLT{\F_j^2f}2^{\f32j} 
+\sum_{\p\leq j-N_0}2^{\f\p2}\LLT{\F_\p g}\LLT{\tF_jh}\LLT{\F_jf}2^{j}\notag\\&&+\sum_{\a\leq j-N_0}\LLT{\tF_jg}\LLT{\F_\a h}\LLT{\F_j^2f}2^{\f32\a} 
+\LLT{\tF_jg}\LLT{\tF_jh}\LLT{\F_jf}2^{\f32 j}.
\een

\underline{Step 2: Dyadic decomposition in phase space.} Noticing that for $Q_{-1}$, we have that $|v-v_*|\ls 1$ which implies that $\<v\>\sim\<v_*\>$. Thus we have that 
\begin{equation*}
	\begin{aligned}
\mathfrak{D}_{2} 
&=\sum_{\l\geq-1}\br{\F_jQ_{-1}(\cP_\l g,\tP_\l h)-Q_{-1}(\cP_\l g,\F_j\tP_\l h),\F_j\tP_\l f}+\sum_{\l\geq-1}\big(\br{Q_{-1}(\cP_\l g,\tP_\l h),[\tP_\l,\F_j^2]f}\\
+&\br{Q_{-1}(\cP_\l g,[\F_j,\tP_\l] h),\F_j\tP_\l f}+\br{Q_{-1}(\cP_\l g,\tP_\l\F_j h),[\F_j,\tP_\l]f}\big)
:=\mathfrak{D}_{2}^1+\mathfrak{D}_{2}^2.
	\end{aligned}
\end{equation*}

Thanks to \eqref{FjQ-1fre}, it holds that
\[\ba
|\mathfrak{D}_{2}^1|&\ls\sum_{\l\geq-1}\sum_{\p\geq j+N_0}\LLT{\F_\p\cP_\l g}\LLT{\tF_\p\tP_\l h}\LLT{\F_j^2\tP_\l f}2^{\f32j} +\sum_{\l\geq-1}\sum_{\p\leq j-N_0}2^{\f{\p}{2}}\LLT{\F_\p\cP_\l g}\LLT{\tF_j\tP_\l h}\LLT{\F_j\tP_\l f}2^{j}\\
&+\sum_{\l\geq-1}\sum_{\a\leq j-N_0}\LLT{\tF_j\cP_\l g}\LLT{\F_\a \tP_\l h}\LLT{\F_j^2\tP_\l f}2^{\f32\a} +\sum_{\l\geq-1}\LLT{\tF_j\cP_\l g}\LLT{\tF_j\tP_\l h}\LLT{\F_j\tP_\l f}2^{\f32j}.
\ea\]
Since $\mathfrak{D}_{2}^2$ enjoys the similar structure as $\mathfrak{D}^2_{1,i}$ in the previous proposition, we omit the details and get that
\beno
&&|\mathfrak{D}_{2}^2|\ls_N \sum_{\l\geq-1}\sum_{\p\geq j+N_0}\LLT{\mF_\p\mP_\l g}\LLT{\mF_\p\mP_\l h}\LLT{\mF_j\mP_\l f}2^{\f j2}2^{-\l}
+\sum_{\l\geq-1}\sum_{\p\leq j-N_0}\LLT{\mF_\p\mP_\l g}2^{-\f{\p}{2}}\\
&&\times\LLT{\mF_j\mP_\l h}\LLT{\mF_j\mP_\l f}2^j2^{-\l}+\sum_{\l\geq-1}\sum_{\a\leq j-N_0}\LLT{\mF_j\mP_\l g}\LLT{\mF_\a\mP_\l h}\LLT{\mF_j\mP_\l f}2^{\f \a2}2^{-\l}\notag\\
&&+\sum_{\l\geq-1}\LLT{\mF_j\mP_\l g}\LLT{\mF_j\mP_\l h}\LLT{\mF_j\mP_\l f}2^{\f j2}2^{-\l}+2^{-jN}\|g\|_{L^2_{-N}}\|h\|_{L^2_{-N}}\HHLL{f}.
\eeno
Thus combining the estimates of $\mathfrak{D}_{2}^1$ and $\mathfrak{D}_{2}^2$, we conclude the desired result.
\end{proof}

\subsection{Commutator between $\cP_k$ and $Q$} Next, we estimate the commutator between $\cP_k$ and $Q$.  We also split it into two parts::
\ben\label{C1C2definition}
\notag\br{\cP_kQ(g,h)-Q(g,\cP_kh),f}=\underbrace{\sum_{\l\geq0}\br{\cP_kQ_\l(g,h)-Q_\l(g,\cP_kh),f}}_{:=\mathfrak{E}_1} 
+ \underbrace{ \br{\cP_kQ_{-1}(g,h)-Q_{-1}(g,\cP_kh),f}}_{:= \mathfrak{E}_2}.\\
\een

\subsubsection{Estimate of $\mathfrak{E}_1$} We start with the estimate of $\mathfrak{E}_1$.
\begin{prop} It holds that
	\ben\label{PkQ}
	|\mathfrak{E}_1|\ls_{N}\sum_{\a\geq-1} 2^{-2k}\|g\|_{L^1_2}\LLT{\mF_\a\mP_kh}\LLT{\mF_\a\mP_kf}2^\a
	+2^{-2k}\|g\|_{L^1_2}\HHL{\mP_kh}\HHL{\mP_kf},
	\een
where $N\in \N$ can be suitably large, and the localized operators $\mF,\mP$ are defined in Definition \ref{Fj}.
\end{prop}
\begin{proof}  \underline{Step 1:} We claim that if $\mathfrak{E}_1^\l:=\br{\cP_kQ_\l(g,h)-Q_\l(g,\cP_kh),f},$ then
\ben 
 |\mathfrak{E}_1^\l| 
\ls_N \sum_{\a\geq-1} 2^{-\l}2^{-k}\LLO{g}\LLT{\mF_\a h}\LLT{\mF_\a f}2^\a+2^{-\l}2^{-k}\LLO{g}\HHL{h}\HHL{f}.\label{PkQl} 
\een
Due to the dyadic decomposition in frequency space, we first have
\[\ba
&\br{\cP_kQ_\l(g,h)-Q_\l(g,\cP_kh),f} 
=\sum_{\a\leq \p-N_0}(\br{Q_\l(\F_\p g,\F_\a h),\tF_\p\cP_kf}- \br{Q_\l(\F_\p g,\F_\a \cP_kh),\tF_\p f})\\
&+\sum_{\a\geq-1}(\br{Q_\l(\tS_{\a-N_0} g,\F_\a h),\tF_\a \cP_kf}- \br{Q_\l(\tS_{\a-N_0} g,\F_\a\cP_k h),\tF_\a f})+\sum_{\p\geq -1}(\br{Q_\l(\F_\p g,\tF_\p h),\tF_\p\cP_kf}\\
&-\br{Q_\l(\F_\p g,\tF_\p\cP_kh),\tF_\p f}) +\sum_{\b\leq \p-N_0}\br{Q_\l(\F_\p g,\tF_\p h),\F_\b\cP_kf}- \br{Q_\l(\F_\p g,\tF_\p\cP_kh),\F_\b f}) :=\sum_{i=1}^4\mathscr{C}_i.
\ea\]
Thanks to \eqref{fm12} and \eqref{fm34}, we further have the decomposition as follows:
\[\ba
\mathscr{C}_1&=\sum_{\a\leq \p-N_0} (\Lambda_1(\tF_\p a_\l, \tF_\p c_\l,\F_{\p}g,\F_{\a}h, \cP_k\tF_{\p}f)-\Lambda_1(\tF_\p a_\l, \tF_\p c_\l,\F_{\p}g,\cP_k\F_{\a}h, \tF_{\p}f))+\sum_{\a\leq \p-N_0}\Lambda_1(\tF_\p a_\l,\\
& \tF_\p c_\l,\F_{\p}g,\F_{\a}h, [\tF_\p,\cP_k]f)+\sum_{\a\leq \p-N_0} \Lambda_1(\tF_\p a_\l, \tF_\p c_\l,\F_{\p}g, 
[\cP_k,\F_{\a}]h, \tF_\p f):=\mathscr{C}_1^1+\mathscr{C}_1^2+\mathscr{C}_1^3;\\
\mathscr{C}_2&=\sum_{\a\geq-1} (\Lambda_1( a_\l,   c_\l,\tS_{\a-N_0}g,\F_{\a}h, \cP_k\tF_{\a}f)-\Lambda_1(  a_\l,   c_\l,\tS_{\a-N_0}g,\cP_k\F_{\a}h, \tF_{\a}f))+\sum_{\a\geq-1}\Lambda_1(  a_\l,   c_\l,\tS_{\a-N_0}g,\\
&\F_{\a}h, [\tF_\a,\cP_k]f)+\sum_{\a\geq-1} \Lambda_1(  a_\l,   c_\l,\tS_{\a-N_0}g, 
[\cP_k,\F_{\a}]h, \tF_\a f) :=\mathscr{C}_2^1+\mathscr{C}_2^2+\mathscr{C}_2^3;\\
\mathscr{C}_3&=\sum_{\p\geq -1}(\Lambda_1(\tF_\p a_\l, \tF_\p c_\l,\F_{\p}g,\tF_{\p}h, \cP_k\tF_{\p}f)-\Lambda_1(\tF_\p a_\l, \tF_\p c_\l,\F_{\p}g,\cP_k\tF_{\p}h, \tF_{\p}f))+\sum_{\p\geq -1}\Lambda_1(\tF_\p a_\l, \tF_\p c_\l,\F_{\p}g,\\
&\tF_{\p}h, [\tF_\p,\cP_k]f)+\sum_{\p\geq -1} \Lambda_1(\tF_\p a_\l, \tF_\p c_\l,\F_{\p}g, 
[\cP_k,\tF_{\p}]h, \tF_\p f):=\mathscr{C}_3^1+\mathscr{C}_3^2+\mathscr{C}_3^3;\\
\mathscr{C}_4&=\sum_{\b\leq \p-N_0} (\Lambda_2(\tF_\p a_\l, \tF_\p b_\l,\F_{\p}g,\tF_{\p}h, \cP_k\F_{\b}f)-\Lambda_2(\tF_\p a_\l, \tF_\p b_\l,\F_{\p}g,\cP_k\tF_{\p}h, \F_{\b}f))+\sum_{\b\leq \p-N_0}\Lambda_2(\tF_\p a_\l, \tF_\p b_\l,\\
&\F_{\p}g, \tF_{\p}h, [\F_\b,\cP_k]f)+\sum_{\b\leq \p-N_0}\Lambda_2(\tF_\p a_\l, \tF_\p b_\l,\F_{\p}g,[\cP_k,\tF_\p]h, \F_\b f):=\mathscr{C}_4^1+\mathscr{C}_4^2+\mathscr{C}_4^3.
\ea\]

In the below, we only provide a detailed proof for $\mathscr{C}_1$ since the other terms can be treated similarly. Thanks to the computation that 
\beno
&&\Lambda_1( a_\l,   c_\l, g, h, \cP_k f)-\Lambda_1(  a_\l,   c_\l,g,\cP_k h, f)
=\int(a_\l*g):\Big(2^{-2k}\na^2\vf(2^{-k}v)hf+2^{1-k}\na\vf(2^{-k}v)\mult\na fh\Big)dv\notag\\
&&\qquad\qquad\qquad\qquad\qquad\qquad\qquad\qquad\qquad+2^{1-k}\int(b_\l*g)\cdot\na\vf(2^{-k}v)hfdv\\
 &&\quad\qquad\qquad\qquad\qquad\qquad\qquad=-\int(a_\l*g):\left(2^{-2k}\na^2\vf(2^{-k}v)hf+2^{1-k}\na\vf(2^{-k}v)\mult\na hf\right)dv.
\eeno
Then by the similar argument as \eqref{M1}, we can derive that 
\beno |\mathscr{C}_1^1| \ls_N 2^{-\l N}2^{-k}\LLO{g}\HHL{h}\HHL{f}.\eeno 
For $\mathscr{C}_1^2$ and $\mathscr{C}_1^3$, by \eqref{M1} and Lemma \ref{PFCom}, one has
\[\ba
|\mathscr{C}_1^2|+|\mathscr{C}_1^3|  
&\ls_{N}\sum_{\a\leq \p-N_0}2^{-\l N} 2^{-\p N}2^{-\a N}\LLO{g}\LLT{\F_\a h}\Big(\sum_{|\alpha|=1}^{2N}\LLT{\tF_{\p,\alpha}\cP_{k,\alpha}f}2^{-\p}2^{-k}+2^{-kN}2^{-\p N}\HHL{f}\Big)\\ &+\sum_{\a\leq \p-N_0}2^{-\l N}2^{-\p N}2^{-\a N}\LLO{g}\Big(\sum_{|\alpha|=1}^{2N}\LLT{\tF_{\a,\alpha}\cP_{k,\alpha}h}2^{-\a}2^{-k}+2^{-kN}2^{-\a N}\HHL{h}\Big)\LLT{\tF_\p f} \\
&\ls_{N}2^{-k}2^{-\l N}\LLO{g}\HHL{h}\HHL{f}.
\ea\]

We also mention that   
for $\mathscr{C}_2^2$, by \eqref{M2} and Lemma \ref{PFCom},
\[\ba
|\mathscr{C}_2^2|  
&\ls_{N}\sum_{\a \geq-1}2^{-\l}\LLO{g}\LLT{\F_\a h}\Big(\sum_{|\alpha|=1}^{2N}\LLT{\cP_{k,\alpha}\tF_{\a,\alpha}f}2^{-\a}2^{-k}+2^{-kN}2^{-\a N}\HHL{f}\Big)2^{2\a}\\
&\ls 2^{-\l}2^{-k}\LLO{g}\LLT{\F_\a h}\sum_{|\alpha|=1}^{2N}\LLT{\tF_{\a,\alpha}f}2^{\a}+2^{-\l}2^{-kN}\HHL{h}\HHL{f}.
\ea\]
These are enough to conclude \eqref{PkQl}. We remark that by the similar argument \eqref{PkQl} holds true for $k=-1$.

\smallskip 

\underline{Step 2.} Now we apply the decomposition in phase space.
\[\ba
\mathfrak{E}_1&=\sum_{\l\geq0}\br{\cP_kQ_\l(g,h)-Q_\l(g,\cP_kh),f}=\sum_{0\leq \l\leq k-N_0}\br{\cP_kQ_\l(\tP_kg,\tP_kh)-Q_\l(\tP_kg,\cP_k\tP_kh),\tP_kf}\\
&+\sum_{\l\geq k+N_0}\br{\cP_kQ_\l(\tP_\l g,\tP_kh)-Q_\l(\tP_\l g,\cP_k\tP_kh),\tP_kf}+\sum_{|\l-k|\leq N_0}\br{\cP_kQ_\l(\U_{k+2N_0}g,\tP_kh)\\
&-Q_\l(\U_{k+2N_0}g,\cP_k\tP_kh),\tP_kf}:=\mathfrak{E}_{1,1}+\mathfrak{E}_{1,2}+\mathfrak{E}_{1,3}.
\ea\]
Thanks to \eqref{PkQl}, we can deduce that
\[\ba
|\mathfrak{E}_{1,1}|&\ls_{N}\sum_{0\leq \l\leq k-N_0}\sum_\a 2^{-\l}2^{-k} \LLO{\tP_kg}\LLT{\mF_\a\tP_kh}\LLT{\mF_\a\tP_kf}2^\a+\sum_{0\leq \l\leq k-N_0}2^{-\l}2^{-k}\LLO{\tP_kg}\HHL{\tP_kh}\\
&\times\HHL{\tP_kf} \ls \sum_\a 2^{-3k}\|g\|_{L^1_2}\LLT{\mF_\a\tP_kh}\LLT{\mF_\a\tP_kf}2^\a+2^{-3k}\|g\|_{L^1_2}\HHL{\tP_kh}\HHL{\tP_kf}; \\
|\mathfrak{E}_{1,2}|&\ls_{N}\sum_{\l\geq k+N_0}\sum_\a 2^{-\l}2^{-k} \LLO{\tP_\l g}\LLT{\mF_\a\tP_kh}\LLT{\mF_\a\tP_kf}2^\a+\sum_{\l\geq k+N_0}2^{-\l}2^{-k}\LLO{\tP_\l g}\HHL{\tP_kh}\HHL{\tP_kf}\\
&\ls \sum_\a 2^{-3k}\|g\|_{L^1_2}\LLT{\mF_\a\tP_kh}\LLT{\mF_\a\tP_kf}2^\a+2^{-3k}\|g\|_{L^1_2}\HHL{\tP_kh}\HHL{\tP_kf}\\
|\mathfrak{E}_{1,3}|&\ls_{N}\sum_{|\l-k|\leq N_0}\sum_\a 2^{-\l}2^{-k} \LLO{g}\LLT{\mF_\a\tP_kh}\LLT{\mF_\a\tP_kf}2^\a 
+\sum_{|\l-k|\leq N_0}2^{-\l}2^{-k}\LLO{g}\HHL{\tP_kh}\HHL{\tP_kf}\\
&\ls\sum_\a 2^{-2k}\|g\|_{L^1}\LLT{\mF_\a\tP_kh}\LLT{\mF_\a\tP_kf}2^\a+2^{-2k}\|g\|_{L^1}\HHL{\tP_kh}\HHL{\tP_kf}.
\ea\] 
These imply our desired result and we end the proof.
\end{proof}

\subsubsection{Estimate of $\mathfrak{E}_2$} We now turn to the estimate of \(\mathfrak{E}_2\), recalling its definition in \eqref{C1C2definition}.
\begin{prop} It holds that
	\ben\label{PkQ-1}
	&&\notag|\mathfrak{E}_2|\ls_{N}\sum_{\a\leq \p-N_0}2^{-k}\LLT{\mF_\p\mP_kg}\LLT{\mF_\a\mP_k h}\LLT{\mF_\p\mP_k f}2^{\f{\a}{2}}
	+\sum_{\p\leq \a-N_0}2^{-k}2^{-\f{\p}{2}}\LLT{\mF_\p\mP_kg}\LLT{\mF_\a\mP_k h}\\
	&&\times\LLT{\mF_\a\mP_k f}2^\a
	+\sum_{\p\geq-1}2^{-k}\LLT{\mF_\p\mP_kg}\LLT{\mF_\p\mP_kh}\LLT{\mF_\p\mP_kf}2^{\f{\p}{2}}+\sum_{\b\leq \p-N_0}2^{-k}\LLT{\mF_\p\mP_kg}\LLT{\mF_\p\mP_kh}\notag\\
	&&\times\LLT{\mF_\b\mP_k f}2^{\f\b2}+\sum_{\p\geq-1}2^{-kN}\LLT{\mF_\p\mP_kg}\LLT{\mF_\p\mP_kf}\HHLL{\mP_kh}+\sum_{\p\geq-1}2^{-kN}\LLT{\mF_\p\mP_kg}\notag\\
	&&\times\LLT{\mF_\p\mP_kh}\HHLL{\mP_kf}
	+2^{-kN}\HHL{\mP_kg}\HHL{\mP_kh}\HHL{\mP_kf},
	\een
	where $N\in \N$ can be suitably large, and the localized operators $\mF,\mP$ are defined in Definition \ref{Fj}.
\end{prop}
\begin{proof} \underline{Step 1.} 
We first apply the decomposition to $\mathfrak{E}_2$ in frequency space (see \eqref{Qkdecom}).
\[\ba
\mathfrak{E}_2  
&=\sum_{\a\leq \p-N_0}(\br{Q_{-1}(\F_\p g,\F_\a h),\tF_\p\cP_kf}-\br{Q_{-1}(\F_\p g,\F_\a \cP_kh),\tF_\p f})+\sum_{\a\geq-1}(\br{Q_{-1}(\tS_{\a-N_0} g,\F_\a h),\tF_\a \cP_kf}\\&-\br{Q_{-1}(\F_\p g,\F_\a\cP_k h),\tF_\a f})
+\sum_{\p\geq -1}(\br{Q_{-1}(\F_\p g,\tF_\p h),\tF_\p\cP_kf}-\br{Q_{-1}(\F_\p g,\tF_\p\cP_kh),\tF_\p f})\\&+\sum_{\b\leq \p-N_0}\br{Q_{-1}(\F_\p g,\tF_\p h),\F_\b\cP_kf}-\br{Q_{-1}(\F_\p g,\tF_\p\cP_kh),\F_\b f}):
=\mathscr{C'}_1+\mathscr{C'}_2+\mathscr{C'}_3+\mathscr{C'}_4,
\ea\]
where
\[\ba
\mathscr{C'}_1&=\sum_{\a\leq \p-N_0}\br{\cP_kQ_{-1}(\F_\p g,\F_\a h)-Q_{-1}(\F_\p g,\cP_k\F_\a h),\tF_\p f}+\sum_{\a\leq \p-N_0}\br{Q_{-1}(\F_\p g,\F_\a h),[\tF_\p,\cP_k]f}\\
&+\sum_{\a\leq \p-N_0}\br{Q_{-1}(\F_\p g,[\cP_k,\F_\a]h),\tF_\p f}:=\mathscr{C'}_1^1+\mathscr{C'}_1^2+\mathscr{C'}_1^3;\\
\mathscr{C'}_2&=\sum_{\a\geq-1}\br{\cP_kQ_{-1}(\tS_{\a-N_0} g,\F_\a h)-Q_{-1}(\tS_{\a-N_0} g, \cP_k\F_\a h),\tF_\a f}+\sum_{\a\geq-1}\br{Q_{-1}(\tS_{\a-N_0} g,\F_\a h),[\tF_\a,\cP_k]f}\\
&+\sum_{\a\geq-1}\br{Q_{-1}(\tS_{\a-N_0} g,[\cP_k,\F_\a]h),\tF_\a f}:=\mathscr{C'}_2^1+\mathscr{C'}_2^2+\mathscr{C'}_2^3;
\ea\]
\[\ba
\mathscr{C'}_3&=\sum_{\p\geq -1}\br{\cP_kQ_{-1}(\F_\p g,\tF_\p h)-Q_{-1}(\F_\p g,\cP_k\tF_\p h),\tF_\p f}+\sum_{\p\geq -1}\br{Q_{-1}(\F_\p g,\tF_\p h),[\tF_\p,\cP_k]f}\\
&+\sum_{\p\geq -1}\br{Q_{-1}(\F_\p g,[\cP_k,\tF_\p]h),\tF_\p f}:=\mathscr{C'}_3^1+\mathscr{C'}_3^2+\mathscr{C'}_3^3;\\
\mathscr{C'}_4&=\sum_{\b\leq \p-N_0}\br{\cP_kQ_{-1}(\F_\p g,\tF_\p h)-Q_{-1}(\F_\p g,\cP_k\tF_\p h),\F_\b f}+\sum_{\b\leq \p-N_0}\br{Q_{-1}(\F_\p g,\tF_\p h),[\F_\b,\cP_k]f}\\
&+\sum_{\b\leq \p-N_0}\br{Q_{-1}(\F_\p g,[\cP_k,\tF_\p]h),\F_\b f}:=\mathscr{C'}_4^1+\mathscr{C'}_4^2+\mathscr{C'}_4^3.
\ea\]

In the sequel, we will choose $\mathscr{C'}_1^1,\mathscr{C'}_4^1$, $\mathscr{C'}_4^2$ and  $\mathscr{C'}_1^3$ as typical terms since the others can be treated similarly.  For $k\geq0$, we hava that
\ben\label{PkQ-1forf}
&&\br{\cP_kQ_{-1}(g,h)-Q_{-1}(g,\cP_kh),f}=\int_{\R^3}(a_{-1}*g):\left(2^{-2k}\na^2\vf(2^{-k}v)hf+2^{1-k}\na\vf(2^{-k}v)\mult\na fh\right)dv\notag\\
&&\qquad\qquad\qquad\qquad\qquad\qquad\qquad\qquad\qquad+2^{1-k}\int_{\R^3}(b_{-1}*g)\cdot\na\vf(2^{-k}v)hfdv\\
\label{PkQ-1forh}&&\qquad\qquad\qquad\qquad\qquad=\int_{\R^3}(a_{-1}*g):\left(2^{-2k}\na^2\vf(2^{-k}v)hf+2^{1-k}\na\vf(2^{-k}v)\mult\na hf\right)dv.
\een
 We note that $\mathscr{C'}_1^1$ and $\mathscr{C'}_4^1$ can be bounded from \eqref{PkQ-1forh} and \eqref{PkQ-1forf}, respectively.
 Indeed, we can deduce that
  \beno |\mathscr{C'}_1^1|+|\mathscr{C'}_4^1|\ls_{N}\sum_{\a\leq \p-N_0}2^{-k}\LLT{\mF_\p g}\LLT{\mF_\a h}\LLT{\mF_\p f}2^{\f\a2}+\sum_{\b\leq \p-N_0}2^{-k}\LLT{\mF_\p g}\LLT{\mF_\p h}\LLT{\mF_\b f}2^{\f\b2}. \eeno

 For $\mathscr{C'}_4^2$, by \eqref{M4s} and Lemma \ref{PFCom}, we  infer that
\[\ba
 |\mathscr{C'}_4^2|  
\ls_{N}\sum_{\b\leq \p-N_0}2^{-k}\LLT{\F_\p g}\LLT{\tF_\p h}\Big(\sum_{|\alpha|=1}^{2N}\LLT{\F_{\b,\alpha}f}\Big)2^{\f{\b}{2}}+\sum_{\p\geq-1}2^{-kN}\LLT{\F_pg}\LLT{\tF_ph}\HHLL{f}.
\ea\]

For $\mathscr{C'}_1^3$, by \eqref{M1s} and Lemma \ref{PFCom}, we deduce that
\[\ba
 |\mathscr{C'}_1^3| 
\ls_{N}\sum_{\a\leq \p-N_0}2^{-k}\LLT{\F_\p g}\Big(\sum_{|\alpha|=1}^{2N}\LLT{\F_{\a,\alpha}h}\Big)\LLT{\tF_\p f}2^{\f\a2}+\sum_{\p\geq-1}2^{-kN}\LLT{\F_\p g}\LLT{\tF_\p f}\HHLL{h}.
\ea\]

Patching together all above estimates, we conclude that 
\ben\label{E2Est}
 &&|\mathfrak{E}_2|\ls_{N}\sum_{\a\leq \p-N_0}2^{-k}\LLT{\mF_\p g}\LLT{\mF_\a h}\LLT{\mF_\p f}2^{\f\a2} 
+\sum_{\p\leq \a-N_0}2^{-k}2^{-\f\p2}\LLT{\mF_\p g}\LLT{\mF_\a h}\LLT{\mF_\a f}2^\a\notag\\&&+\sum_{\p\geq-1}2^{-k}\LLT{\mF_\p g}\LLT{\mF_\p h}\LLT{\mF_\p f}2^{\f\p2}
+\sum_{\b\leq \p-N_0}2^{-k}\LLT{\mF_\p g}\LLT{\mF_\p h}\LLT{\mF_\b f}2^{\f\b2}+\sum_{\p\geq-1}2^{-kN}\LLT{\mF_\p g}\notag\\
&&\quad\times\LLT{\mF_\p f}\HHLL{h}+\sum_{\p\geq-1}2^{-kN}\LLT{\mF_\p g}\LLT{\mF_\p h}\HHLL{f}+2^{-kN}\HHL{g}\HHL{h}\HHL{f}.
\een
We remark that the same estimate holds for $k=-1$ by the similar argument.

\underline{Step 2.} Since for $Q_{-1}$, we have that $\<v\>\sim\<v_*\>$, the following dyadic decomposition to $\mathfrak{E}_2$ in phase space holds:
\[\br{\cP_kQ_{-1}(g,h)-Q_{-1}(g,\cP_kh),f}=\br{\cP_kQ_{-1}(\tP_kg,\tP_kh)-Q_{-1}(\tP_kg,\cP_k\tP_kh)\tP_kf}.\]
Thus by \eqref{E2Est} and we can obtain the desired result and it ends the proof.
\end{proof}

\section{Proof of Theorem \ref{uniquethm}(1) for \texorpdfstring{$\sss>0$}~~and Theorem \ref{T1} (1.1) for \texorpdfstring{$\mathsf{r}\in(-\f12,0]$}.}\label{section3}
With the previous estimates in hand, we initiate the proof of certain results from Theorem \ref{uniquethm} and Theorem \ref{T1} in this section. We start by addressing the local existence in the space $H^{-\frac{1}{2}, \sss}_l$ with $\sss > 0, l\geq3$.

\begin{proof}[Proof of Theorem \ref{uniquethm}(1) for $\sss>0$] We will split the proof into several steps.

\smallskip
	
\underline{Step 1: The localized equation from weak solutions.} Since $f_0\in L^1_{2l+1}\cap L\log L$, then due to \cite{D} and \cite{Vi}, there exists a $H-$solution or weak solution such that
\ben\label{l1l3}
f\in L^\infty([0,1],L^1_{2l+1})\cap L^1([0,1],L^3_{-3}).
\een

Recall the definition of weak solutions in \cite{D} for Landau-Coulomb equation, for any $0\leq s<t\leq 1$, choose $$\cP_k\F_j^2\cP_kf(s),\cP_k\F_j^2\cP_kf(t)\in C^\infty_c(\R^3)$$ as test functions, we can derive that
\beno
&&\int_{\R^3} f(t)\cP_k\F_j^2\cP_kf(t)dv-	\int_{\R^3} f(s)\cP_k\F_j^2\cP_kf(t)dv=\int_s^t \<Q(f,f)(\tau,\cdot),	 \cP_k\F_j^2\cP_kf(t)\>_vd\tau;\\
&&\int_{\R^3} f(t)\cP_k\F_j^2\cP_kf(s)dv-	\int_{\R^3} f(s)\cP_k\F_j^2\cP_kf(s)dv=\int_s^t \<Q(f,f)(\tau,\cdot),	 \cP_k\F_j^2\cP_kf(s)\>_vd\tau.
\eeno
Taking sum of both equations, yields 
\ben\label{006}
\|\F_j\cP_kf(t)\|^2_{L^2}-\|\F_j\cP_kf(s)\|^2_{L^2}=\int_s^t \<Q(f,f)(\tau,\cdot),	 \cP_k\F_j^2\cP_kf(t)+\cP_k\F_j^2\cP_kf(s)\>_vd\tau.
\een
Thanks to the proof of Corollary 1.1 in \cite{D}, we deduce from \eqref{006} that
\begin{equation}\label{007}
	\begin{aligned}
\big|\|\F_j\cP_kf(t)\|^2_{L^2}-\|\F_j\cP_kf(s)\|^2_{L^2}\big|&\ls\sup_{\tau\in [s,t]}\|f(\tau)\|_{L^1_2}\int_{s}^t\|f(\tau)\|_{L^3_{-3}}d\tau \|\cP_k\F_j^2\cP_k( f(s)+f(t))\|_{H^4}\\
&\ls_{k,j}\sup_{\tau\in [s,t]}\|f(\tau)\|^2_{L^1_2}\int_{s}^t\|f(\tau)\|_{L^3_{-3}}d\tau,
	\end{aligned}
	\end{equation}
where we use the fact $\|\cdot\|_{C^2}\ls \|\cdot\|_{H^4}$ and Lemma \ref{7.8} in the last step. Observing that $f\in L^1([0,1],L^3_{-3}(\R^3))$, thus $\|\F_j\cP_k f(t)\|_{L^2}$ is a continuous function with respect to \( t \). In fact, it is differentiable almost everywhere. By the same argument, choose $\cP_k \F_j\phi$  as the test function with $\phi\in L^2(\R^3)$, we can deduce that $\<\F_j\cP_k f(t),\phi\>$ is also continuous with respect to \( t \).
Therefore, by Randon-Riesz Theorem, we have that
\ben\label{010}
\F_j\cP_kf\in C([0,1],L^2(\R^3)).
\een

Now multiply $(t-s)^{-1}$ on both side of \eqref{006} and let $t\rightarrow s$. Use \eqref{010}, we have
\beno
&&\f1{t-s}\int_s^t\<Q(f,f)(\tau,\cdot),\cP_k\F_j^2\cP_k(f(t)+f(s)-2f(\tau))\>_v d\tau
\ls_{k,j} \sup_{\tau\in [s,t]}\|f(\tau)\|_{L^1_2}\\
&&\times\f1{t-s}\int_{s}^t\|f(\tau)\|_{L^3_{-3}}d\tau\sup_{\tau\in[s,t]}(\|\F_j\cP_k(f(t)-f(\tau))\|_{L^2}+\|\F_j\cP_k(f(s)-f(\tau))\|_{L^2})\rightarrow 0
\eeno
as $t\rightarrow s$. Then it holds that
\ben\label{005}
\f12 \f d{dt}\|\F_j\cP_kf\|^2_{L^2}=\<Q(f,f), \cP_k\F_j^2\cP_kf\>_v,
\een
or equivalently,
\beno
\|\F_j\cP_k f(t)\|^2_{L^2}-\|\F_j\cP_k f(s)\|^2_{L^2}=2\int_s^t\<Q(f,f),\cP_k\F_j^2\cP_k f\>_{v}(\tau) d\tau,\quad \mbox{for}\quad 0\leq s<t\leq 1.
\eeno
Henceforth, we will utilize \eqref{005} for our energy estimates.

\underline{Step 2: Energy estimates from the localized equation.} Starting from \eqref{005} and the definition of collision operator $Q$, we can further derive that
	\[\ba
	&\f{1}{2}\f{d}{dt}\LLT{\F_j\cP_kf}^2+\int_{\R^3}(a*f):\na\F_j\cP_kf\mult\na\F_j\cP_kfdv\\
	=&4\pi\int_{\R^3} f|\F_j\cP_kf|^2dv+\br{\F_j\cP_kQ(f,f)-Q(f,\F_j\cP_kf),\F_j\cP_kf}.
	\ea\]
	For the coercivity term, thanks to Proposition \ref{coer}, Lemma \ref{PFCom} and Lemma \ref{profileofHSk}, one can check that
	\[\ba
    &\int_{\R^3}(a*f):\na\F_j\cP_kf\mult\na\F_j\cP_kfdv \geq \kappa \LLT{\F_j\cP_kf}^22^{-3k}2^{2j}-C_N\left(\LLT{\mF_j\mP_kf}^22^{-3k}+2^{-kN}2^{-jN}\HHLL{f}^2\right)
	\ea\]
	for some $\kappa>0$. For the first term on the right hand side, by \eqref{gFjkhf}, we have that
	\beno
	4\pi\int_{\R^3} f|\F_j\cP_kf|^2dv\ls_N \sum_{m\leq j+N_0}\LLT{\mF_m\mP_kf}2^{\f32m}\LLT{\mF_j\mP_kf}^2+\HHLL{f}\LLT{\mF_j\mP_kf}^2+2^{-jN}2^{-kN}\HHLL{f}^3.
	\eeno
	For the second term, we observe the following decomposition
	\ben\label{decomofcom}
	&&\notag\br{\F_j\cP_kQ(f,f)-Q(f,\F_j\cP_kf),\F_j\cP_kf}\\
	&=&\br{\F_jQ(f,\cP_kf)-Q(f,\F_j\cP_kf),\F_j\cP_kf}+\br{\cP_kQ(f,f)-Q(f,\cP_kf),\F_j^2\cP_kf}.
	\een
	Thus by virtue of  \eqref{FjQ}\eqref{FjQ-1}\eqref{PkQ} and \eqref{PkQ-1}, we can deduce that
	\beno 
	&&\nr{\br{\F_j\cP_kQ(f,f)-Q(f,\F_j\cP_kf),\F_j\cP_kf}}\ls_N\sum_{\p\geq j+N_0}\LLT{\mF_\p\mP_kf}^2\LLT{\mF_j\mP_kf}2^{\f32j}+\sum_{\p\leq j+N_0}\|\mF_\p\mP_kf\|_{L^2}2^{\f\p2}\\
	&&\times\|\mF_j\mP_kf\|^2_{L^2}2^j+2^{-2k}\|f\|_{L^1_2}\|\mF_j\mP_kf\|^2_{L^2}2^j+2^{-2k}2^{-jN}\|f\|_{L^1_2}\|\mP_kf\|^2_{H^{-N}}+\sum_{\p\leq j+N_0}2^{-k}\|\mF_\p\mP_k f\|_{L^2}2^{-\f\p2}\\
	&&\times \|\mF_j\mP_kf\|^2_{L^2}2^j+\sum_{\p\geq j+N_0}2^{-k}\|\mF_\p\mP_kf\|^2_{L^2}\|\mF_j\mP_kf\|_{L^2}2^{\f j 2}+2^{-kN}\|f\|_{H^{-N}_{-N}}(\|\mF_j\mP_kf\|^2_{L^2}+2^{-jN}\|f\|^2_{L^2_{-N}}).
    \eeno 
	 Combining the above estimate and after some appropriate simplification, we obtain that
\ben\label{energylocalized}
	&&\notag\f{1}{2}\f{d}{dt}\LLT{\F_j\cP_kf}^2+\kappa\LLT{\F_j\cP_kf}^22^{-3k}2^{2j}\ls_N\sum_{\p\geq j+N_0}\LLT{\mF_\p\mP_kf}^2\LLT{\mF_j\mP_kf}2^{\f32j}+\sum_{\p\leq j+N_0}\LLT{\mF_\p\mP_kf}2^{\f{\p}{2}}\\
	&&\times\notag\LLT{\mF_j\mP_kf}^22^j
	+2^{-2k}\|f\|_{L^1_2}\LLT{\mF_j\mP_kf}^22^j+2^{-2k}2^{-jN}\|f\|_{L^1_2}\HHL{\mP_kf}^2+2^{-jN}2^{-kN}\|f\|_{L^2_{-N}}^2\HHLL{f}\\
	&&+\HHLL{f}\LLT{\mF_j\mP_kf}^2+\LLT{\mF_j\mP_kf}^22^{-3k}+2^{-jN}2^{-kN}\LLO{f}\HHLL{f}^2.
\een

We take \( l = 3 \) as a typical case and other cases $l>3$ can be proven similarly. To estimate the energy \(\|f\|_{ H^{-\frac{1}{2}, \sss}_3 }\), according to \eqref{Ber}, we should consider the sum \( \sum_{j,k} \|\F_j \mathcal{P}_k f\|_{L^2}^2 2^{-j} j^{2\sss} 2^{6k} \). We emphasize that, for the sake of rigor, we should actually multiply $j^{2\sss}2^{-j}2^{6k}/(1+\delta 2^k)^6(1+\delta j^{2\sss}), 0<\delta\ll 1$ on both side of \eqref{energylocalized} that makes the summation respect to $j,k$ valid because \( f(t) \in H^{-\f12} \) a.e. from \eqref{l1l3}. Since the proof that follows does not depend on the parameter \( \delta \), we can ultimately let \( \delta \rightarrow 0 \) to obtain the desired result, one may see the proof of Lemma 3.1 and Lemma A.6 in \cite{HJ}. To keep the paper concise, we will omit the parameter \( \delta \) here and in the subsequent sections.

Multiply $2^{-j}2^{2\sss}2^{6k}$ on both side of \eqref{energylocalized} and sum up with respect to $j,k$. Thanks to Lemma \ref{profileofHSk}\eqref{Ber}, it firstly holds that
	\ben\label{b001}
	\sum_{j,k}\LLT{\F_j\cP_kf}^2j^{2\sss}2^{-j}2^{6k}\sim\|f\|_{H^{-\f12,\sss}_{3}}^2,\quad \sum_{j,k}\LLT{\F_j\cP_kf}^2j^{2\sss}2^{2(-\f12+1)j}2^{2(3-\f{3}{2})k}\sim\|f\|_{H^{\f12,\sss}_{\f32}}^2.
	\een
Now we deal with the first term on the right hand side of \eqref{energylocalized}. We have that
\beno
&&\sum_{j,k}\sum_{\p\geq j+N_0}\LLT{\mF_\p\mP_kf}^2\LLT{\mF_j\mP_kf}2^{\f12j}j^{2\sss}2^{6k}\\
	&=&\sum_{\p,k}\sum_{j\leq \p-N_0}\LLT{\mF_j\mP_kf}2^{-\f12j}j^{\sss}\LLT{\mF_\p\mP_kf}^22^{\p}\p^{2\sss}2^{(j-\p)}\p^{-2\sss}j^{\sss}2^{6k}\\
    &=&\sum_{\p>M}\sum_{k,j\leq \p-N_0}+\sum_{\p\leq M}\sum_{k,j\leq \p-N_0}
    \leq C M^{-s}\|f\|_{H^{-\f12,\sss}_{3}}\|f\|_{H^{\f12,\sss}_{\f{3}{2}}}^2+C_M\|f\|_{H^{-\f12,\sss}_{3}}\|f\|_{L^{2}_{\f{3}{2}}}^2,
\eeno
where we can choose $M$ suitably large.

Similarly, for the second term, it holds that
\beno
	&&\sum_{j,k}\sum_{\p\leq j+N_0}\LLT{\mF_\p\mP_kf}2^{\f{\p}{2}}\LLT{\mF_j\mP_kf}^2j^{2\sss}2^{6k}\\
	&=&\sum_{j,k}\sum_{\p\leq j+N_0}\LLT{\mF_\p\mP_kf}\p^\sss2^{-\f\p 2}\LLT{\mF_j\mP_kf}^22^{j}j^{2\sss}2^{(\p-j)}\p^{-\sss}2^{6k}\\
	&\leq&\sum_{\p >M}\sum_{k\geq-1,j\geq \p-N_0}+\sum_{\p \leq M}\sum_{k\geq-1,j\geq \p-N_0}\leq CM^{-\sss}\|f\|_{H^{-\f12,\sss}_{3}}\|f\|_{H^{\f12,\sss}_{\f32}}^2+C_M\|f\|_{H^{-\f12,\sss}_{3}}\|f\|_{H^{0,\sss}_{\f32}}^2.
\eeno

For the third term, noticing that $\|f\|_{L^1_2}$ is bounded, then we have that
\beno
\sum_{j,k}\|f\|_{L^1_2}\LLT{\mF_j\mP_kf}^2j^{2
\sss}2^{4k}\ls \|f\|_{H_{2}^{0,\sss}}^2.
\eeno

The rest terms can be handled by the similar argument and we omit the details to obtain that
 \beno
 \mathrm{The~rest~term}\ls \|f\|_{H^{-\f12,\sss}_{3}}\|f\|_{H^{0,\sss}_{\f{3}{2}}}^2+\|f\|^2_{H^{-\f12,\sss}_{3}}+\|f\|_{H_{2}^{0,\sss}}^2.
 \eeno

By interpolation inequalities, we have that
\beno
\|f\|_{H^{0,\sss}_2}\ls_\sss \|f\|^{\f45}_{H^{\f12,\sss}_{\f32}}\|f\|^{\f15}_{L^1_7}\leq \eps \|f\|_{H^{\f12,\sss}_{\f32}}+C_\eps \|f\|_{L^1_7}.
\eeno

 Thus combining all above estimates and denote
 \beno
  X(t):=\sum_{j,k}2^{-j}j^{2s}2^{6k}\|\F_j\cP_kf\|^2_{L^2},\quad  Y(t):=\sum_{j,k}2^{j}j^{2s}2^{3k}\|\F_j\cP_kf\|^2_{L^2},
  \eeno
   we can derive that
 \ben\label{ODIE}
 X'(t)+(\kappa-(CM^{-\sss}+\eps C_M)X^{\f12}(t))Y(t)\leq C_{M,\eps}(1+X(t)).
 \een
Recall that $\sss>0$, then by choosing sufficiently large $M$ and small $\eps$ such that
\beno
(CM^{-\sss}+\eps C_M)X^{\f12}(0)\leq \f \kappa 4.
\eeno
 It can be easily deduced from \eqref{ODIE} that there exists $T>0$ depending on $X(0),M$ and $\eps$ such that 
\beno
\sup_{t\in[0,T]}X(t)+\int_0^TY(\tau)d\tau\ls X(0).
\eeno
Noticing that $X\sim_\sss \|f\|^2_{H^{-\f12,\sss}_3}$ and $Y(t)\sim_\sss \|f\|^2_{H^{-\f12,\sss}_{\f32}}$ by \eqref{b001}, we get that 
\ben\label{011}
f\in L^\infty([0,T];H^{-\f12,\sss}_3)\cap L^2([0,T];H^{\f12,\sss}_{\f32}).
\een

\underline{Step 3: Proof of the continuity.} Now, we will address the issue of \(f \in C([0,T]; H^{-\frac{1}{2}, \sss}_3)\). Since we already have that continuity of function $X(t)\sim_\sss \|f(t)\|^2_{H^{-\f12,\sss}_3}$ for $t\in[0,T]$, thus by Randon-Riesz Theorem, we only need to show the continuity of function $\<f(t),\phi\>$ with $\phi\in H^{\f12,-\sss}_{-3}$ which is the dual space of $H^{-\f12,\sss}_{3}$. 

Indeed, by the same argument as \eqref{010}, the continuity holds true for any $\phi\in C^\infty_c$. If $\phi\in H^{\f12,-\sss}_{-3}$, choose $\{\phi_n\}\subset C_c^\infty$ such that $\phi_n\rightarrow \phi$ in $H^{\f12,-\sss}_{-3}$, then 
\beno
\lim_{t\rightarrow s}(f(t)-f(s),\phi)&=&\lim_{t\rightarrow s}(f(t)-f(s),\phi_n)+\lim_{t\rightarrow s}(f(t)-f(s),\phi-\phi_n)\\
&\leq&\lim_{t\rightarrow s}(f(t)-f(s),\phi_n)+2 \sup_{t\in[0,T]}\|f(t)\|_{H^{-\f12,\sss}_3}\|\phi-\phi_n\|_{H^{\f12,-\sss}_{-3}}.
\eeno
Therefore from \eqref{011}, the above limit is $0$ if we take the sufficiently large \( n \) and \( t \) approaches \( s \). 

We end the proof of Theorem \ref{uniquethm}(1) for $\sss>0$.

\end{proof}
\medskip
Now we begin to prove Theorem \ref{T1}(1.1) and we first handle with the case $\mathsf{r}>-\f12$.
\begin{proof}[Proof of Theorem \ref{T1}(1.1) for $\mathsf{r}>-\f12$]
Recall that $f\in C([0,\infty);H^{\mathsf{r}}_l)\cap L^2_{loc}([0,\infty);H^{\mathsf{r}+1}_{l-\f32})$
with $\mathsf{r}\in (-\f12,0]$ and $l\geq \f94-\f32\mathsf{r}$, it follows that \eqref{005}  and \eqref{energylocalized} are satisfied. For $\nn\geq \mathsf{r}$ and $\ell \leq l-\f32 \nn+\f32\mathsf{r}$, 
we multiply $2^{2j\nn}2^{2k\ell}$ on both sides of \eqref{energylocalized} and sum up with respect to $j$ and $k$. By \eqref{Ber}, it holds that
	\[\ba
    \sum_{j,k}\LLT{\F_j\cP_kf}^22^{2j\mathrm{n}}2^{2k\ell}\sim_{\nn,\ell}\|f\|_{H^{\nn}_{\ell}}^2,\quad \sum_{j,k}\LLT{\F_j\cP_kf}^22^{2(\mathrm{n}+1)j}2^{2(\ell-\f{3}{2})k}\sim_{\nn,\ell}\|f\|_{H^{\nn+1}_{\ell-\f32}}^2.
	\ea\]
 Similar to the estimates in the previous proof, we begin with the first term on the right hand side of  \eqref{energylocalized}, it holds that
\beno
&&\sum_{j,k}\sum_{\p\geq j+N_0}\LLT{\mF_\p\mP_kf}^2\LLT{\mF_j\mP_kf}2^{(2\nn+\f32) j}2^{2\ell k}
	=\sum_{\p,k}\sum_{j\leq \p-N_0}\LLT{\mF_j\mP_kf}2^{lk+\mathsf{r}j}\\
	&&\times\LLT{\mF_\p\mP_kf}^22^{(2\nn+\f32-\mathsf{r})\p}2^{(2\ell-l)k}2^{(2\nn+\f32-\mathsf{r})(j-\p)}
    \ls\|f\|_{H^{\mathsf{r}}_{l}}\|f\|_{H^{\nn+\f{3}{4}-\f{\mathsf{r}}2}_{\ell-\f{l}{2}}}^2.
\eeno
 For the second term, we have that
\beno
	&&\sum_{j,k}\sum_{\p\leq j+N_0}\LLT{\mF_\p\mP_kf}2^{\f{\p}{2}}\LLT{\mF_j\mP_kf}^22^{(2\nn+1)j}2^{2k\ell}
	\ls\sum_{j,k}\sum_{\p\leq j+N_0}\LLT{\mF_\p\mP_kf}2^{lk+\mathsf{r}\p}\\
	&&\times\LLT{\mF_j\mP_kf}^22^{(2\nn+\f32-\mathsf{r})j}2^{(\p-j)(\f12-\mathsf{r})}2^{(2\ell-l)k}
    \ls\|f\|_{H^{\mathsf{r}}_{l}}\|f\|_{H^{\nn+\f{3}{4}-\f{\mathsf{r}}2}_{\ell-\f{l}{2}}}^2.
\eeno
For the rest terms, choose $N=2\nn+2|\ell|+3$, then by the similar argument, we can derive that
 \beno
 \mathrm{The~rest~term}\ls \|f\|_{H^{\mathsf{r}}_{l}}\|f\|_{H^{\nn+\f{3}{4}-\f{\mathsf{r}}2}_{\ell-\f{l}{2}}}^2+\|f\|^2_{H^{\nn+\f12}_{\ell-1}}+\|f\|^2_{H^{\nn}_{\ell}}.
 \eeno
 
By interpolation, we have that
\ben\label{interpo}
&&\|f\|_{H^{\nn+\f{3}{4}-\f{\mathsf{r}}2}_{\ell-\f{l}{2}}}\ls_{\nn,\ell} \|f\|^{\theta_1}_{H^{\mathsf{r}}_l}\|f\|^{1-\theta_1}_{H^{\nn+1}_{\ell-\f32}},~~\|f\|_{H^{\nn+\f12}_{\ell-1}}\ls_{\nn,\ell}\|f\|^{\theta_2}_{H^{\mathsf{r}}_l}\|f\|^{1-\theta_2}_{H^{\nn+1}_{\ell-\f32}},\\
&&\notag\|f\|_{H^{\nn}_{\ell}}\ls_{\nn,\ell}\|f\|^{\theta_3}_{H^{\mathsf{r}}_l}\|f\|^{1-\theta_3}_{H^{\nn+1}_{\ell-\f32}}\quad \mbox{with}\quad\theta_1=\f{\f14+\f{\mathsf{r}}2}{n+1-\mathsf{r}},~~\theta_2=\f{\f12}{n+1-\mathsf{r}},~~\theta_3=\f{1}{n+1-\mathsf{r}},
\een
where we use the fact that $\ell\leq l-\f32\nn+\f32\mathsf{r}$ and $l\geq \f94-\f32\mathsf{r}.$ Thus denote 
\ben\label{017}
X_{\nn,\ell}(t):=\sum_{j,k}2^{2\nn j}2^{2\ell k}\|\F_j\cP_kf\|^2_{L^2},\quad Y_{\nn,\ell}(t):=\sum_{j,k}2^{2(\nn+1) j}2^{(2\ell-3)k}\|\F_j\cP_kf\|^2_{L^2},
\een
we can derive that
\beno
X_{\nn,\ell}'(t)+(\kappa-\eps-\eps X^{\f12}_{\mathsf{r},l}(t))Y_{\nn,\ell}(t)\ls_\eps  X^{\f32}_{\mathsf{r},l}(t)+X_{\mathsf{r},l}(t),
\eeno
where $\eps$ can be sufficiently small. Since $f\in C([0,\infty);H^{\mathsf{r}}_l)$ and $X_{\mathsf{r},l}\sim_{\mathsf{r},l}\|f\|^2_{H^{\mathsf{r}}_l}$. For any $T>0$, by choosing $\eps<\f\kappa 4\min\Big\{1,\big(\sup_{t\in[0,T]}X_{\mathsf{r},l}\big)^{-\f12}\Big\}$, we can deduce  that
\beno
X_{\nn,\ell}'(t)+\f\kappa 4X_{\nn,\ell}(t)^{1+\f1{\nn-\mathsf{r}}}\leq C,
\eeno
where $C$ depends on $\sup_{t\in[0,T]}\|f\|_{H^{\mathsf{r}}_l}$. It leads to 
\ben\label{0001}
\|f(t)\|_{H^\nn_\ell}\sim_{\nn,\ell} X_{\nn,\ell}^{\f12}(t)\ls_{\nn,\ell} t^{-\f{\nn}{2}+\f{\mathsf{r}}2}+1,\quad t\in[0,T].
\een
Due to the arbitrariness of \( T \), we end the proof of Theorem \ref{T1}(1.1) for $\mathsf{r}>-\f12$.

\end{proof}

At the end of this section, we present the result concerning the propagation of regularity, which will be used for the proof in the critical space in Section \ref{section4}.
\begin{prop}\label{018}
For any $\nn\geq2, \ell\geq3$, if $\|f_0\|_{H^\nn_\ell}<\infty$, then there exists $T>0$ depending solely on $\nn,\ell$ and $\|f_0\|_{H^\nn_\ell}$ such that the weak solution to \eqref{1} with initial data $f_0$ satisfies
\beno
\sup_{t\in[0,T]}\|f(t)\|_{H^\nn_\ell}\ls_{\nn,\ell} 2\|f_0\|_{H^\nn_\ell}.
\eeno
\end{prop}
\begin{proof}
Starting again from \eqref{energylocalized}, we multiply $2^{2j\nn}2^{2k\ell}$ on both sides of \eqref{energylocalized} and sum up with respect to $j$ and $k$. For the first term on the right hand side of \eqref{energylocalized}, we have 
\beno
\sum_{j,k}\sum_{\p\geq j+N_0}\LLT{\mF_\p\mP_kf}^2\LLT{\mF_j\mP_kf}2^{(2\nn+\f32) j}2^{2\ell k}
\ls\|f\|^3_{H^{\nn}_{\ell}}.
\eeno
where use the fact that $n\geq2$. Similarly, for the second term, we have 
\beno
\sum_{j,k}\sum_{\p\leq j+N_0}\LLT{\mF_\p\mP_kf}2^{\f{\p}{2}}\LLT{\mF_j\mP_kf}^22^{(2\nn+1)j}2^{2k\ell}
\ls\|f\|^3_{H^{\nn}_{\ell}}.
\eeno

For the rest term, choose $N=2\nn+2|\ell|+3$, we can derive that
 \beno
\mathrm{The~rest~term}\ls \|f\|^2_{H^{\nn+\f12}_{\ell-1}}+\|f\|^2_{H^\nn_\ell}\leq \eps \|f\|^2_{H^{\nn+1}_{\ell-\f32}}+C_\eps\|f\|^2_{H^\nn_\ell}.
\eeno
where we use interpolation inequality $\|f\|_{H^{\nn+\f12}_{\ell-1}}\ls_{\nn,\ell}\|f\|^{\f12}_{H^{\nn+1}_{\ell-\f32}}\|f\|^{\f12}_{H^{\nn}_\ell}$ for $\ell\geq3$.
Therefore, by choosing $\eps$ suitably small, we can obtain that
\beno
X'_{\nn,\ell}(t)+\f \kappa 4 Y_{\nn,\ell}(t)\ls_{\nn,\ell} X^{\f32}_{\nn,\ell}(t),
\eeno
where $X_{\nn,\ell}$ and $Y_{\nn,\ell}$ are defined in \eqref{017}. It leads to 
\beno
X_{\nn,\ell}(t)\ls_{\nn,\ell}(X^{-\f12}_{\nn,\ell}(0)-C_{\nn,\ell}t)^{-2},
\eeno
which yields that $X_{\nn,\ell}(t)\ls_{\nn,\ell}2X_{\nn,\ell}(0)$ for $t\in[0,X_{\nn,\ell}^{-\f12}(0)/(10C_{\nn,\ell})]$.
Then we conclude the desired result.
\end{proof}

\section{Proof of Theorem \ref{uniquethm}(1) for \texorpdfstring{$\sss=0$}~~and Theorem \ref{T1} (1.1) for \texorpdfstring{$\mathsf{r}=-\f12$}.}\label{section4}
In this section, we are committed to proving the local existence and regularity in the critical space \( H^{-\frac{1}{2}}_l\) with \(l\geq3 \). Unlike the direct energy estimates in the previous section, we need to decompose the initial value into a regular part and an irregular but small part. 

 \begin{proof}[Proof of Theorem \ref{uniquethm}(1) for $\sss=0$]
We also prove it for $l=3$ and the case $l>3$ can be handled by the same method.   Let $J_\delta(v)=\f{1}{\delta^3}J(\f{v}{\delta})$ be the Fedrich mollifier, that is $J(v)\in C^\infty_c(\R^3)$ is a nonnegative function satisfying $\int_{\R^3} J(v)dv=1$. Choose $\delta$ sufficiently small such that $\|f_0-J_\delta*f_0\|_{H^{-\f12}_3\cap L^1_7}<\eps_1$, where $\eps_1$  will be determined later(see \eqref{019}). Then $J_\delta*f_0$ is a nonnegative function that belongs to $L^1_{7}\cap H^3_3$. 
 Thus by the previous local well-posedness argument, Proposition \ref{018} and Proposition 2 in \cite{D}, there exists $T_1>0$ and a solution 
\ben\label{0002}
 f_1\in L^\infty([0,T_1],L^1_7\cap H^3_3)\cap C([0,T_1],H^{-\f12}_3)\cap L^2([0,T_1],H^{\f12}_{\f32})
\een
  of \eqref{1} with initial data $f_1(0)=J_\delta*f_0$.  Now let $f_2=f-f_1$, where $f$ is the weak solution of Landau equation \eqref{1} with initial data $f(0)=f_0$,  then the following equation holds in weak sense:
    \beno
    \pa_tf_2=Q(f,f_2)+Q(f_2,f_1),\quad f_2(0)=f_0-J_\delta*f_0,~~t\in[0,T_1].\eeno
    We now prove that there exists $0<T_2\leq T_1$ such that $\|f_2(t)\|_{H^{-\f12}_3}\ls 2\eps_1$ for any $t\in[0,T_2]$.
     
     By the same energy method as the Step 1 of the proof for $\sss>0$ , it holds that
    \[\ba
    \f12\f{d}{dt}\LLT{\F_j\cP_kf_2}^2+\int_{\R^3}(a*f):\na\F_j\cP_kf_2\mult\na\F_j\cP_kf_2dv=\mathfrak{G}_1+\mathfrak{G}_2,
    \ea\]
    where 
  \[\ba
    \mathfrak{G}_1&:=\int_{\R^3} f_1|\F_j\cP_kf_2|^2dv+\int_{\R^3} f_2|\F_j\cP_kf_2|^2dv;\\
    \mathfrak{G}_2&:=\br{\F_j\cP_kQ(f_1,f_2)-Q(f_1,\F_j\cP_kf_2),\F_j\cP_kf_2}
    +\br{\F_j\cP_kQ(f_2,f_2)\\
    	&-Q(f_2,\F_j\cP_kf_2),\F_j\cP_kf_2}
    +\br{\F_j\cP_kQ(f_2,f_1),\F_j\cP_kf_2}:=\mathfrak{G}^1_2+\mathfrak{G}^2_2+\mathfrak{G}^3_2.
    \ea\]
    On one hand, due to \eqref{gFjkhf}, we have that
    \ben\label{RR122}
    |\mathfrak{G}_1|\notag&\ls_N& \sum_{\p\leq j}\LLT{\mF_\p\mP_kf_1}2^{\f32\p}\LLT{\mF_j\mP_kf_2}^2+\sum_{\p\leq j}\LLT{\mF_\p\mP_kf_2}2^{\f32\p}\LLT{\mF_j\mP_kf_2}^2\\
    &&+(\|f_1\|_{H^{-N}_{-N}}+\|f_2\|_{H^{-N}_{-N}})\|\mF_j\mP_k f_2\|^2_{L^2}+2^{-jN}2^{-kN}(\|f_1\|_{H^{-N}_{-N}}+\|f_2\|_{H^{-N}_{-N}})\|f_2\|^2_{H^{-N}_{-N}}.
    \een
    On the other hand, thanks to the decomposition \eqref{decomofcom} together with the estimates \eqref{FjQ}\eqref{FjQ-1}\eqref{PkQ} and \eqref{PkQ-1}, we have
    \ben\label{Com122}
    &&\notag|\mathfrak{G}^1_2|\ls_N\sum_{\p\geq j+N_0}\LLT{\mF_\p\mP_kf_1}\LLT{\mF_\p\mP_kf_2}\LLT{\mF_j\mP_kf_2}2^{\f32j}+\sum_{\p\leq j+N_0}\LLT{\mF_\p\mP_kf_1}2^\f\p2\LLT{\mF_j\mP_kf_2}^22^j\\
    &+&\notag\sum_{\p\leq j+N_0}\LLT{\mF_\p\mP_kf_2}\LLT{\mF_j\mP_kf_1}\LLT{\mF_j\mP_kf_2}2^{\f32\p}
   +2^{-jN}2^{-kN}\|f_1\|_{L^2_{-N}}\|f_2\|_{L^2_{-N}}\HHLL{f_2}
    +2^{-2k}\\
    &&\times\|f_1\|_{L^1_2}\LLT{\mF_j\mP_kf_2}^22^j+2^{-2k}2^{-jN}\|f_1\|_{L^1_2}\HHL{\mP_kf_2}^2+2^{-jN}2^{-kN}\|f_1\|_{L^1_2}\HHLL{f_2}^2.
    \een
    and
    \ben\label{Com222}
   &&\notag|\mathfrak{G}^2_2|\ls_N\sum_{\p\geq j+N_0}\LLT{\mF_\p\mP_kf_2}^2\LLT{\mF_j\mP_kf_2}2^{\f32j}
    +\sum_{\p\leq j+N_0}\LLT{\mF_\p\mP_kf_2}2^{\f{\p}{2}}\LLT{\mF_j\mP_kf_2}^22^j\\	&+&\notag2^{-2k}\|f_2\|_{L^1_2}\LLT{\mF_j\mP_kf_2}^22^j
    +2^{-2k}2^{-jN}\|f_2\|_{L^1_2}\HHL{\mP_kf_2}^2+2^{-kN}\LLT{\mF_j\mP_kf_2}^2\HHLL{f_2}\\
    &+&2^{-jN}2^{-kN}\|f_2\|_{L^2_{-N}}^2\HHLL{f_2}.\een
   
    Now we focus on the estimate of $\mathfrak{G}^3_2$. We split it into two parts: 
    \[\mathfrak{G}^3_2=\sum_{\l\geq0}\br{Q_\l(f_2,f_1),\cP_k\F_j^2\cP_kf_2}+\br{Q_{-1}(f_2,f_1),\cP_k\F_j^2\cP_kf_2}:=\mathfrak{H}_1+\mathfrak{H}_2.\]
    
    \underline{Estimate of $\mathfrak{H}_1$.} First we perform dyadic decomposition in frequency space. Thanks to decomposition \eqref{Qkdecom} together with the estimates \eqref{M1}\eqref{M4}\eqref{M2} and \eqref{M3}, we can deduce that
    \ben\label{asin}
    &&\br{Q_\l(f_2,f_1),\cP_k\F_j^2\cP_kf_2}=\sum_{\a\leq\p-N_0}\br{Q_\l(\F_\p f_2,\F_\a f_1),\tF_\p\cP_k\F_j^2\cP_k f_2}+\sum_{\a\geq-1}\br{Q_\l(\tS_{\a-N_0} f_2,\F_\a f_1),\\
    	&&\notag\tF_\a\cP_k\F_j^2\cP_k f_2}
    +\sum_{\m\leq\p+N_0}\br{Q_\l(\F_\p f_2,\tF_\p f_1),\F_\m\cP_k\F_j^2\cP_kf_2}
    \ls_N\sum_{\p\geq-1}2^{-\l}\LLO{f_2}\LLT{\tF_\p f_1}\LLT{\tF_\p\cP_k\F_j^2\cP_kf_2}\\
    &&\times2^{2\p}\notag
    +\sum_{\m\leq\p+N_0}2^{-\l N}2^{-2\p N}2^{-2\m N}\LLO{f_2}
    \rr{\LLT{\F_\m f_1}\LLT{\tF_\p\cP_k\F_j^2\cP_k f_2}+\LLT{\tF_\p f_1}\LLT{\F_\m\cP_k\F_j^2\cP_kf_2}}.    
    \een
    Then we perform dyadic decomposition in phase space(see \eqref{phasedecom}) to get that
    \ben\label{asin2}
    \mathfrak{H}_1&=&\notag\sum_{0\leq\l\leq k-N_0}\br{Q_\l(\tP_kf_2,\tP_kf_1),\cP_k\F_j^2\cP_kf_2}+\sum_{\l\geq k+N_0}\br{Q_\l(\tP_\l f_2,\tP_kf_1),\cP_k\F_j^2\cP_kf_2}\notag\\
    &&+\sum_{\l\geq0,|k-\l|< N_0}\br{Q_\l(\U_{k+2N_0}f_2,\tP_kf_1),\cP_k\F_j^2\cP_kf_2}.
    \een
    Thus by \eqref{asin}, we have that
    \beno
      &&|\mathfrak{H}_1|
    \ls_N\sum_{\p\geq-1}2^{-k}\|f_2\|_{L^1_1}\LLT{\tF_\p\tP_kf_1}\LLT{\tF_\p\cP_k\F_j^2\cP_kf_2}2^{2\p}+\sum_{\m\leq\p+N_0}2^{-2\p N}2^{-2\m N}2^{-k}\\
    &&\times\|f_2\|_{L^1_1}\rr{\LLT{\F_\m\tP_kf_1}\LLT{\tF_\p\cP_k\F_j^2\cP_kf_2}+\LLT{\tF_\p\tP_kf_1}\LLT{\F_\m\cP_k\F_j^2\cP_kf_2}}.
    \eeno
   Futhermore, due to Proposition \ref{PFCom}, we can derive that
    \ben\label{lge0212}
    &&|\mathfrak{H}_1|
    \ls\sum_{|\p-j|\leq N_0}2^{-k}\|f_2\|_{L^1_1}\LLT{\tF_\p\tP_kf_1}\LLT{\F_j^2\cP_kf_2}2^{2\p}+\sum_{|\p-j|> N_0}2^{-k}\|f_2\|_{L^1_1}\LLT{\tF_\p\tP_kf_1}\LLT{\tF_\p\cP_k\F_j^2\cP_kf_2}\notag\\
    &&\times2^{2\p}+\notag2^{-k}\Big(\sum_{|\p-j|\leq N_0}\sum_{\m\leq\p+N_0}2^{-2\p N}2^{-2\m N}\|f_2\|_{L^1_1}\LLT{\F_\m\tP_kf_1}\LLT{\F_j^2\cP_kf_2}+\sum_{|\p-j|> N_0}\sum_{\m\leq\p+N_0}2^{-2\p N}2^{-2\m N}
   \\
    &&\notag\times \|f_2\|_{L^1_1}\LLT{\F_\m\tP_kf_1}\LLT{\tF_\p\cP_k\F_j^2\cP_kf_2}+\sum_{|\m-j|\leq N_0}\sum_{\p\geq\m-N_0}2^{-2\p N}2^{-2\m N}\|f_2\|_{L^1_1}\LLT{\tF_\p\tP_kf_1}\LLT{\F_j^2\cP_kf_2}\\
    &+&\notag\sum_{|\m-j|> N_0}\sum_{\p\geq\m-N_0}2^{-2\p N}2^{-2\m N}\|f_2\|_{L^1_1}\LLT{\tF_\p\tP_kf_1}\LLT{\F_\m\cP_k\F_j^2\cP_kf_2}\Big)\ls_N2^{-k}\|f_2\|_{L^1_1}\LLT{\mF_j\mP_kf_1}\LLT{\mF_j\mP_kf_2}\\
    &&\times2^{2j}+2^{-k}2^{-jN}\|f_2\|_{L^1_1}\HHL{\mP_kf_1}\HHL{\mP_kf_2}+2^{-jN}2^{-kN}\|f_2\|_{L^1_1}\HHLL{f_1}\HHLL{f_2}.
    \een
    \underline{Estimate of $\mathfrak{H}_2$.} First we perform dyadic decomposition in frequency space. Thanks to decomposition \eqref{Qkdecom} together with the estimates \eqref{M1s}\eqref{M4s}\eqref{M2s} and \eqref{M3s}, we have
    \ben\label{f12pk}
    &&|\mathfrak{H}_2|\leq\sum_{\a\leq\p-N_0}|\br{Q_{-1}(\F_\p f_2,\F_\a f_1),\tF_\p\cP_k\F_j^2\cP_k f_2}|+\sum_{\a\geq-1}|\br{Q_{-1}(\tS_{\a-N_0} f_2,\F_\a f_1),\tF_\a\cP_k\F_j^2\cP_k f_2}|\\
    &&\notag+\sum_{\m\leq\p+N_0}|\br{Q_{-1}(\F_\p f_2,\tF_\p f_1),\F_\m\cP_k\F_j^2\cP_kf_2}|\ls\sum_{\a\leq\p-N_0}\LLT{\F_\p f_2}\LLT{\tF_\p\cP_k\F_j^2\cP_k f_2}\LLT{\F_\a f_1}2^{\f32\a}\\
    	&+&\notag\sum_{\p\leq\a-N_0}\LLT{\F_\p f_2}2^{-\f\p2}\LLT{\F_\a f_1}\LLT{\tF_\a\cP_k\F_j^2\cP_k f_2}2^{2\a}+\sum_{\m\leq\p+N_0}2^{\f32\m}\LLT{\F_\p f_2}\LLT{\tF_\p f_1}\LLT{\F_\m\cP_k\F_j^2\cP_kf_2}.
    \een
    Then we perform the decomposition in phase space. Note that $\<v\>\sim\<v_*\>$ for $Q_{-1}$,  thus by \eqref{f12pk}, we obtain that
    \beno
    &&|\mathfrak{H}_2|\ls \sum_{\a\leq\p-N_0}\LLT{\F_\p\tP_k f_2}\LLT{\tF_\p\cP_k\F_j^2\cP_k f_2}\LLT{\F_\a\tP_kf_1}2^{\f32\a}
    +\sum_{\p\leq\a-N_0}\LLT{\F_\p\tP_kf_2}2^{-\f\p2}\LLT{\F_\a\tP_k f_1}\\
    &&\times\LLT{\tF_\a\cP_k\F_j^2\cP_k f_2}2^{2\a}
    +\sum_{\m\leq\p+N_0}\LLT{\F_\p\tP_k f_2}\LLT{\tF_\p\tP_kf_1}\LLT{\F_\m\cP_k\F_j^2\cP_kf_2}2^{\f32\m}.
    \eeno
   Futhermore, due to Proposition \ref{PFCom} and similar argument as \eqref{lge0212}, we can derive that
   \ben\label{l0212}
    &&|\mathfrak{H}_2|\ls_N \sum_{\a\leq j}\LLT{\mF_\a\mP_kf_1}2^{\f32\a}\LLT{\mF_j\mP_kf_2}^2+\sum_{\p\leq j}\LLT{\mF_\p\mP_kf_2}2^{-\f\p2}
    \LLT{\mF_j\mP_kf_1}\\
    &&\times\notag\LLT{\mF_j\mP_kf_2}2^{2j}+\sum_{\p\geq j}\LLT{\mF_\p\mP_kf_2}\LLT{\mF_\p\mP_kf_1}\LLT{\mF_j\mP_kf_2}2^{\f32j}
    +2^{-jN}2^{-kN}\\
    &&\times\notag\HHLL{f_1}\HHLL{f_2}^2+2^{-jN}2^{-kN}\|\mP_kf_2\|_{L^2_{-N}}\|\mP_kf_1\|_{L^2_{-N}}\HHLL{f_2}.
   \een
    Combining the previous estimate \eqref{RR122}\eqref{Com122}\eqref{Com222}\eqref{lge0212} and \eqref{l0212} together with Proposition \ref{coer}, the energy equation comes to 
    \ben\label{f2evo}
    \f12\f{d}{dt}\LLT{\F_j\cP_kf_2}^2+\kappa\LLT{\F_j\cP_kf_2}^22^{2j}2^{-3k}+\kappa\LLT{(-\Delta_{\SS^2})^\f12\F_j\cP_kf_2}^22^{-3k}\ls_N \mathfrak{I}_1+\mathfrak{I}_2,
    \een
    where
    \ben\label{i1}
    &&\notag\mathfrak{I}_1=\sum_{\p\geq j}\LLT{\mF_\p\mP_kf_2}^2\LLT{\mF_j\mP_kf_2}2^{\f32j}\notag
    +\sum_{\p\leq j+N_0}\LLT{\mF_\p\mP_kf_2}2^{\f{\p}{2}}\LLT{\mF_j\mP_kf_2}^22^j\\
    &&\notag+\sum_{\p\geq j}\LLT{\mF_\p\mP_kf_1}\LLT{\mF_\p\mP_kf_2}\LLT{\mF_j\mP_kf_2}2^{\f32j}\notag
    +\sum_{\p\leq j+N_0}\LLT{\mF_\p\mP_kf_1}2^\f\p2\LLT{\mF_j\mP_kf_2}^22^j\\
    &&+\sum_{\p\leq j+N_0}\LLT{\mF_\p\mP_kf_2}2^{-\f\p2}\LLT{\mF_j\mP_kf_1}\LLT{\mF_j\mP_kf_2}2^{2j}
    \een
    and
    \ben\label{i2}
    \mathfrak{I}_2&=&(\|f_1\|_{H^{-N}_{-N}}+\|f_2\|_{H^{-N}_{-N}})\|\mF_j\mP_k f_2\|^2_{L^2}+2^{-k}\|f_2\|_{L^1_1}\LLT{\mF_j\mP_kf_1}\LLT{\mF_j\mP_kf_2}2^{2j}\\
    &+&\notag2^{-2k}\|f_1\|_{L^1_2}\LLT{\mF_j\mP_kf_2}^22^j+2^{-2k}\|f_2\|_{L^1_2}\LLT{\mF_j\mP_kf_2}^22^j+2^{-kN}\LLT{\mF_j\mP_kf_2}^2\HHLL{f_2}\\
    &+&\notag2^{-k}2^{-jN}\|f_2\|_{L^1_1}\HHL{\mP_kf_1}\HHL{\mP_kf_2}+2^{-2k}2^{-jN}\big(\|f_1\|_{L^1_2}\HHL{\mP_kf_2}^2+\|f_2\|_{L^1_2}\HHL{\mP_kf_2}^2\big)\\
    &+&\notag2^{-jN}2^{-kN}\big(\|f_1\|_{L^2_{-N}}\|f_2\|_{L^2_{-N}}+\|f_1\|_{L^1_2}\HHLL{f_2}+\|f_2\|_{L^2_{-N}}^2+\|f_2\|_{L^1_1}\HHLL{f_1}\big)\HHLL{f_2}.
   \een
   
    Now we choose $N=10$ and multiply $2^{-j}2^{6k}$  on both sides of \eqref{f2evo} and sum up with respect to $j$ and $k$. For the first term and the second term of $\mathfrak{I}_1$, we have
    \[\ba
    &\sum_{j,k}\sum_{\p\geq j}\LLT{\mF_\p\mP_kf_2}^2\LLT{\mF_j\mP_kf_2}2^{\f j2}2^{6k}+\sum_{j,k}\sum_{\p\leq j+N_0}\LLT{\mF_\p\mP_kf_2}2^{\f{\p}{2}}\LLT{\mF_j\mP_kf_2}^22^{6k}\\
    \ls&\sum_{j,k}\sum_{\p\geq j}\LLT{\mF_\p\mP_kf_2}^22^\p\LLT{\mF_j\mP_kf_2}2^{-\f12j}2^{j-\p}2^{6k}\ls\|f_2\|_{H^{-\f12}_3}\|f_2\|_{H^{\f12}_{\f32}}^2.
    \ea\]
    For the third and fifth terms, we have
    \[\ba
    &\sum_{j,k}\sum_{\p\geq j}\LLT{\mF_\p\mP_kf_1}\LLT{\mF_\p\mP_kf_2}\LLT{\mF_j\mP_kf_2}2^{\f12j}2^{6k}+\sum_{j,k}\sum_{\p\leq j+N_0}\LLT{\mF_\p\mP_kf_2}2^{-\f\p2}\LLT{\mF_j\mP_kf_1}\\
    &\times\LLT{\mF_j\mP_kf_2}2^{j}2^{6k}
    \ls\sum_{j,k}\sum_{\p\geq j}\LLT{\mF_\p\mP_kf_1}2^{\f32\p}\LLT{\mF_\p\mP_kf_2}2^{-\f12\p}\LLT{\mF_j\mP_kf_2}2^{-\f12j}2^{j-\p}2^{6k}\\
    &+\sum_{j,k}\sum_{\p\leq j+N_0}\LLT{\mF_\p\mP_kf_2}2^{-\p}\LLT{\mF_j\mP_kf_1}2^{2j}\LLT{\mF_j\mP_kf_2}2^{-\f j2}2^{6k}    \ls\|f_1\|_{H^3_3}\|f_2\|_{H^{-\f12}_3}^2.
    \ea\]
    For the fourth term, use Cauchy inequality, we have
    \[\ba
    &\sum_{j,k}\sum_{\p\leq j+N_0}\LLT{\mF_\p\mP_kf_1}2^\f\p2\LLT{\mF_j\mP_kf_2}^22^{6k}=\sum_{j,k}\sum_{\p\leq j+N_0}\LLT{\mF_\p\mP_kf_1}2^\f\p22^{-\f j2}\LLT{\mF_j\mP_kf_2}^22^{\f j2}2^{6k}\\
    \ls&\|f_1\|_{H^3_3}\|f_2\|_{H^{\f12}_{\f32}}\|f_2\|_{H^{-\f12}_3}\ls\eps\|f_2\|_{H^{\f12}_{\f32}}^2+C_\eps\|f_1\|_{H^{3}_{3}}^2\|f_2\|_{H^{-\f12}_3}^2.
    \ea\]
    Thus we obtain that
    \beno
    \sum_{i,k}(2^{-j}2^{6k}\mathfrak{I}_1)\ls \eps \|f_2\|^2_{H^{\f12}_{\f32}}+\|f_2\|_{H^{-\f12}_3}\|f_2\|_{H^{\f12}_{\f32}}^2+\|f_1\|_{H^3_3}\|f_2\|_{H^{-\f12}_3}^2+C_\eps\|f_1\|_{H^{3}_{3}}^2\|f_2\|_{H^{-\f12}_3}^2.
    \eeno
    
   We remark that $\mathfrak{I}_2$ can be handled by the similar manner. Noticing that $\|f_1\|_{H^{-N}_{-N}}+\|f_2\|_{H^{-N}_{-N}}+\|f_2\|_{L^1_1}+\|f_1\|_{L^1_2}\leq C$, we omit the details to get that
   \beno
   \sum_{i,k}(2^{-j}2^{6k}\mathfrak{I}_2)\ls \|f_2\|^2_{H^{-\f12}_3}+\|f_1\|_{H^3_3}\|f_2\|_{H^{-\f12}_3}+\|f_2\|^2_{L^2_2}\ls\eps\|f\|^2_{H^{\f12}_{\f32}}+C_\eps\|f_2\|^2_{H^{-\f12}_3}+\|f_1\|_{H^3_3}\|f_2\|_{H^{-\f12}_3}.
   \eeno
   
    Let $X(t):=\sum_{j,k}\LLT{\F_j\cP_kf_2}^22^{-j}2^{6k}, Y(t):=\sum_{j,k}\LLT{\F_j\cP_kf_2}^22^{j}2^{3k}$. Combining  the estimates of $\mathfrak{I}_1$ and $\mathfrak{I}_2$, we can get that
    \ben\label{019}
    X'(t)+\rr{\kappa-\eps -C_\eps X^\f12}Y(t)\leq C\big(\sup_{t\in[0,T_1]}\|f_1(t)\|_{H^3_3}\big)\rr{X+1}.
    \een
 Note that the constants on the left hand side of \eqref{019} does not depend on $\sup_{t\in[0,T_1]}\|f_1(t)\|_{H^3_3}$. Then we choose $\eps_1$ such that $X(0)\ls \eps^2_1$ is suitably small, by the continuity argument,  there exists $0<T_2<T_1$ such that 
    \beno
    \sup_{t\in[0,T_2]}X(t)+\int_0^{T_2}Y(\tau)d\tau \ls 2\eps_1^2. 
    \eeno
Moreover,  $T_2$ only depends on $\kappa, T_1,\eps_1$ and $\sup_{t\in[0,T_1]}\|f_1(t)\|_{H^3_3}$.

Therefore, we have 
 $f_2\in L^\infty ([0,T_2],H^{-\f12}_3)\cap L^2([0,T_2],H^{\f12}_{\f32}).$
     Following the same approach as outlined in Step 3 of the proof for
Theorem \ref{uniquethm}(1) for $\sss>0$, we can obtain that $f_2\in C([0,T_2],H^{-\f12}_3)$, 
together with \eqref{0002}, it holds that $f=f_1+f_2\in C([0,T_2],H^{-\f12}_3)\cap L^2([0,T_2],H^{\f12}_{\f32})$ and we end the proof.
\end{proof}

Next, we will establish the regularity estimates in the critical space $H^{-\f12}_l,l\geq3$. 
\begin{proof}[Proof of Theorem \ref{T1}(1.1) for $\mathsf{r}=-\f12$] To prove the desired result, we shall prove that for any fixed $\nn\geq-\f12$ and $T>0$, it holds that
\ben\label{HnlE}\|f(t)\|_{H^\nn_{l-\f{3}{2}\nn-\f{3}{4}}(\R^3)}\ls t^{-\f{\nn}{2}+\f{\mathsf{r}}{2}}\mathbf{1}_{t\leq1}+C(t)\mathbf{1}_{t>1},\quad \forall t\in(0,T],\een 
where  $C(t)$ is a bounded function. The proof is split into two parts.

\underline{Step 1: Validity of \eqref{HnlE} over short time inverval.}
	Recall that $f\in C([0,T],H^{-\f12}_l)\cap L^2([0,T],H^{\f12}_{l-\f32}),l\geq3$, in particular, $f(0)\in H^{-\f12}_l$. Let $f_1$ be the smooth  solution of \eqref{1} with initial data $J_\delta*f(0)$ such that $\|f(0)-J_\delta*f(0)\|_{H^{-\f12}_3}<\eps_1$, in which \( \eps_1 \) is a certain small constant depending on $\nn,\ell$ and $\kappa$(see\eqref{019}  and \eqref{015}).  The parameter $\delta$, apparently, depends on $\eps_1$ and function $f(0)$.

  Since it holds that $\|f_1(0)\|_{H^m_{l}}\ls C_{m,l,\delta} \|f(0)\|_{H^{-\f12}_l}$ for any $m\geq2$, due to Proposition \ref{018}, we can obtain that there exists  $T_1>0$  that only depends on $\nn,l,\delta$ and $\|f(0)\|_{H^{-\f12}_l}$ such that
\ben\label{020}
\sup_{t\in[0,T_1]}\|f_1(t)\|_{H^{3}_l}\leq C_{l,\delta} \|f(0)\|_{H^{-\f12}_l},\\
\label{f1nl}
	\sup_{t\in[0,T_1]}\|f_1(t)\|_{H^{\nn+2}_l}\leq C_{\nn,l,\delta} \|f(0)\|_{H^{-\f12}_l}.
	\een
We address that \(\nn \) is fixed at the beginning of the proof. 

By \eqref{019} and \eqref{020}, there exists $T_2>0$ that only depending on $\nn,\kappa, l,\delta, \eps_1$ and $\|f(0)\|_{H^{-\f12}_l}$ such that $f_2:=f-f_1\in C([0,T_2],H^{-\f12}_3)$ and
\ben\label{014}
\sup_{t\in[0,T_2]}\|f_2(t)\|_{H^{-\f12}_l}\ls 2\eps_1.
\een
 
Now define $T_0:=\min\{T_1,T_2\}$. Because of \eqref{f1nl},  \eqref{HnlE} is reduced to prove that
\ben\label{004}
\|f_2(t)\|_{H^\nn_{l-\f32\nn-\f34}}\ls t^{-\f{\nn}{2}-\f14}+1,\quad \forall t\in(0,T_0].
\een
We emphasize again that   $T_0$ depends on $\nn,\kappa,l,\delta,\eps_1$ and $\|f(0)\|_{H^{-\f12}_l}$.

For $\nn\geq-\f12$ and $ \ell\leq l-\f{3}{2}\nn-\f34$, we choose $N=2\nn+2|\ell|+3$ and multiply  $2^{2j\mathrm{n}}2^{2k\ell}$ on both sides of \eqref{f2evo}  and sum up with respect to $j$ and $k$. For the first term and the second term of $\mathfrak{I}_1$, we have
    \[\ba
    &\sum_{j,k}\sum_{\p\geq j}\LLT{\mF_\p\mP_kf_2}^2\LLT{\mF_j\mP_kf_2}2^{(2\nn+\f32)j}2^{2k\ell}+\sum_{j,k}\sum_{\p\leq j+N_0}\LLT{\mF_\p\mP_kf_2}2^{\f{\p}{2}}\LLT{\mF_j\mP_kf_2}^2\\
    &\times2^{(2\nn+1)j}2^{2k\ell}
    =\sum_{j,k}\sum_{\p\geq j}\LLT{\mF_\p\mP_kf_2}^22^{(2\nn+2)\p}\LLT{\mF_j\mP_kf_2}2^{-\f12j}2^{2k\ell}2^{(2\nn+2)(j-\p)}\\&+\sum_{j,k}\sum_{\p\leq j+N_0}\LLT{\mF_\p\mP_kf_2}2^{-\f{\p}{2}}\LLT{\mF_j\mP_kf_2}^22^{(2\nn+2)j}2^{\p-j}2^{2k\ell}\ls\|f_2\|_{H^{-\f12}_3}\|f_2\|_{H^{\nn+1}_{\ell-\f32}}^2.
    \ea\]
    For the third term and the fifth term, we have
    \[\ba
    &\sum_{j,k}\sum_{\p\geq j}\LLT{\mF_\p\mP_kf_1}\LLT{\mF_\p\mP_kf_2}\LLT{\mF_j\mP_kf_2}2^{(2\nn+\f32)j}2^{2k\ell}+\sum_{j,k}\sum_{\p\leq j+N_0}\LLT{\mF_\p\mP_kf_2}2^{-\f\p2}\LLT{\mF_j\mP_kf_1}\\
    &\times\LLT{\mF_j\mP_kf_2}2^{(2\nn+2)j}2^{2k\ell}
    =\sum_{j,k}\sum_{\p\geq j}\LLT{\mF_\p\mP_kf_1}2^{(\nn+1)\p}\LLT{\mF_\p\mP_kf_2}2^{(\nn+1)\p}\LLT{\mF_j\mP_kf_2}2^{-\f12j}2^{(2\nn+2)(j-\p)}\\
    &\times2^{2k\ell}
    +\sum_{j,k}\sum_{\p\leq j+N_0}\LLT{\mF_\p\mP_kf_2}2^{-\f\p2}\LLT{\mF_j\mP_kf_1}\LLT{\mF_j\mP_kf_2}2^{(2\nn+2)j}2^{2k\ell}\ls\|f_2\|_{H^{-\f12}_3}\|f_1\|_{H^{\nn+2}_\ell}\|f_2\|_{H^{\nn+1}_{\ell-\f32}}\\
    &\ls\eps\|f_2\|_{H^{\nn+1}_{\ell-\f32}}^2+C_\eps\|f_2\|_{H^{-\f12}_3}^2\|f_1\|_{H^{\nn+2}_\ell}^2.
    \ea\]
    For the fourth term, we have
    \[\sum_{j,k}\sum_{\p\leq j+N_0}\LLT{\mF_\p\mP_kf_1}2^\f\p2\LLT{\mF_j\mP_kf_2}^22^{(2\nn+1)j}2^{2k\ell}\ls\|f_1\|_{H^1_l}\|f_2\|_{H^{\nn+\f12}_{\ell-\f l2}}^2\ls\eps\|f_2\|_{H^{\nn+1}_{\ell-\f32}}^2+C_\eps\|f_1\|_{H^1_l}^{2\nn+3}\|f_2\|_{H^{-\f12}_3}^2,\]
    where we use the interpolation inequality and Young inequality.
    Thus we conclude that
    \beno
    \sum_{j,k}(2^{2j\mathrm{n}}2^{2k\ell}\mathfrak{I}_1)\ls (\|f_2\|_{H^{-\f12}_3}+\eps)\|f_2\|_{H^{\nn+1}_{\ell-\f32}}^2+C_\eps\|f_2\|_{H^{-\f12}_3}^2(\|f_1\|_{H^{\nn+2}_\ell}^2+\|f_1\|_{H^1_l}^{2\nn+3}).
    \eeno
    The term $\mathfrak{I}_2$ can be estimated in a similar way and we can derive that
    \beno
    \sum_{j,k}(2^{2j\mathrm{n}}2^{2k\ell}\mathfrak{I}_2)\ls \|f_2\|^2_{H^\nn_\ell}+\|f_1\|^2_{H^{\nn+2}_l}+\|f_2\|^2_{H^{\nn+\f12}_{\ell-1}}+\|f_2\|^2_{L^2_{-N}}.
    \eeno
    
    Now let 
     \beno
     X_{\nn,\ell}(t):=\sum_{j,k}\LLT{\F_j\cP_kf_2}^22^{2j\mathrm{n}}2^{2k\ell},~~ Y_{\nn,\ell}(t):=\sum_{j,k}\big(\LLT{\F_j\cP_kf_2}^22^{(2\mathrm{n}+2)j}2^{(2\ell-3)k}+\LLT{(-\Delta_{\SS^2})^\f12\F_j\cP_kf_2}^22^{(2\ell-3)k}\big).
     \eeno Then by \eqref{interpo} and interpolation inequality $\|f_2\|_{L^2_{-N}}\leq \eps \|f_2\|_{H^{n+1}_{\ell-\f32}}+C_\eps\|f_2\|_{H^{-\f12}_3}$, 
      we can deduce that
     \ben\label{015}
     X_{\nn,\ell}'(t)+(\kappa-C_{\nn,\ell}\|f_2\|_{H^{-\f12}_3}-\eps)Y_{\nn,\ell}(t)\ls C(\sup_{t\in[0,T_0]}\|f_1\|_{H^{\nn+2}_\ell})(\|f_2\|^2_{H^{-\f12}_3}+1).
     \een
  From \eqref{014}, we have that $\sup_{t\in[0,T_0]}\|f_2\|_{H^{-\f12}_3}\ls 2\eps_1$   with $\eps_1$ sufficiently small, and \eqref{f1nl} holds, thus we can choose small $\eps$ to get that 
\beno
{X'_{\nn,\ell}}(t)+\f\kappa 4X_{\nn,\ell}(t)^{1+\f{1}{\nn+\f12}}\ls C,\quad t\in(0,T_0],
\eeno
Therefore, \eqref{004} holds ture if we let $\ell=l-\f32\nn-\f34$, together with \eqref{f1nl}, we conclude that
\ben\label{013}
\|f(t)\|_{H^\nn_{l-\f32\nn-\f34}}\leq \|f_1(t)\|_{H^\nn_{l-\f32\nn-\f34}}+\|f_2(t)\|_{H^\nn_{l-\f32\nn-\f34}}\ls t^{-\f{\nn}{2}-\f14}+1,\quad t\in(0,T_0].
\een
In particular, it holds that  $\sup\limits_{t\in[T_0/2,T_0]}\|f(t)\|_{H^\nn_{l-\f32\nn-\f34}}\ls 1$.

\underline{Step 2: Validity of \eqref{HnlE} over $[0,T]$.} Note that $T>0$ is fixed.  Since $f\in C([0,T],H^{-\f12}_l)$, we can take any moment as the initial time and repeat the aforementioned process, the issue lies in demonstrating that the time for each extension is a common constant, namely,  $T_0$  only depends on parameters $\nn,l$ and \(\sup_{t \in [0, T]}\|f(t)\|_{H^{-\frac{1}{2}}_l}\). Following the notations used in the previous step and their interdependencies,   we only need to demonstrate that for \(t \in [0, T]\), \(\delta\) can be independent of the function \(f(t)\).

Indeed, by the fact that $f\in C([0,T],H^{-\f12}_l)$, we can easily see that $f-J_\delta*f\in C([0,T],H^{-\f12}_l)$. Moreover, it is also continuous with respect to the parameter \(\delta\).
 For any $t_0\in[0,T]$ and $\eps_1$, there exists $\delta(t_0)$ such that $\|f(t_0)-J_{\delta}*f(t_0)\|_{H^{-\f12}_l}<\eps_1/2$ for any $\delta\leq\delta(t_0)$, by continuity, there exists $\tau_{t_0}>0$ such that
\beno
\|f(t)-J_{\delta}*f(t)\|_{H^{-\f12}_l}<\eps_1,\quad t\in(t_0-\tau_{t_0},t_0+\tau_{t_0}),\delta\leq\delta(t_0).
\eeno
Thus $\cup_{t_0\in[0,T]}(t_0-\tau_{t_0},t_0+\tau_{t_0})$ is an open cover of set $[0,T]$, by Heine-Borel Theorem, there exists a finite subcover such that $[0,T]\subset\cup_{i=1}^N(t_i-\tau_{t_i},t_i+\tau_{t_i})$. Then we choose $\delta_0=\min\{\delta(t_i),\cdots,\delta(t_N)\}$ and it holds that
\beno
\|f(t)-J_{\delta_0}*f(t)\|_{H^{-\f12}_l}<\eps_1,\quad \forall t\in[0,T].
\eeno
Therefore, we can conclude that $T_0$ is a common lifespan which only depends on $\nn,l$ and $\sup_{t\in[0,T]}\|f(t)\|_{H^{-\f12}_l}$. 

Now we take $T_0/2$ as the initial time to get that $\sup_{t\in[T_0,3T_0/2]}\|f(t)\|_{H^\nn_{l-\f32\nn-\f34}}\ls 1$. Repeat this process, we can obtain that $\sup_{t\in[T_0/2,T]}\|f(t)\|_{H^\nn_{l-\f32\nn-\f34}}\ls 1$, together with \eqref{013}, we conclude that
\[\|f(t)\|_{H^\nn_{l-\f32\nn-\f34}}\ls t^{-\f{\nn}{2}-\f14}\mathbf{1}_{t\leq 1}+C(t)\mathbf{1}_{t>1},\quad t\in(0,T].\]
Due to the arbitrariness of $\nn$ and \( T \), we end the proof of Theorem \ref{T1}(1.1) for $\mathsf{r}=-\f12$.


\end{proof}

\section{Proof of Theorem \ref{uniquethm}(2) and (3): Global existence and uniqueness}\label{section5}
In the preceding sections, we have derived local results for the spaces \( H^{-\frac{1}{2}, \sss} \) with \( \sss \in [0,1] \). This section aims to show that these local solutions can, in fact, be extended to global solutions.
 This relies on the emergence of Fisher information beyond the initial time which is a direct consequence of the regularity estimates and its monotonic property(see \cite{GS}). Subsequently, we will use direct energy estimates to prove the uniqueness of solutions in the space \( H^{-\frac{1}{2}, \sss}_l \) with \( \sss > \frac{1}{2} \) and $l\geq5$.

 \begin{proof}[Proof of Theorem \ref{uniquethm}(2)]
From previous sections, we already have the local existence results, i.e.
\beno
f\in C([0,T];H^{-\f12,\sss}_l)\cap L^2([0,T];H^{\f12,\sss}_{l-\f32}),\quad l\geq3
\eeno
for some $T>0$. Thus within $t\in[0,T]$, by the argument in the proof of Theorem \ref{T1}(1.1), we can obtain the smoothing estimates
\ben\label{003}
\|f(t)\|_{H^\nn_{l-\f32\nn-\f34}}\ls_\nn t^{-\f\nn 2-\f14}+1,\quad t\in(0,T].
\een
Noticing that Fisher information 
\beno
I(f):=\int _{\R^3}{|\na \sqrt{f}|^2}dv\ls_\eps \|f\|_{H^2_{\f32+\eps}},\quad \mbox{for~any}\quad\eps>0.
\eeno
By interpoaltion, it holds that
 \beno
 I(f)\ls_\eps \|f\|^\theta_{H^{-\f32-\eps}_7}\|f\|^{1-\theta}_{H^{\nn}_{3-\f32\nn-\f34}}\ls_\eps \|f\|^{\theta}_{L^1_7}\|f\|^{1-\theta}_{H^{\nn}_{3-\f32\nn-\f34}},\quad \nn=34,~~\theta=\f{\nn-2}{\nn+\f32+\eps},~~\eps<\f1{100}.
 \eeno
Thus due to \eqref{003}, the fact $f\in L^\infty([0,T],L^1_{2l+1}),l\geq3$ and monotonic property of $I(f)$, we know that for any $t>0$, the fisher information $I(f)(t)<+\infty$.
Observing that $\|f\|_{L^3}\leq I(f)$, thus it holds that
\ben\label{0000}
\|f\|_{H^{-\f12,\sss}_l}\ls \|f\|^{\f12}_{L^3}\|f\|^{\f12}_{L^1_{2l+1}}\ls I(f)^{\f12}\|f\|^{\f12}_{L^1_{2l+1}}.
\een
In particular, $\|f(T)\|_{H^{-\f12,\sss}_l}<\infty$, as a consequence, the solution can be extended globally and 
\ben\label{finL2}
f\in C([0,\infty),H^{-\f12,\sss}_l)\cap L^2_{loc}([0,\infty),H^{\f12,\sss}_{l-\f32}). 
\een
Noticing that \eqref{SMesti} also implies that the solution is $C^\infty(\R^3_v)$ for any positive time.

We ends the proof of this theorem.
\end{proof}

Now, we are in the position of proving uniqueness  in space \( H^{-\frac{1}{2}, \sss}_l \) with \( \sss > \frac{1}{2} \) and $l\geq5$.
\begin{proof}[Proof of Theorem \ref{uniquethm}(3)]
It is sufficient to prove the case $l=5$. Let $\mathbf{f}(t), \ggg(t)$ be two solutions obtained as above with the same initial data $\fff(0)=\ggg(0)=f_0\in H^{-\f12,\sss}_5\cap L^1_{11}$. For $\eps_1>0$ fixed suitably small, choose $\delta>0$ sufficiently small such that $\|f_0-J_\delta*f_0\|_{H^{-\f12,\sss}_5}<\eps_1$. Let $\fff_1(t)\in L^\infty([0,T_1],L^1_{11}\cap H^5_5)$ be the unique solution of \eqref{1} with $\fff_1(0)=J_\delta*f_0$ for some $T_1>0$ and set $\fff_2(t)=\fff(t)-\fff_1(t), \ggg_2(t)=\ggg(t)-\fff_1(t)$. Then $\fff_2(t), \ggg_2(t)$ satisfies
	\[\pa_t\fff_2=Q(\fff,\fff_2)+Q(\fff_2,\fff_1),\quad\pa_t\ggg_2=Q(\ggg,\ggg_2)+Q(\ggg_2,\ggg_1),\quad \fff_2(0)=\ggg_2(0)=f_0-J_\delta*f_0.\]
	We remark that $\|\fff\|^2_{H^{-\f12,\sss}_5}\sim\sum_{j,k}2^{-j}j^{2\sss}2^{10k}\LLT{\F_j\cP_k\fff}^2$. Now multiply on both sides of \eqref{f2evo} $2^{-j}j^{2\sss}2^{10k}$ and sum up with respect to $j$ and $k$. By the previous argument we get that there exists $0<T_2<T_1$ such that $\|\fff_2(t)\|_{H^{-\f12,\sss}_5}\ls2\eps_1, \|\ggg_2(t)\|_{H^{-\f12,\sss}_5}\ls 2\eps_1$ for any $t\in[0,T_2]$. Moreover, it also holds that
	\ben\label{inte}
	 \|\fff_2(t)\|_{H^{\f12,\sss}_{\f72}}^2+\|\ggg_2(t)\|_{H^{\f12,\sss}_{\f72}}^2+\|(-\Delta_{\SS^2})^\f12\fff_2(t)\|_{H^{-\f12,\sss}_{\f72}}^2+\|(-\Delta_{\SS^2})^\f12\ggg_2(t)\|_{H^{-\f12,\sss}_\f72}^2\in L^1[0,T_2],
	 \een
	 where the spherical Laplacian terms come from \eqref{f2evo} and \eqref{LaplaceBeltrami}.
	 
		Let $\hhh(t)=\fff(t)-\ggg(t)$, then $\hhh(t)$ satisfies
	\[\pa_t\hhh=Q(\fff,\hhh)+Q(\hhh,\fff_1)+Q(\hhh,\ggg_2),\quad \hhh(0)=0.\]
	In the following, we are dedicated to proving that $h(t)=0,t\in[0,T_2]$.
	
	By the energy method as before, we have
	\ben\label{eneq}
	\f12\f{d}{dt}\LLT{\F_j\cP_k\hhh}^2+\int_{\R^3}(a*\fff):\na\F_j\cP_k\hhh\mult\na\F_j\cP_k\hhh dv=\mathfrak{J}_1+\mathfrak{J}_2.
	\een
	where
	\[\ba
	\mathfrak{J}_1&:=\int_{\R^3} \fff_1|\F_j\cP_k\hhh|^2dv+\int_{\R^3} \fff_2|\F_j\cP_k\hhh|^2dv\\
	+&\br{\F_j\cP_kQ(\fff_1,\hhh)-Q(\fff_1,\F_j\cP_k\hhh),\F_j\cP_k\hhh}+\br{\F_j\cP_kQ(\fff_2,\hhh)-Q(\fff_2,\F_j\cP_k\hhh),\F_j\cP_k\hhh},\\
	\mathfrak{J}_2&:=\br{\F_j\cP_kQ(\hhh,\fff_1),\F_j\cP_k\hhh}
	+\br{\F_j\cP_kQ(\hhh,\ggg_2),\F_j\cP_k\hhh}.
	\ea\]
	For $\mathfrak{J}_1$, we use the estimate \eqref{RR122} and \eqref{Com122} to get that
$|\mathfrak{J}_1|\ls_N \mathfrak{J}^1_{1}+\mathfrak{J}_1^2$ with
	\ben\label{J11}
	\mathfrak{J}^1_{1}&=&\sum_{\p\leq j}(\LLT{\mF_\p\mP_k\fff_1}+\LLT{\mF_\p\mP_k\fff_2})2^{\f32\p}\LLT{\mF_j\mP_k\hhh}^2\\
	&&+\notag(\|\fff_1\|_{H^{-N}_{-N}}+\|\fff_2\|_{H^{-N}_{-N}})\|\mF_j\mP_k \hhh\|^2_{L^2}+2^{-jN}2^{-kN}(\|\fff_1\|_{H^{-N}_{-N}}+\|\fff_2\|_{H^{-N}_{-N}})\|\hhh\|^2_{H^{-N}_{-N}}.
	\een
and
\ben\label{J12}
\mathfrak{J}_1^2&=&\notag\sum_{\p\geq j+N_0}(\LLT{\mF_\p\mP_k\fff_1}+\LLT{\mF_\p\mP_k\fff_2})\LLT{\mF_\p\mP_k\hhh}\LLT{\mF_j\mP_k\hhh}2^{\f32j}+\sum_{\p\leq j+N_0}(\LLT{\mF_\p\mP_k\fff_1}+\LLT{\mF_\p\mP_k\fff_2})\\
&&\notag\times2^\f\p2\LLT{\mF_j\mP_k\hhh}^22^j+\notag\sum_{\p\leq j+N_0}\LLT{\mF_\p\mP_k\hhh}(\LLT{\mF_j\mP_k\fff_1}+\LLT{\mF_j\mP_k\fff_2})\LLT{\mF_j\mP_k\hhh}2^{\f32\p}
+2^{-jN}2^{-kN}\\
&&\notag\times(\|\fff_1\|_{L^2_{-N}}+\|\fff_2\|_{L^2_{-N}})\|\hhh\|_{L^2_{-N}}\HHLL{\hhh}
+2^{-2k}(\|\fff_1\|_{L^1_2}+\|\fff_2\|_{L^1_2})\LLT{\mF_j\mP_k\hhh}^22^j+2^{-2k}2^{-jN}\\
&&\times(\|\fff_1\|_{L^1_2}+\|\fff_2\|_{L^1_2})\HHL{\mP_k\hhh}^2+2^{-jN}2^{-kN}(\|\fff_1\|_{L^1_2}+\|\fff_2\|_{L^1_2})\HHLL{\hhh}^2.
\een

	 We then focus on the estimates of $\mathfrak{J}_2$ and split it into:
	\beno
	&&\mathfrak{J}_2=\br{\F_j\cP_kQ_{-1}(\hhh,\ggg_2),\F_j\cP_k\hhh}+\sum_{\l\geq0}\br{\F_j\cP_kQ_\l(\hhh,\ggg_2),\F_j\cP_k\hhh}\\
	&&+\br{\F_j\cP_kQ_{-1}(\hhh,\fff_1),\F_j\cP_k\hhh}+\sum_{\l\geq0}\br{\F_j\cP_kQ_\l(\hhh,\fff_1),\F_j\cP_k\hhh}:=\mathfrak{J}_2^1+\mathfrak{J}_2^2+\mathfrak{J}_2^3+\mathfrak{J}_2^4.
	\eeno
	We use \eqref{l0212} to get the estimates of $\mathfrak{J}_2^1$ and $\mathfrak{J}_2^3$ and we have that
   \ben\label{j213}
&&\notag|\mathfrak{J}_2^1|+|\mathfrak{J}_2^3|\ls_N \sum_{\a\leq j}(\LLT{\mF_\a\mP_k\fff_1}+\LLT{\mF_\a\mP_k\ggg_2})2^{\f32\a}\LLT{\mF_j\mP_k\hhh}^2+\sum_{\p\leq j}\LLT{\mF_\p\mP_k\hhh}2^{-\f\p2}\\
&&\notag\times
(\LLT{\mF_j\mP_k\fff_1}+\LLT{\mF_j\mP_k\ggg_2})\LLT{\mF_j\mP_k\hhh}2^{2j}+\sum_{\p\geq j}\LLT{\mF_\p\mP_k\hhh}(\LLT{\mF_\p\mP_k\fff_1}+\LLT{\mF_\p\mP_k\ggg_2})\\
&&\notag\times\LLT{\mF_j\mP_k\hhh}2^{\f32j}
+2^{-jN}2^{-kN}(\HHLL{\fff_1}+\HHLL{\ggg_2})\HHLL{\hhh}^2\\
&+&2^{-jN}2^{-kN}\|\mP_k\hhh\|_{L^2_{-N}}(\|\mP_k\fff_1\|_{L^2_{-N}}+\|\mP_k\ggg_2\|_{L^2_{-N}})\HHLL{\hhh}.
\een

	 While for $\mathfrak{J}_2^2$ and $\mathfrak{J}_2^4$, we first perform the dyadic decomposition in frequency space as in \eqref{asin}. Rather than directly applying \eqref{M1}\eqref{M4}\eqref{M2} and \eqref{M3}, we require additional estimates for $\fM^1, \fM^2, \fM^3$ and $\fM^4$.  Noticing that for $\p\geq0$, it holds that
	\[\|\F_\p a_\l*\F_\p g\|_{L^\infty}\leq\int_{\R^3}\vf(2^{-\p}\xi)2^{2\l}\nr{\wh{a_0}(2^\l\xi)}\nr{\wh{\F_\p g}(\xi)}d\xi\ls_N2^{-2\l N}2^{-2\p N}\LLT{\F_\p g}\ls\|g\|_{H^{-\f12}}2^{-\l N}2^{-\p N},\]
	where we use $|\wh{a_0}(\eta)|\ls_N\<\eta\>^{-4N}$. Thus
	\ben\label{M12}
	\nr{\fM^1_{\l,\p,\a}(g,h,f)}\ls_{N}2^{-\l N}2^{-\p N}2^{-\a N}\|g\|_{H^{-\f12}}\LLT{\F_\a h}\LLT{\tF_\p f},\\\label{M22}
	\nr{\fM^4_{\l,\p,\m}(g,h,f)}\ls_{N}2^{-\l N}2^{-\p N}2^{-\m N}\|g\|_{H^{-\f12}}\LLT{\tF_\p h}\LLT{\F_\m f}.
	\een
	For $\fM^2$, we first have that
\beno
\fM^2_{\l,\a}(g,h,f)=-\int_{\R^3}(a_\l*\tS_{\a-N_0}g):\na \F_\a h\mult\na\tF_\a fdv+\int_{\R^3}(b_\l*\tS_{\a-N_0}g)\cdot\F_\a h\na\tF_\a fdv.
\eeno
For the first term, note that
	\beno
	&&\int_{\R^3}(a_\l*\tS_{\a-N_0}g):\na \F_\a h\mult\na\tF_\a fdv\\
	&=&\iint_{\R^6}|v-v_*|^{-3}\vf(2^{-\l}(v-v_*))\rr{\tS_{\a-N_0}g}_* \big(\rr{v-v_*}\times \na \F_\a h\big)\cdot\big({\rr{v-v_*}}\times \na \tF_\a f\big)dv_*dv.
	\eeno
Since 
	\[\ba
	\int_{\R^3}|v-v_*|^{-3}\vf(2^{-\l}(v-v_*))\rr{\tS_{\a-N_0}g}_*dv_*\leq\int_{\R^3}\nr{\wh{|\cdot|^{-3}\vf}(2^{\l}\xi)}\nr{\wh{g}(\xi)}d\xi\\
	\ls\sum_{j\geq0}2^{-3l}2^{-j}\LLT{\F_j g}+\LLO{\wh{|\cdot|^{-3}\vf(\cdot)}}2^{-3l}\|\F_{-1}g\|_{L^2_2}\ls2^{-3l}\|g\|_{H^{-\f12}_2},
	\ea\]
	where we use $|\wh{|\cdot|^{-3}\vf(\cdot)}(\eta)|\ls\<\eta\>^{-4}$ and $\|\hat{g}(\xi)\|_{L^\infty}\ls \|\<v\>^2g\|_{L^2}$. Thus we get
	\beno
	&&\Big|\iint_{\R^6}|v-v_*|^{-3}\vf(2^{-\l}(v-v_*))\rr{\tS_{\a-N_0}g}_* \rr{v\times \na \F_\a h}\cdot\rr{v\times \na \tF_\a f}dv_*dv\Big|\\
	&\ls& 2^{-3\l}\|g\|_{H^{-\f12}_2}\Big(\LLT{\rr{-\Delta_{\S^2}}^{\f{1}{2}}\F_\a h}+\|\F_\a h\|_{L^2}2^\a\Big)\Big(\LLT{\rr{-\Delta_{\S^2}}^{\f{1}{2}}\F_\a f}+\|\tF_\a f\|_{L^2}2^\a\Big).
	\eeno
	By the similar argument, we can derive that
	\beno
	&&\Big|\int_{\R^3}(a_\l*\tS_{\a-N_0}g):\na \F_\a h\mult\na\tF_\a fdv\Big|\\
	&\ls&  2^{-3\l}\|g\|_{H^{-\f12}_4}\Big(\LLT{\rr{-\Delta_{\S^2}}^{\f{1}{2}}\F_\a h}+\|\F_\a h\|_{L^2}2^\a\Big)\Big(\LLT{\rr{-\Delta_{\S^2}}^{\f{1}{2}}\F_\a f}+\|\tF_\a f\|_{L^2}2^\a\Big).
	\eeno
	and
	\beno
	\Big|\int_{\R^3}(b_\l*\tS_{\a-N_0}g)\cdot\F_\a h\na\tF_\a fdv\Big|\ls 2^{-2\l}\|g\|_{H^{-\f12}_2}\|\F_\a h\|_{L^2}\|\tF_\a f\|_{L^2}2^\a.
	\eeno
Then we conclude that
	\ben\label{dfg}
&&|\fM^2_{\l,\a}(g,h,f)|\ls2^{-3\l}\|g\|_{H^{-\f12}_4}\rr{\LLT{\rr{-\Delta_{\S^2}}^{\f{1}{2}}\F_\a h}+\LLT{\F_\a h}2^\a}\\
&&\notag\times\rr{\LLT{\rr{-\Delta_{\S^2}}^{\f{1}{2}}\tF_\a f}+\LLT{\tF_\a f}2^\a}+2^{-2\l}\|g\|_{H^{-\f12}_2}\LLT{\F_\a h}\LLT{\tF_\a f}2^\a.
\een
As a matter of fact, we can repeat the above amplification to obtain

	\ben\label{M32}
&&|\fM^2_{\l,\a}(g,h,f)|\ls\min\Big\{2^{-3\l}\|g\|_{H^{-\f12}_4}\rr{\LLT{\rr{-\Delta_{\S^2}}^{\f{1}{2}}\F_\a h}+\LLT{\F_\a h}2^\a}\\
	&&\notag\times\rr{\LLT{\rr{-\Delta_{\S^2}}^{\f{1}{2}}\tF_\a f}+\LLT{\tF_\a f}2^\a}+2^{-2\l}\|g\|_{H^{-\f12}_4}\LLT{\F_\a h}\LLT{\tF_\a f}2^\a, 2^{-\l}\|g\|_{H^{-\f12}_4}\LLT{\F_\a h}\LLT{\tF_\a f}2^{2\a}\Big\}.
	\een

The term $\fM^3$ can be estimated in a similar way and has the same bound as above. Now we plug \eqref{M12}\eqref{M22} and \eqref{M32} into \eqref{asin} to obtain that
	\[\ba
	&\br{Q_\l(\hhh,\ggg_2),\cP_k\F_j^2\cP_k\hhh}\ls_N\min\Big\{\sum_{\p\geq-1}2^{-3\l}\|\hhh\|_{H^{-\f12}_4}\rr{\LLT{\rr{-\Delta_{\S^2}}^{\f{1}{2}}\F_\p \ggg_2}+\LLT{\F_\p \ggg_2}2^\p}\\
	&\times\rr{\LLT{\rr{-\Delta_{\S^2}}^{\f{1}{2}}\tF_\p \cP_k\F_j^2\cP_k\hhh}+\LLT{\tF_\p \cP_k\F_j^2\cP_k\hhh}2^\p}+\sum_{\p\geq-1}2^{-2\l}\|\hhh\|_{H^{-\f12}_4}\LLT{\F_\p \ggg_2}\LLT{\tF_\p \cP_k\F_j^2\cP_k\hhh}2^\p,\\
	&\sum_{\p\geq-1}2^{-\l}\|\hhh\|_{H^{-\f12}_4}\LLT{\tF_\p \ggg_2}\LLT{\tF_\p\cP_k\F_j^2\cP_k\hhh}2^{2\p} \Big\}\\
	+&\sum_{\m\leq\p+N_0}2^{-\l N}2^{-2\p N}2^{-2\m N}\|\hhh\|_{H^{-\f12}_2}
	\rr{\LLT{\F_\m \ggg_2}\LLT{\tF_\p\cP_k\F_j^2\cP_k \hhh}+\LLT{\tF_\p \ggg_2}\LLT{\F_\m\cP_k\F_j^2\cP_k\hhh}}. 
	\ea\]
	Next, we perform the dyadic decomposition in phase space as in \eqref{asin2} and use above estimate to get
	\ben\label{j22}
	&&\mathfrak{J}_2^2\ls_N2^{-3k}\|\hhh\|_{H^{-\f12}_5}\Big({\LLT{\rr{-\Delta_{\S^2}}^{\f{1}{2}}\mF_j\mP_k \ggg_2}+\LLT{\mF_j\mP_k \ggg_2}2^j}\Big)\\
	&&\notag\times\Big(\LLT{\rr{-\Delta_{\S^2}}^{\f{1}{2}}\mF_j \mP_k\hhh}
	+\LLT{\mF_j\mP_k\hhh}2^j\Big)+2^{-2k}\|\hhh\|_{H^{-\f12}_5}\LLT{\mF_j\mP_k \ggg_2}\LLT{\mF_j\mP_k\hhh}2^j\\
	&+&\notag2^{-k}2^{-jN}\|\hhh\|_{H^{-\f12}_5}\HHL{\mP_k\ggg_2}\HHL{\mP_k\hhh}
	+2^{-jN}2^{-kN}\|\hhh\|_{H^{-\f12}_5}\HHLL{\ggg_2}\HHLL{\hhh}.
	\een
	By the same process, we can derive that
	\ben\label{j24}
&&\notag\mathfrak{J}_2^4\ls2^{-k}\|\hhh\|_{H^{-\f12}_5}\LLT{\mF_j\mP_k \fff_1}\LLT{\mF_j\mP_k\hhh}2^{2j}
	+2^{-k}2^{-jN}\|\hhh\|_{H^{-\f12}_5}\HHL{\mP_k\fff_1}\HHL{\mP_k\hhh}\\
	&+&2^{-jN}2^{-kN}\|\hhh\|_{H^{-\f12}_5}\HHLL{\fff_1}\HHLL{\hhh}.
	\een
	Plug all the estimates of $\mathfrak{J}_1^1,\mathfrak{J}_1^2$ and $\mathfrak{J}_2^i,i=1,2,3,4$, \eqref{J11}\eqref{J12}\eqref{j213}\eqref{j22} and \eqref{j24} into \eqref{eneq}, and in conjunction with Proposition \ref{coer}, we proceed with some simplifications to derive the following result:
	\ben\label{unique}
	\f12\f{d}{dt}\LLT{\F_j\cP_k\hhh}^2+\kappa\LLT{\F_j\cP_k\hhh}^22^{2j}2^{-3k}+\kappa\LLT{(-\Delta_{\SS^2})^\f12\F_j\cP_k\hhh}^22^{-3k}\ls_N\mathfrak{K}_1+\mathfrak{K}_2,
	\een
	where 
	\beno
	&&\mathfrak{K}_1=\sum_{\p\geq j}\LLT{\mF_\p\mP_k\fff_2}\LLT{\mF_\p\mP_k\hhh}\LLT{\mF_j\mP_k\hhh}2^{\f32j}
	+\sum_{\p\leq j+N_0}\LLT{\mF_\p\mP_k\fff_2}2^{\f{\p}{2}}\LLT{\mF_j\mP_k\hhh}^22^j\\
	&+&\sum_{\p\leq j+N_0}\LLT{\mF_\p\mP_k\hhh}2^{\f{3}{2}\p}\LLT{\mF_j\mP_k\fff_2}\LLT{\mF_j\mP_k\hhh}
	+\sum_{\p\geq j}\LLT{\mF_\p\mP_k\fff_1}\LLT{\mF_\p\mP_k\hhh}\LLT{\mF_j\mP_k\hhh}2^{\f32j}\\
	&+&\sum_{\p\leq j+N_0}\LLT{\mF_\p\mP_k\fff_1}2^{\f{\p}{2}}\LLT{\mF_j\mP_k\hhh}^22^j
	+\sum_{\p\leq j+N_0}\LLT{\mF_\p\mP_k\hhh}2^{-\f{\p}{2}}\LLT{\mF_j\mP_k\fff_1}\LLT{\mF_j\mP_k\hhh}2^{2j}\\
	&+&\sum_{\p\geq j}\LLT{\mF_\p\mP_k\ggg_2}\LLT{\mF_\p\mP_k\hhh}\LLT{\mF_j\mP_k\hhh}2^{\f32j}
	+\sum_{\p\leq j+N_0}\LLT{\mF_\p\mP_k\ggg_2}2^{\f{3}{2}\p}\LLT{\mF_j\mP_k\hhh}^2\\
	&+&\sum_{\p\leq j+N_0}\LLT{\mF_\p\mP_k\hhh}2^{-\f{\p}{2}}\LLT{\mF_j\mP_k\ggg_2}\LLT{\mF_j\mP_k\hhh}2^{2j}
	+2^{-3k}\|\hhh\|_{H^{-\f12}_5}\Big(\LLT{\rr{-\Delta_{\S^2}}^{\f{1}{2}}\mF_j\mP_k \ggg_2}\\
	&+&\LLT{\mF_j\mP_k \ggg_2}2^j\Big)\left(\LLT{\rr{-\Delta_{\S^2}}^{\f{1}{2}}\mF_j \mP_k\hhh}+\LLT{\mF_j\mP_k\hhh}2^j\right)
	:=\sum_{i=1}^{10}\mathfrak{K}_1^i,
	\eeno
	\beno
	&&\mathfrak{K}_2=2^{-2k}\|\hhh\|_{H^{-\f12}_5}\LLT{\mF_j\mP_k \ggg_2}\LLT{\mF_j\mP_k\hhh}2^j+2^{-k}\|\hhh\|_{H^{-\f12}_5}\LLT{\mF_j\mP_k \fff_1}\LLT{\mF_j\mP_k\hhh}2^{2j}\\
	&+&2^{-jN}2^{-kN}\|\fff_1\|_{L^2_{-N}}\|\hhh\|_{L^2_{-N}}\HHLL{\hhh}+2^{-2k}\|\fff_1\|_{L^1_2}\LLT{\mF_j\mP_k\hhh}^22^j
	+2^{-2k}2^{-jN}\|\fff_1\|_{L^1_2}\HHL{\mP_k\hhh}^2\\
	&+&2^{-jN}2^{-kN}\|\fff_1\|_{L^1_2}\HHLL{\hhh}^2
	+2^{-jN}2^{-kN}\|\fff_2\|_{L^2_{-N}}\|\hhh\|_{L^2_{-N}}\HHLL{\hhh}+2^{-2k}\|\fff_2\|_{L^1_2}\LLT{\mF_j\mP_k\hhh}^22^j\\
	&+&2^{-2k}2^{-jN}\|\fff_2\|_{L^1_2}\HHL{\mP_k\hhh}^2
	+2^{-jN}2^{-kN}\|\fff_2\|_{L^1_2}\HHLL{\hhh}^2+2^{-jN}2^{-kN}\HHLL{\ggg_2}\HHLL{\hhh}^2\\
	&+&2^{-jN}2^{-kN}\HHLL{\hhh}\|\mP_k\hhh\|_{L^2_{-N}}\|\mP_k\ggg_2\|_{L^2_{-N}}
	+2^{-k}2^{-jN}\|\hhh\|_{H^{-\f12}_5}\HHL{\mP_k\ggg_2}\HHL{\mP_k\hhh}\\
	&+&2^{-jN}2^{-kN}\|\hhh\|_{H^{-\f12}_5}\HHLL{\ggg_2}\HHLL{\hhh}
	+2^{-k}2^{-jN}\|\hhh\|_{H^{-\f12}_5}\HHL{\mP_k\fff_1}\HHL{\mP_k\hhh}\\
	&+&2^{-jN}2^{-kN}\|\hhh\|_{H^{-\f12}_5}\HHLL{\fff_1}\HHLL{\hhh}.	
	\eeno

	We choose $N=10$ and multiply  $2^{-j}j^{2\sss}2^{10k}$ on both sides of \eqref{unique} and sum up with respect to $j$ and $k$. For the first three terms of $\mathfrak{K}_1$, we have
	\ben\label{K123}
	&&\notag\sum_{j,k}(\mathfrak{K}_1^1+\mathfrak{K}_1^2+\mathfrak{K}_1^3)2^{-j}j^{2\sss}2^{10k}\ls
	\sum_{j,k}\sum_{\p\geq j}\LLT{\mF_\p\mP_k\fff_2}2^{\f\p2}p^\sss\LLT{\mF_\p\mP_k\hhh}2^{\f\p2}p^\sss\LLT{\mF_j\mP_k\hhh}\\
	&&\notag\times 2^{-\f12j}2^{j-\p}\p^{-2\sss}j^{2\sss}2^{10k}+\sum_{j,k}\sum_{\p\leq j+N_0}\LLT{\mF_\p\mP_k\fff_2}2^{-\f{\p}{2}}\LLT{\mF_j\mP_k\hhh}^2 2^j2^{\p-j} j^{2\sss}2^{10k}\\
	&+&\sum_{j,k}\sum_{\p\leq j+N_0}\LLT{\mF_\p\mP_k\hhh}2^{-\f{\p}{2}}\LLT{\mF_j\mP_k\fff_2}2^{\f j2}j^{\sss}\LLT{\mF_j\mP_k\hhh}2^{\f j2}j^{\sss}2^{2(\p-j)}2^{10k}\\
	&\ls&\notag\|\hhh\|_{H^{-\f12}_5}\|\fff_2\|_{H^{\f12,\sss}_\f72}\|\hhh\|_{H^{\f12,\sss}_\f72}+\|\fff_2\|_{H^{-\f12}_5}\|\hhh\|_{H^{\f12,\sss}_\f72}^2\ls\rr{\|\fff_2\|_{H^{-\f12}_5}+\eps}\|\hhh\|_{H^{\f12,\sss}_\f72}^2+C_\eps\|\hhh\|_{H^{-\f12}_5}^2\|\fff_2\|_{H^{\f12,\sss}_\f72}^2.
	\een
	
	For the fourth to sixth terms of $\mathfrak{K}_1$ which contains $\fff_1$, we can use the high regularity of $\fff_1$ to get that
	\ben\label{K456}
	&&\notag\sum_{j,k}(\mathfrak{K}_1^4+\mathfrak{K}_1^5+\mathfrak{K}_1^6)2^{-j}j^{2\sss}2^{10k}\\
	&\ls& \|\fff_1\|_{H^3_5}\|\hhh\|^2_{H^{-\f12}_5}+\|\fff_1\|_{H^1_5}\|\hhh\|^2_{H^{0,\sss}_{\f72}}\ls\eps \|\hhh\|^2_{H^{\f12,\sss}_{\f72}}+\big( \|\fff_1\|_{H^3_5}+C_\eps\|\fff_1\|^2_{H^1_5}\big)\|\hhh\|^2_{H^{-\f12,\sss}_{5}}.
	\een

	While for the seventh to ninth terms containing $\ggg_2$, we have that
	\ben\label{K789}
		&&\notag\sum_{j,k}(\mathfrak{K}_1^7+\mathfrak{K}_1^8+\mathfrak{K}_1^9)2^{-j}j^{2\sss}2^{10k}=\sum_{j,k}\sum_{\p\geq j}\LLT{\mF_\p\mP_k\ggg_2}2^{\f\p2}\p^\sss\LLT{\mF_\p\mP_k\hhh}2^{\f\p2}\p^\sss\LLT{\mF_j\mP_k\hhh}2^{-\f j2}2^{-(\p-j)}2^{10k}\\
		&+&\notag\sum_{j,k}\sum_{\p\leq j+N_0}\LLT{\mF_\p\mP_k\ggg_2}2^{-\f{\p}{2}}\LLT{\mF_j\mP_k\hhh}^22^{j}j^{2\sss}2^{2(\p-j)}2^{10k}
		+\sum_{j,k}\sum_{\p\leq j+N_0}\LLT{\mF_\p\mP_k\hhh}2^{-\f{\p}{2}}\p^\sss\LLT{\mF_j\mP_k\ggg_2}\\
		&&\notag\times\LLT{\mF_j\mP_k\hhh}2^{j}j^{2\sss}2^{10k}\p^{-\sss}\ls \|\ggg_2\|_{H^{-\f12}_5}\|\hhh\|_{H^{\f12,\sss}_\f72}^2+\|\hhh\|_{H^{-\f12,\sss}_5}\|\ggg_2\|_{H^{\f12,\sss}_\f72}\|\hhh\|_{H^{\f12,\sss}_\f72}\\
		&\ls&\rr{\|\ggg_2\|_{H^{-\f12}_5}+\eps}\|\hhh\|_{H^{\f12,\sss}_\f72}^2+C_\eps\|\hhh\|_{H^{-\f12,\sss}_5}^2\|\ggg_2\|_{H^{\f12,\sss}_\f72}^2.
	\een
In which the last term can be summed because \( \sss>\f12 \).

	For the last term $\mathfrak{K}_1^{10}$, we have that
	\ben\label{K10}
	&&\notag\sum_{j,k}\mathfrak{K}_1^{10}2^{-j}j^{2\sss}2^{10k}\ls \sum_{j,k}\|\hhh\|_{H^{-\f12}_5}\rr{\LLT{\rr{-\Delta_{\S^2}}^{\f{1}{2}}\mF_j\mP_k \ggg_2}2^{-\f j2}2^{\f{7k}{2}}j^\sss+\LLT{\mF_j\mP_k \ggg_2}2^{\f j2}2^{\f{7k}{2}}j^\sss}\\
	&&\notag\times\left(\LLT{\rr{-\Delta_{\S^2}}^{\f{1}{2}}\mF_j \mP_k\hhh}2^{-\f j2}2^{\f{7k}{2}}j^\sss+\LLT{\mF_j\mP_k\hhh}2^{\f j2}2^{\f{7k}{2}}j^\sss\right)\\
	&\ls&\notag \|\hhh\|_{H^{-\f12}_5}\Big({\|\ggg_2\|_{H^{\f12,\sss}_\f72}+\|(-\Delta_{\SS^2})^\f12\ggg_2\|_{H^{-\f12,\sss}_\f72}}\Big)\Big({\|\hhh\|_{H^{\f12,\sss}_\f72}+\|(-\Delta_{\SS^2})^\f12\hhh\|_{H^{-\f12,\sss}_\f72}}\Big)\\
	&\ls&\eps\Big({\|\hhh\|_{H^{\f12,\sss}_\f72}^2+\|(-\Delta_{\SS^2})^\f12\hhh\|_{H^{-\f12,\sss}_\f72}^2}\Big)+C_\eps\|\hhh\|_{H^{-\f12}_5}^2\Big({\|\ggg_2\|_{H^{\f12,\sss}_\f72}^2+\|(-\Delta_{\SS^2})^\f12\ggg_2\|_{H^{-\f12,\sss}_\f72}^2}\Big).
	\een

	While the term in $\mathfrak{K}_2$ can be estimated in a similar way, we skip the details here. One may check that
	\ben\label{K2}
	\sum_{j,k}\mathfrak{K}_2 2^{-j}j^{2\sss}2^{10k}\ls_N \eps \|\hhh\|^2_{H^{\f12,\sss}_{\f72}}+(C_\eps+\|\fff_1\|^2_{H^3_5}+\|\ggg_2\|^2_{H^{\f12,\sss}_{\f72}})\|\hhh\|^2_{H^{-\f12,\sss}_5}.
	\een

	Now, let us define \( X(t) := \sum_{j,k} \left\| \mathcal{F}_j \mathcal{P}_k \mathbf{h} \right\|^2 2^{-j} 2^{10k} j^{2\sss} \). By combining all the aforementioned estimates \eqref{K123}\eqref{K456}\eqref{K789}\eqref{K10} and \eqref{K2}, and observing that \( \|\mathbf{f}_2\|_{H^{-\frac{1}{2}}_5} + \|\mathbf{g}_2\|_{H^{-\frac{1}{2}}_5} \lesssim 2\epsilon_1 \) and $\sup_{t}\|\fff_1(t)\|_{H^3_5}\leq C$, we can select a sufficiently small \(\eps\) to deduce that
	\[{X}'\leq CZX, \quad\mbox{with}\quad Z(t)=\|(-\Delta_{\SS^2})^\f12\ggg_2\|_{H^{-\f12,\sss}_\f72}^2+\|\ggg_2\|_{H^{\f12,\sss}_\f72}^2+\|\fff_2\|_{H^{\f12,\sss}_\f72}^2+1.\]
Given that \( Z(t) \in L^1[0, T_2] \) due to equation \eqref{inte} and \( X(0) = 0 \), we deduce that \( X(t) = 0,t\in[0,T_2] \). This implies that \( \mathbf{h}(t) = 0 \) for $t\in[0,T_2]$, and consequently, \( \mathbf{f}(t) = \mathbf{g}(t) \) for $t\in[0,T_2]$. Clearly, this result can be extended to any time $T>0$. This concludes the proof of Theorem \ref{uniquethm}(3).

\end{proof}

\section{Appendix}\label{app}
In the Appendix, we give some useful lemmas on pseudo-Differential operator and basic commutators.
 
\begin{lem}[see \cite{HE} Lemma 5.3]\label{le1.1}
	Let $s, r\in\R$ and $a(v),b(\xi)\in C^\infty$ satisfy for any $\alpha\in\Z^3_+$,
	\ben\label{abconstants}
	|D_v^\al a(v)|\leq C_{1,\al}\<v\>^{r-|\al|},~|D_\xi^\al b(\xi)|\leq C_{2,\al}\<\xi\>^{s-|\al|}
	\een
	for constants $C_{1,\al},C_{2,\al}$. Then there exists a constant $C$ depending only on $s,r$ and finite numbers of $C_{1,\al},C_{2,\al}$ such that for any $f\in \mathscr{S}(\R^3)$,
	\beno
	\|a(v)b(D)f\|_{L^2}\leq C\|\<D\>^s\<v\>^rf\|_{L^2},~\|b(D)a(v)f\|_{L^2}\leq C\|\<v\>^r\<D\>^sf\|_{L^2}.
	\eeno
	As a direct consequence, we get that $\|\<D\>^m\<v\>^lf\|_{L^2}\sim\|\<v\>^l\<D\>^mf\|_{L^2}\sim\|f\|_{H^m_l}$.
\end{lem}

\begin{lem}[see \cite{HJ1} Lemma 4.12]\label{le1.2}
	Let $l,s,r\in\R,M(\xi)\in S_{1,0}^r$ and $\Phi(v)\in S_{1,0}^l$. Then there exists a constant $C$ such that $\|[M(D_v),\Phi(v)]f\|_{H^s}\leq C\|f\|_{H_{l-1}^{r+s-1}}$. Moreover, for any $N\in\N,$
	\ben\label{MPHICOMMU}
	M(D_v)\Phi=\Phi M(D_v)+\sum_{1\leq|\al|\leq N}\frac{1}{\al!}\Phi_\al M^\al(D_v)+r_N(v,D_v),
	\een
	where $\Phi_\al(v)=\pa_v^\al\Phi,~M^\al(\xi)=\pa_\xi^\al M(\xi)$ and $\<v\>^{N-l}r_N(v,\xi)\in S^{r-N}_{1,0}$. Moreover, for any $\beta,\beta' \in\Z^3_+$, we have
	\ben\label{rN}
	|\pa^\beta_v\pa^{\beta'}_\xi r_N(v,\xi)|\leq C_{\beta,\beta'}\<\xi\>^{r-N-|\beta|}\<v\>^{l-N-|\beta'|},\quad \|r_{2N+1}(v,D_v)\<D\>^{N}\<v\>^{N}f\|_{L^2}\leq C\|f\|_{L^2}.
	\een
	Furthermore, using (\ref{MPHICOMMU}) repeatedly, we can also obtain that
	\ben\label{MPHICOMMU2}
	M(D_v)\Phi=\Phi M(D_v)+\sum_{1\leq|\al|\leq N}C_{\al} M^\al(D_v)\Phi_\al+C_Nr_N(v,D_v).
	\een
\end{lem} 
 
\begin{rmk}\label{CONSTS} We emphasize that in the statement of Lemma \ref{le1.2}, the constant $C$ appearing in the inequality depends only on $C_{1,\al},C_{2,\al}$ in \eqref{abconstants} with $a=\Phi$ and $b=M$ and also the constants $ C_{\beta,\beta'}$ for $r_N(v,\xi)$. This fact is crucial for the estimates of commutators and the profiles of weighted Sobolev spaces. For instance,  if $M(D_v)$ and $\Phi(v)$ are chosen to be the localized operators $\F_j$ and $\cP_k$, the constant $C$ in Lemma \ref{le1.2} does not depend on $j$ and $k$. Indeed,
	for any $k\geq0,N\in\N$, $2^{Nk}\vphi(2^{-k}v)$ satisfies that for any $\alpha\in\Z^3_+$,
	\ben\label{Ncon}
	|D_v^\al 2^{Nk}\vphi(2^{-k}v) |\leq C_{N,\al}\<v\>^{N-|\al|}|\vphi_\al(2^{-k}v)|\leq C_{N,\al}\<v\>^{N-|\al|}.
	\een
\end{rmk}

\begin{lem}\label{7.8}
	(Bernstein inequality). There exists a constant $C$ independent of $j$ and $f$ such that
	
	(1) For any $s\in\R$ and $j\geq 0$,
	\beno
	C^{-1}2^{js}\|\F_jf\|_{L^2(\R^3)}\leq\|\F_jf\|_{H^s(\R^3)}\leq C 2^{js}\|\F_jf\|_{L^2(\R^3)}.
	\eeno
	
	(2) For integers $j,k\geq0$ and $p,q\in[1,\infty],q\geq p$, the Bernstein inequality are shown as
	\beno
	&&\sup_{|\al|=k}\|\pa^\al\F_jf\|_{L^q(\R^3)}\ls2^{jk}2^{3j(\frac{1}{p}-\frac{1}{q})}\|\F_jf\|_{L^p(\R^3)},
	\sup_{|\al|=k}\|\pa^\al S_jf\|_{L^q(\R^3)}\ls2^{jk}2^{3j(\frac{1}{p}-\frac{1}{q})}\|S_jf\|_{L^p(\R^3)},\\
	&&2^{jk}\|\F_jf\|_{L^p(\R^3)}\ls\sup_{|\al|=k}\|\pa^\al\F_jf\|_{L^p(\R^3)}\ls2^{jk}\|\F_jf\|_{L^p(\R^3)}.
	\eeno
\end{lem}

\begin{lem}\label{PFCom}
	If $\cP_k,\U_k$ and $\F_j$ are defined in Section \ref{DDS} and $n\in\R^+$, then
	
	(i) For any $N\in \N$, there exists a constant $C_{N}$ such that
	\beno
	\|[\cP_k,\F_j ]f\|_{L^2}&=&\|(\cP_k\F_j -\F_j\cP_k)f\|_{L^2}\leq C_{N} \big( 2^{ -j}2^{-k}\sum_{|\al|=1}^{2N}\|\cP_{k,\al}\F_{j,\al}f\|_{L^2}+2^{-jN}2^{-kN}\|f\|_{H_{-N}^{-N}}\big),\\
	\|[\U_k,\F_j]f\|_{L^2}&=&\|(\U_k\F_j-\F_j \U_k)f\|_{L^2}\leq C_{N } \big( 2^{ -j}\sum_{|\al|=1}^{2N}\|\U_{k,\al}\F_{j,\al}f\|_{L^2}+2^{-jN}\|f\|_{H_{-N}^{ -N}}\big),
	\eeno
	where $\cP_{k,\al},\F_{j,\al}$ and $\U_{k,\al}$ are defined in Definition \ref{Fj}. Moreover, replace $\cP_{k,\al}$ and $\F_{j,\al}$ by $\tP_{k,\al}$ and $\tF_{j,\al}$ respectively, then the above results still hold.
	
	(ii) For $|m-p|>N_0$ and $\forall N\in \N$, there exists a constant $C_N$ such that
	\beno
	\|\F_m\cP_k\F_pg\|_{L^2}\leq C_N2^{-(p+m+k)N}\|\F_pg\|_{L^2_{-N}},\quad \|\F_m\U_{k}\F_pg\|_{L^2}\leq C_N2^{-(p+m)N}\|\F_pg\|_{L^2_{-N}}.
	\eeno
	If $m>p+N_0$, we have
	$\|\F_m\cP_{k}S_pg\|_{L^2}\leq C_N2^{-(m+k)N}\|S_pg\|_{L^2_{-N}},
    \|\F_m\U_{k}S_pg\|_{L^2}\leq C_N2^{-mN}\|S_pg\|_{L^2_{-N}}.$

	(iii) For any $a,w\in \R$, we have
	$\|\U_{k+N_0}h\|_{H^a}\leq C_{a,w}2^{k(-w)^+}\|h\|_{H^a_w}$ and
	$\|S_{p+N_0}h\|_{L^2_l}\leq C_{l}\|h\|_{L^2_l}.$

	(iv) For any $j\geq -1$ and $l\in\R$, we have
	$ \|\F_jf\|_{L^1_l}+\|S_jf\|_{L^1_l}\leq C_l\|f\|_{L^1_l}.$
\end{lem}
\begin{proof} It follows the proof of Lemma A.4 in \cite{HJ1} and the estimate \eqref{rN}. We omit the details here.
\end{proof}

 \begin{lem}\label{profileofHSk}
 	Recall the localized operators $\F,\tF,\cP,\tP$ and $\mF,\mP$ defined in Section \ref{DDS}  and Definition \ref{Fj}, we have 
 	
	(i) Let $m, l\in \R.$ Then for $f\in H_l^m$,
	\ben\label{Ber}
	\sum_{k=-1}^\infty2^{2kl}\|\tP_k f\|^2_{H^m}\sim\sum_{k=-1}^\infty2^{2kl}\|\cP_k f\|^2_{H^m}\sim\|f\|^2_{H^m_l}\sim\sum_{j=-1}^\infty2^{2 j m}\|\F_jf\|^2_{L^2_l}\sim\sum_{j=-1}^\infty2^{2 j m}\|\tF_jf\|^2_{L^2_l}.
	\een
	Moreover, we   have
	\ben\label{7.70}
	\sum_{k=-1}^\infty2^{2kl}\|\mP_{k}f\|^2_{H^m}\le C_{m,l}\sum_{k=-1}^\infty2^{2kl}\|\cP_kf\|^2_{H^m},\quad
	\sum_{j=-1}^\infty2^{2jm}\|\mF_{j}f\|^2_{L^2_l}\le C_{m,l}\sum_{j=-1}^\infty2^{2jm}\|\F_jf\|^2_{L^2_l}.
	\een
	
	(ii) If $m, n, l\in\R$ and $\delta>0$, then we have
	\ben\label{7.77}
	\sum_{j=-1}^\infty2^{2 j n}  \|\F_jf\|^2_{H^m_l}  \lesssim C_{m, n, l}\|f\|^2_{H^{m+n}_l},\quad \|f\|_{H^{-\frac{3}{2}-\delta}_l}   \lesssim C_l\|f\|_{L^1_l}.
	\een

     (iii) It holds that $\|\F_j\cP_k f\|_{H^1_{-3/2}}^2\gs  2^{2j}2^{-3k}\|\F_j\cP_k f\|_{L^2}^2-C_N2^{-5k}\|\mF_j\mP_k f\|_{L^2}^2-C_N2^{-2Nj}2^{-2Nk}\|f\|_{H^{-N}_{-N}}^2$ and 
     \ben\label{gFjkhf}|\int_{\R^3} g(\F_j\cP_k h)(\F_j\cP_k f)dv|\ls\sum_{m\leq j+N_0}\LLT{\mF_m\mP_kg}2^{3m/2}\LLT{\mF_j\mP_kh}\LLT{\mF_j\mP_kf}\\\notag+C_N\rr{\HHLL{g}\LLT{\mF_j\mP_kh}\LLT{\mF_j\mP_kf}+2^{-jN}2^{-kN}\HHLL{g}\HHLL{h}\HHLL{f}}. \een
\end{lem} 
\begin{rmk}
	We emphasize that \eqref{Ber} and \eqref{7.70} also hold if we replace norm $H^m$ by $H^{m,\sss},\sss\geq0$ and the proof is analogous.
\end{rmk}
\begin{proof} We only need to provide a proof for $(iii)$ since the others have been proved in \cite{HJ1}. We observe that 
\beno \|\F_j\cP_k f\|_{H^1_{-3/2}}^2\sim \sum_{m\ge-1} 2^{-3m}\|\cP_m\lr{D_v}\F_j\cP_k f\|_{L^2}^2\ge \sum_{m\ge-1} 2^{-3m}\rr{\|\lr{D_v}\F_j\cP_m\cP_k f\|_{L^2}^2-\|[\cP_m,\lr{D_v}\F_j]\cP_k f\|_{L^2}^2}. \eeno
Thanks to Lemma \ref{7.8}(1) and Lemma \ref{PFCom}, we have 
\beno&&\sum_{m\ge-1} 2^{-3m}\|\lr{D_v}\F_j\cP_m\cP_k f\|_{L^2}^2\sim \sum_{|m-k|\le  N_0} 2^{-3m}2^{2j}\|\F_j\cP_m\cP_k f\|_{L^2}^2\\&&\gs 2^{-3k}2^{2j}\|\F_j\cP_k f\|_{L^2}^2-C_N2^{-5k}\|\mF_j\mP_k f\|_{L^2}^2-C_N2^{-2Nj}2^{-2Nk}\|f\|_{H^{-N}_{-N}}^2.
\eeno 

From \eqref{MPHICOMMU2} ,  if $M(D_v):=2^{jN}\lr{D_v}\F_j\in S^{N+1}_{1,0}$, then 
  \ben [\cP_m,2^{jN}\lr{D_v}\F_j]=\sum_{1\leq|\al|\leq 2N}C_{\al}2^{-m|\alpha|}  M^\al(D_v)\cP_{m,\alpha}+ r_{2N+1}(v,D_v). \een
Using \eqref{rN} and Lemma \ref{le1.1}, we further have
\beno \|r_{2N+1}(v,D_v)\cP_k f\|_{L^2}\ls \|\lr{D_v}^{-N} \lr{\cdot}^{-N} \cP_k f\|_{L^2}\ls 2^{-kN/2}\|f\|_{H^{-N}_{-N/2}}.\eeno 
From this together with Lemma \ref{PFCom}, we deduce that
\beno \sum_m 2^{-3m}\|[\lr{D_v}\F_j,\cP_m]\cP_k f\|_{L^2}^2\ls C_N2^{-5k}\|\mF_j\mP_k f\|_{L^2}^2+2^{-kN}2^{-jN}\|f\|_{H^{-N}_{-N}}^2.\eeno
Then desired result follows.

To prove \eqref{gFjkhf}, we notice that
\beno \int_{\R^3} g(\F_j\cP_k h)(\F_j\cP_k f)dv=\sum_{m\le j+N_0}\sum_{l,\ell\ge-1} \int_{\R^3} \cP_\ell(\F_m\cP_l g)\tP_\ell(\F_j\cP_k h)\tP_\ell(\F_j\cP_k f)dv\\
=\sum_{m\le j+N_0}\sum_{|l-\ell|\leq N_0}\sum_{|k-\ell|\leq 2N_0}+\sum_{m\le j+N_0}\sum_{|l-\ell|\leq N_0}\sum_{|k-\ell|> 2N_0}+\sum_{m\le j+N_0}\sum_{|l-\ell|> N_0}:=\mathscr{A}_1+\mathscr{A}_2+\mathscr{A}_3. \eeno 
For $\mathscr{A}_1$, by Bernstein inequality, it holds that
\[\ba
\nr{\mathscr{A}_1}&\leq\sum_{m\le j+N_0}\sum_{|l-\ell|\leq N_0}\sum_{|k-\ell|\leq 2N_0}\LLF{\cP_\ell\F_m\cP_l g}\LLT{\tP_\ell\F_j\cP_k h}\LLT{\tP_\ell\F_j\cP_k f}\\
&\ls\sum_{m\leq j+N_0}\sum_{|l-k|\leq 3N_0}\LLF{\F_m\cP_l g}\LLT{\F_j\cP_k h}\LLT{\F_j\cP_k f}\\&\ls\sum_{m\leq j+N_0}\LLT{\mF_m\mP_kg}2^{3m/2}\LLT{\mF_j\mP_kh}\LLT{\mF_j\mP_kf}.
\ea\]
For $\mathscr{A}_2$, by \eqref{7.70}, it holds that
\[\ba
\nr{\mathscr{A}_2}&\ls\sum_{m\le j+N_0}\sum_{|l-\ell|\leq N_0}\sum_{|k-\ell|> 2N_0}\LLT{\cP_\ell\F_m\cP_l g}\|\tP_\ell\F_j\cP_k h\|_{H^2}\LLT{\tP_\ell\F_j\cP_k f}\\
&\ls_N \sum_{m\le j+N_0}\sum_{\ell,l}\LLT{\F_m\cP_l g}2^{-lN}2^{-\ell N}2^{-2jN}2^{-kN}\HHLL{f}\HHLL{h}\\
&\ls_N2^{-jN}2^{-kN}\HHLL{g}\HHLL{h}\HHLL{f}.
\ea\]
Similarly for $\mathscr{A}_3$, it holds that
\[\ba
\nr{\mathscr{A}_3}&\ls\sum_{m\le j+N_0}\sum_{|l-\ell|> N_0}\|\cP_\ell\F_m\cP_l g\|_{H^2}\LLT{\F_j\cP_k h}\LLT{\F_j\cP_k f}\\
&\ls_N\sum_{m\le j+N_0}\sum_{l,\ell}2^{-lN}
2^{-mN}2^{-\ell N}\HHLL{g}\LLT{\mF_j\mP_kh}\LLT{\mF_j\mP_kf}\\
&\ls_N\HHLL{g}\LLT{\mF_j\mP_kh}\LLT{\mF_j\mP_kf}.\ea\]
Then we complete the proof.
\end{proof}

{\bf Conflict of interested.} The authors declare that they have no conflict of interest.

{\bf Data availability statement.} Data sharing not applicable to this article as no datasets were generated or analyzed.

{\bf Acknowledgments.}6 The research of L.-B. He was supported by NSF of China under Grant No.11771236 and New Cornerstone Investigator Program 100001127. Jie Ji was supported by Jiangsu Funding Program for Excellent Postdoctoral Talent and Basic Research Program of Jiangsu under Grant No.BK20251376.

\end{document}